\documentclass[bj,authoryear]{imsart}

\usepackage{amsfonts}
\usepackage{amsmath}
\usepackage{amsthm}
\usepackage{amssymb}
\usepackage{multirow}
\usepackage{multicol}
\usepackage{mathrsfs}
\usepackage{enumitem}
\usepackage{times}
\usepackage{graphicx}
\usepackage{subcaption}
\usepackage{bbm}

\startlocaldefs
\theoremstyle{plain}
\newtheorem{theorem}{Theorem}
\newtheorem{lemma}{Lemma}[section]
\newtheorem{proposition}{Proposition}
\newtheorem{definition}{Definition}[section]
\newtheorem{corollary}{Corollary}
\newtheorem{assumption}{Assumption}

\theoremstyle{remark}
\newtheorem{example}{Example}
\newtheorem{remark}{Remark}

\def\aff{{\sf aff}}

\def\bM{\mathbf{M}}

\def\bX{\mathbf{X}}

\def\bZ{\mathbf{Z}}

\def\bTheta{{\boldsymbol\Theta}}

\def\ba{\mathbf{a}}
\def\bb{\mathbf{b}}
\def\bc{\mathbf{c}}

\def\bv{\mathbf{v}}

\def\bx{\mathbf{x}}

\def\balpha{{\boldsymbol\alpha}}
\def\bbeta{{\boldsymbol\beta}}

\def\btheta{{\boldsymbol\theta}}

\def\bmu{{\boldsymbol\mu}}

\def\bphi{{\boldsymbol\phi}}

\def\bbE{\mathbb{E}}

\def\bbP{\mathbb{P}}

\def\bbR{\mathbb{R}}
\def\bbS{\mathbb{S}}

\def\bbZ{\mathbb{Z}}

\def\cA{\mathcal{A}}

\def\cC{\mathcal{C}}

\def\cE{\mathcal{E}}

\def\cG{\mathcal{G}}
\def\cH{\mathcal{H}}
\def\cI{\mathcal{I}}

\def\cL{\mathcal{L}}
\def\cM{\mathcal{M}}
\def\cN{\mathcal{N}}

\def\cP{\mathcal{P}}
\def\cQ{\mathcal{Q}}

\def\cS{\mathcal{S}}
\def\cT{\mathcal{T}}

\def\cX{\mathcal{X}}

\def\Cov{{\sf Cov}}

\def\conv{{\sf conv}}

\def\diag{{\sf diag}}
\def\diam{{\sf diam}}

\def\intr{{\sf int}}

\def\KL{{\sf KL}}
\def\extr{{\sf extr }}

\def\norm#1{\left\|#1\right\|}

\def\supp{{\sf supp}}

\def\TV{{\sf TV}}

\newcommand{\innerprod}[1]{\left\langle#1\right\rangle}

\def\bTheta{\boldsymbol{\Theta}}

\endlocaldefs

\begin{document}

\begin{frontmatter}

\title{Learning Mixtures of Nonparametric and Convolutional Measures on Effectively Low-dimensional Affine Spaces}
\runtitle{Convolutional measures on effectively low-dimensional spaces}

\begin{aug}
\author[A]{\inits{}\fnms{Sunrit}~\snm{Chakraborty$^\star$}\ead[label=e1]{sunritc@umich.edu}}
\author[B]{\inits{}\fnms{XuanLong}~\snm{Nguyen}\ead[label=e3]{xuanlong@umich.edu}}
\address[A]{Duke University\printead[presep={,\ }]{e1}}
\address[B]{University of Michigan, Ann Arbor\printead[presep={,\ }]{e3}}

\end{aug}

\begin{abstract}
In this paper, we develop a finite mixture of convolutional distributions, a statistical model to analyze continuous data distributed approximately on a mixture of low-dimensional affine subspaces. The observations are assumed independent and identically distributed from the mixture of distributions, where each component arises from a convolution of a distribution supported on a low-dimensional subspace with a suitable noise kernel. We discuss theoretical properties of such class of models, including identifiability under very general conditions - in particular, showing that the minimal representation for such mixtures is uniquely identifiable in a semi-parametric setting. We further study the posterior contraction rates for the parameters for a parametrized class of such models where the supports of the component mixing measures are assumed to be convex polytopes under a suitable well-specified Bayesian regime. This still requires developing novel inverse bounds for problems involving a nested mixture structure, where the mixture kernel is itself another continuous mixture. Our approach for both the identifiability theory and posterior contraction rates is to exploit the geometric structure of the underlying support of the latent measures. Apart from applications in end-member analysis, spectral unmixing and topic models, this study provides a grounded framework for subspace clustering with the goal of exploring conditions for learning multiple latent low-dimensional structures.  We illustrate our findings through careful simulation study, which also includes developing new algorithms for such class of models
\end{abstract}


\end{frontmatter}

\section{Introduction}\label{sec:introduction}
Mixture models are some of the most powerful tools to a statistician for understanding heterogeneity in data, while also retaining some aspects of homogeneity within a cluster. Such models are characterized by discrete latent variables indicating the component association of the observations and each component is characterized by a probability distribution, typically from a parametric family of distributions, to be referred to as the mixture kernel. Finite mixture models have been successfully applied in various fields including astronomy, biology, genetics, economy, social sciences and engineering \cite{mclachlan2000finite}. While the Gaussian kernel is the most popular choice due to its simplicity, other kernels have also been used in the literature including mixtures of T-distribution \cite{burgess2020finite}, skew-normal \cite{lin2007finite} and Gamma \cite{wiper2001mixtures}.  Finite mixture models have been used for both density estimation as well as clustering and understanding sub-population structure. Since the parameters capture the essential information related to sub-populations, understanding parameter estimation is of fundamental importance to maximize their utility. Recently, theoretical understanding of such mixture models, including the study of identifiability and estimation rates has seen rapid development. 

In the current era of statistics, often the data resides in a relatively high-dimensional ambient space, but the sub-population level homogeneity is related to some latent low-dimensional structure, which is the core idea behind the current paper. In this work, we will study finite mixture model with nonparametric components of the form
\begin{equation}
    p(x) = \sum_{k=1}^K \pi_k p_k(x), \quad x\in \bbR^D,
\end{equation}
where the component distributions $p_k$ are non-parametrically specified as $p_k = G_k \star \phi_k$ for a convolutional kernel with density $\phi_k$ on $\bbR^D$, and $G_k$'s are probability measures such that $\dim \supp (G_k) < D$. Such models generalize many existing latent variable models include the Mixture of Probabilistic Principal Component Analysis (MPPCA) \cite{tipping1999mixtures} and Mixture of Factor Analyzers (MFA) \cite{mclachlan2003modelling}, and they both can be seen as particular cases for the general formulation above. 
The first question that we seek to answer in this work is whether it is possible to extend such models going beyond the Gaussian assumptions used in MPPCA and MFA while retaining interpretability. 
The extension of finite mixture models to general nonparametric settings has been challenging due to extreme non-identifiability issues, for which existing results require strong assumptions \cite{bordes2006semiparametric, hunter2007inference, jochmans2017inference,aragam2020identifiability}. In the first part of the paper, we demonstrate that identifiability of the nonparametric components can be recovered even if the component supports overlap, when restricted to such low-dimensional structure. 

The second part of the paper deals with a specific natural parametrization of the component distributions, drawing inspiration from archetypal analysis \cite{cutler1994archetypal} for clustering and topic models \cite{blei2003latent}. In such applications, each observation can be seen as a noisy version of the convex combination of certain end-members (called archetypes or topics depending on context), i.e., an observation $x\in\bbR^D$ is expressed as $x \approx \sum_j \beta_j\theta_j$ where $\theta_1,\dots,\theta_d$ represent the end-members. $\beta\in\Delta^{d-1}$ can be seen as a latent variable in this case and endowed with some probability distribution on the probability simplex. Note that such modeling imply that while $x\in\bbR^D$, it is roughly lying on a convex polytope of dimension no greater than $d-1$. We consider a mixture setting, where each component has such a distribution. While such ideas exist in the machine learning literature (such as sparse non-negative matrix factorization or sparse subspace clustering), the current works deals with a completely probabilistic latent variable model -- we explore connections of our model to several others in the sequel. The goal of this part of the paper is to develop tools to study the estimation rate for such mixture models. To the best of our knowledge, this is the first attempt to understand parameter estimation in complex latent variable models, going beyond finite mixtures. The combination of discrete and continuous latent variables make this problem extremely challenging, since existing methods cannot be used directly. 

We briefly discuss the technical tools used in our analysis. The vast literature on identifiability in mixture models and associated latent variable models (finite mixtures \cite{teicher1961identifiability, teicher1967identifiability}, product mixture \cite{vandermeulen2015identifiability}, topic models \cite{anandkumar2012spectral}, latent attribute model \cite{gu2023identifiability}, non-parametric mixtures \cite{aragam2020identifiability}) illustrates that for mixtures, identifiability is related to linear independence of the component kernel. In our approach, we exploit the geometric structure of probability measures supported on distinct affine spaces to establish identifiability of the components. In the presence of convolutional noise, we use characteristic functions to extract the noise parts and then isolate the components based on a geometric approach. While identifiability in mixture models require linear independence of the component probability kernel, understanding estimation rates need a finer characterization. Optimal contraction rates for the latent mixing measure for the univariate data was first presented in \cite{chen1995optimal} and later followed up by for multivariate data and wider classes of kernel functions in various works including \cite{rousseau2011asymptotic}, \cite{ishwaran2001bayesian} and in more recent papers like \cite{ho2016strong}. To overcome the \textit{label-switching} non-identifiability issue associated with mixture models, it is often helpful to view a mixture model through the lens of a latent mixing measure $p_G(x) = \int f(x|\theta)dG(\theta)$, where, for GMM, $f$ is the Gaussian location-scale kernel and $G=\sum_k \pi_k \delta_{\mu_k,\Sigma_k}$ is the associated mixing measure (note that this is a discrete probability measure), which captures all the information about the latent structure. The line of work \cite{nguyen2013convergence}, \cite{nguyen2015posterior}, \cite{ho2016strong}, \cite{wei2020convergence}, \cite{guha2021posterior} utilize convergence in Wasserstein metric on the space of this latent mixing measure to study parameter estimation, often using an \textit{inverse bound} (lower bounding the metric on the parameter space in terms of total variation metric on the distribution space) as the core tool to translate density estimation results to parameter estimation. In our case, the latent mixing measure turns out to be more complex, since it is no longer discrete. We use a similar approach and the crux of the analysis is establishing a suitable inverse bound. For this we use geometric ideas associated with convex polytopes to disentangle the components and isolate the vertices to obtain such results. Our inverse bound provides stronger control compared to that in the general setting in \cite{nguyen2013convergence}, and leads to a parametric estimation rate for the parameters. Although we present estimation rates in the Bayesian setting, it is worth noting that the inverse bound allows such results to be useful with an MLE-based approach as well.

The main contributions of this paper can be summarized as follows. Firstly, we study a class of probabilistic semi-parametric mixture models for analyzing data lying approximately on a mixture of (possibly intersecting) low-dimensional affine spaces -- our results establish identifiability for such general mixtures and also discuss uniqueness of the minimal representation of such models. Secondly, we utilize an intuitive and natural parametrization for such class of models to study posterior contraction rates for the parameters, under a well-specified setting from a Bayesian perspective - this is accomplished by developing a novel inverse bound for such class of models. The posterior analysis considered in this paper is not under a `high-dimensional' setting in that we focus on the fixed ambient dimension regime. We discuss connections of our class of models to several others in the literature. Finally, we develop some algorithms for inference in such models and illustrate the theoretical findings through simulations. 

The remainder of the paper is structured as follows: Section \ref{sec:model}, we introduce the model and discuss various ways to view the model and present a unifying framework to compare it with some of the existing models that can be viewed under the same lens. In Section \ref{sec:identifiability}, we discuss non-identifiability issues in such class of models and prove that under fairly mild conditions, identifiability can be recovered. In Section \ref{sec:posterior_contraction}, we simplify the general model under a specific parametrization and provide a detailed analysis of the posterior contraction rates for the parameters under a suitable metric. We also discuss the consistency in identifying the number of components under the Bayesian paradigm. In Section \ref{sec:numerical study}, we first describe the MCMC algorithm for this parametrized family of models and then, present numerical study to demonstrate the effectiveness of this model using simulations and real data. 

\subsection*{Notation}\label{sec:notation}
We denote the $d-$dimensional Euclidean space by $\bbR^d$ and $\norm{\cdot}, \norm{\cdot}_1$ respectively denotes the Euclidean norm and the $1-$norm. Under a metric $d$ on a space $\cX$, for $x\in\cX, A\subset \cX$, we let $d(x,A):=\inf_{y\in A} d(x,y)$ to denote the distance of the point $x$ from the set $A$. $B_d(x,r)$ denotes the open ball $\{y\in \cX: d(y,x)<r\}$. $\Delta^{K-1}:=\{p\in\bbR^K: \sum_i p_i = 1, p_i\geq 0 \, \forall i\}$ denotes the $(K-1)-$dimensional probability simplex. For $\theta_1,\dots,\theta_K\in\bbR^d$, $\conv(\theta_1,\dots,\theta_K)$ denotes the convex hull of these points; similarly $\aff(\theta_1,\dots,\theta_K)$ denotes their affine hull. A general simplex is the convex hull of $K$ affinely independent points in $\bbR^d$ and is parametrized in terms of its extreme points (vertices). The dimension of a simplex is the dimension of its affine hull, which in turn is the dimension of the unique linear subspace which is its translation. A polytope is a more general object, formed as the convex hull of any $K$ points (not necessarily affinely independent). We call $v$ a vertex of a polytope if $v$ cannot be written as a non-degenerate convex combination of any two other points in that polytope. The affine hull and affine dimension of a polytope are defined similarly as for simplexes. We use $[K]$ to denote the set $\{1,\dots,K\}$ for $K\in\bbZ_+$. For an array of objects $\Theta=(\theta_1,\dots,\theta_K)$ and a permutation $\sigma:[K]\to[K]$, we denote $\sigma(\Theta)=(\theta_{\sigma(1)},\dots,\theta_{\sigma(K)})$. 

We use the notations $d_{\TV}(p,q), h(p,q)$ and $K(p,q)$ to denote the total variation, Hellinger metric and the Kullback-Leibler divergence respectively between probability densities $p,q$. We generally use lower case to denote the density of a particular measure, e.g. if $P$ is a measure, then $p$ is its density (with respect to Lebesgue on the ambient space, unless specified). $\cL_n$ denotes the Lebesgue measure on $\bbR^n$ and $\cH_n$ denotes the $n-$dimensional Hausdorff measure. Use $\mu\ll\nu$ to indicate that $\mu$ is absolutely continuous with respect to $\nu$. We use $\cN_d(\mu,\Sigma)$ and $ \text{Dir}_K(\alpha)$ to denote the $d-$dimensional Gaussian distribution and the Dirichlet distribution on $\Delta^{K-1}$ respectively and $\text{Cat}, \text{Mult}$ to denote the Categorical and Multinomial distributions. $\phi_d(x;\mu,\Sigma)$ is used as the density of $\cN_d(\mu,\Sigma)$ at $x$. For two probability measures $G$ and $Q$, denote their convolution as $G\star Q$, which is the law of $X+Y$ where $X\sim G\perp Y\sim Q$. For a probability measure $P$ on $\bbR^D$, we recall that its support $\supp(P)$ is defined as the smallest closed set $\cS$ such that $P(\cS)=1$. Write $\mu+Q$ to denote the translation of the measure $Q$ by $\mu$, i.e., the distribution of $Y=\mu+X$, where $X\sim Q$. Lastly, we denote the push-forward of $\mu$ under transformation $f$ as $f_{\#}\mu$, defined as $f_{\#}\mu= \mu\circ f^{-1}$.

\section{Model}\label{sec:model}
We study a class of mixture models, where each component is a convolutional distribution around a low-dimensional probability distribution. We start by formalizing these notions. 

\begin{definition}\label{def:low-dim}
    For a probability measure $G$ on $\bbR^D$, we say $G$ is \textbf{continuously supported} on $\cS$ if $\supp(G)=\cS$ and $G\ll \cH_d$ restricted to $\cA=\aff \cS$, where $d=\dim \cA$. 
\end{definition}

For $\cS\subset \bbR^D$, we say $\cS$ is \textbf{low-dimensional} if $\dim \cA < D$ and we identify $\dim(\cA)$ as the \textit{dimension of } $\cS$. We call a probability measure $G$ on $\bbR^D$ low-dimensional if $G$ is continuously supported on $\cS$ with $\dim \cA<D$. Note that the above definition implies that if $G$ is continuously supported on $\cS$ with $d=\dim \cA$, then the Hausdorff dimension of $\cS$ is precisely $d$. In such a case, we also say $G$ is $d-$dimensional. The condition $G\ll \cH_d$ ensures that the measure is non-degenerate within its affine hull, in the sense that any subset of strictly lower Hausdorff dimension is a null set under $G$. This property plays a crucial role in the identifiability arguments in Section \ref{sec:identifiability}. 

\begin{example}\label{example:1}
    Here are a few examples of low-dimensional probability measures.
    \begin{enumerate}[label=(\roman*)]
    \item For any $B\in\bbR^{D\times d}$ with full column rank, the probability distribution $G=\cN_D(\mu, BB^\top)$ has support $\cS=\mu + \text{col}(B)$, which is itself an affine space i.e. $\cA=\cS$ and hence, $G$ is $d-$dimensional. 
    \item For any set of $d$ points $\theta_1,\dots,\theta_d\in\bbR^D$ and any probability distribution $\mu$ on $\Delta^{d-1}$ such that $\mu\ll\cH_{d-1}$ on $\Delta^{d-1}$, take $G=f_{\#}\mu$ where $f(\beta)=\sum_j \beta_j \theta_j$, $f:\Delta^{d-1}\to \bbR^D$. This $G$ is supported on a convex polytope $\cS=\conv(\theta_1,\dots,\theta_d)$ and $G\ll \cH_{d'}$ on $\cA=\aff \cS$ with $d'\leq d-1$ being the dimension of $\cS$. Thus $G$ is $d'-$dimensional. 
    \item Consider a sequence $\{\epsilon_n\}$ such that $\sum_n \epsilon_n<1$ (e.g. $\epsilon_n=2^{-(n+1)}$). Start with $C_0=[0,1]$; at step $n$, from each interval in $C_{n-1}$ remove an open middle interval whose length is an $\epsilon_n-$fraction of that interval; let $C=\cap_n C_n$. Note that $C$ is closed, nowhere dense with Lebesgue measure $|C|=\prod_n (1-\epsilon_n)>0$ and empty interior. Now define $\cS=C\times [0,1]^{d-1}\times \{0\}^{D-d}\subset\cA=\bbR^d\times \{0\}^{D-d}\subset \bbR^D$ and $G$ a uniform distribution on $\cS$, then $G\ll \cH_d$ on $\cA$ with density $g(x)=|C|^{-1}1_{\cS}(x)$ with respect to $\cH_d$ on $\cA$. So $G$ is $d-$dimensional in the sense defined above, however, its support $\cS$ contains no open subset of $\cA$.
\end{enumerate}
\end{example}

Next, let us define convolutional distributions which will be used to capture the essence of decomposing latent structure and noise in our model. 

\begin{definition}\label{def:conv}
    We call a probability measure $P$ on $\bbR^D$ \textbf{convolutional} around $G$ if $P=G\star Q$ for some probability measure $Q$ on $\bbR^D$ with mean 0 and density $q$ with respect to $\cL_D$.
\end{definition}

If $P=G\star Q$ as in the above definition, the density of $P$ with respect to $\cL_D$ is given by $p(x)=\int q(x-\eta)dG(\eta)$, which exists even if $G$ is not absolutely continuous with respect to $\cL_D$. If $P=G\star Q$ as in the definition above, we call $Q$ the \textit{noise} kernel, to distinguish it from the mixture kernel to appear below. We will consider a parametric family of zero-mean probability distributions $\{Q_\phi\}$ for the noise kernel. Examples include the multivariate normal, the multivariate Cauchy or the multivariate Laplace distribution. 

Now we are ready to describe the class of mixture models. Let $\cP_{D,d}$ be the collection of all $d$-dimensional probability measures in $\bbR^D$, and $\cP_D= \cup_{d<D} \cP_{D,d}$ be the collection of all low-dimensional probability measures in $\bbR^D$. Given a noise kernel family $\cQ=\{Q_{\phi}:\phi\in \Phi\}$, let $\cP_D(\cQ)=\{G\star Q_{\phi} : G\in \cP_D, \phi\in \Phi\}$. Finally, we consider (finite) mixtures over $\cP_D(\cQ)$, i.e., probability distributions of the form
$$P = \sum_{k=1}^K \pi_k \mu_k, \quad \mu_k\in \cP_D(\cQ), \, \pi\in\Delta^{K-1}.$$
We note that each component measure $\mu_k$ in such a mixture is actually absolutely continuous with respect to $\cL_D$, but has the representation $\mu_k=G_k\star Q_{\phi_k}$, where $G_k$ captures the latent low-dimensional structure. 

This mixture can be best understood from the following hierarchical model. Assume that the observations arise from $K$ latent sub-populations -- for the $k-$th subpopulation, the main source of variability within this subpopulation is captured by a latent low-dimensional distribution, $G_k\in \cP_D$. To model the observations $X_1,\dots,X_n\in \bbR^D$, we assume that for each observation $X_i$, firstly, a component-indicator variable $z_i$ taking values in $[K]$ is drawn from a categorical distribution with parameter $\pi\in\Delta^{K-1}$ denoting the mixture probabilities. Given $z_i=k$, a further latent variable $\eta$ is drawn from the sub-population specific low-dimensional measure $G_k$ and finally, conditional on $\eta$, the observation $X_i$ is drawn from $\eta + Q_{\phi_k}$. The hierarchical structure of this model is shown below.

\begin{equation}
\begin{aligned}\label{model}
    z_i | \pi &\overset{\text{iid}}{\sim} \text{Cat}(\pi),\\
    \eta_i | \cG, z_i=k &\overset{\text{ind}}{\sim} G_k, \\
    \epsilon_i |\bphi, z_i=k &\overset{\text{ind}}{\sim} Q_{\phi_k}, \\
    X_i &= \eta_i + \epsilon_i,
    \end{aligned}
\end{equation}

where $\cG=(G_1,\dots,G_K)$ is the collection of the low-dimensional measures with $G_k\in \cP_D$ and $\bphi=(\phi_1,\dots,\phi_K)$ are the parameters for the noise kernel. We shall refer to $G_1,\dots,G_K$ as the \textit{component latent measures} since the distribution of the $k$th component can be expressed as $p_k(x|\cG,\bphi) = \int q_{\phi_k}(x|\eta) dG_k(\eta)$, which forms the mixture kernel in the overall mixture. $z_i$ and $\eta_i$ are considered latent variables in this model, since we only observe $X_1,\dots,X_n$. While $z_i$ is the latent sub-population indicator, $\eta_i$ captures the main variability within the cluster, while $\epsilon_i$ captures the noise. Marginalizing over these variables, the probability measure defined by the above model becomes
\begin{equation}\label{eq:model measure}
    P_{\cG,\pi,\bphi} = \sum_{k=1}^K \pi_k (G_k \star Q_{\phi_k}).
\end{equation}
Denoting $q_\phi$ to be the density of $Q_\phi$ with respect to $\cL_D$, the density of a particular observation $X$ under this model with respect to $\cL_D$ is given by
\begin{align}\label{eq:density}
    p_{\cG,\pi,\bphi}(x) = \sum_{k=1}^K \pi_k \int q_{\phi_k}(x-\eta) dG_k(\eta).
\end{align}
The joint density of all the observations is then simply the product since they are assumed independent and identically distributed $p(X_1,\dots,X_n|\pi,\cG,\bphi) = \prod_{i=1}^n p_{\cG,\pi,\bphi}(X_i)$. Other than specific choices for $\{Q_{\phi}\}$ and $\cG$, the likelihood involves an intractable integral (for the marginalization over the continuous latent variables $\eta_i$) and a sum (for the marginalization over the discrete latent variables $z_i$) -- this nested mixture structure is the primary source of both modeling flexibility and analytical difficulty.

Note that the component distributions are semi-parametric with a non-parametric low-dimensional part convolved with a full-dimensional noise kernel. While each component density is fully supported on $\bbR^D$,  the underlying latent geometric structure is encoded in the supports of $G_k$. This interplay between intrinsic low-dimensional geometry and full-dimensional convolution is the key feature that enables identifiability, in contrast to general nonparametric mixtures where such recovery is impossible.
Our approach in Section \ref{sec:identifiability} involves imposing a mild geometric assumption on the supports of $G_k$, ensuring only that they are on distinct affine spaces. This allows us to disentangle the components, while allowing the component mixing measures to be general, with potentially intersecting supports. 

\begin{figure}
    \centering
    \includegraphics[clip, trim=0.5cm 1.3cm 0.5cm 3.1cm, width=0.7\linewidth]{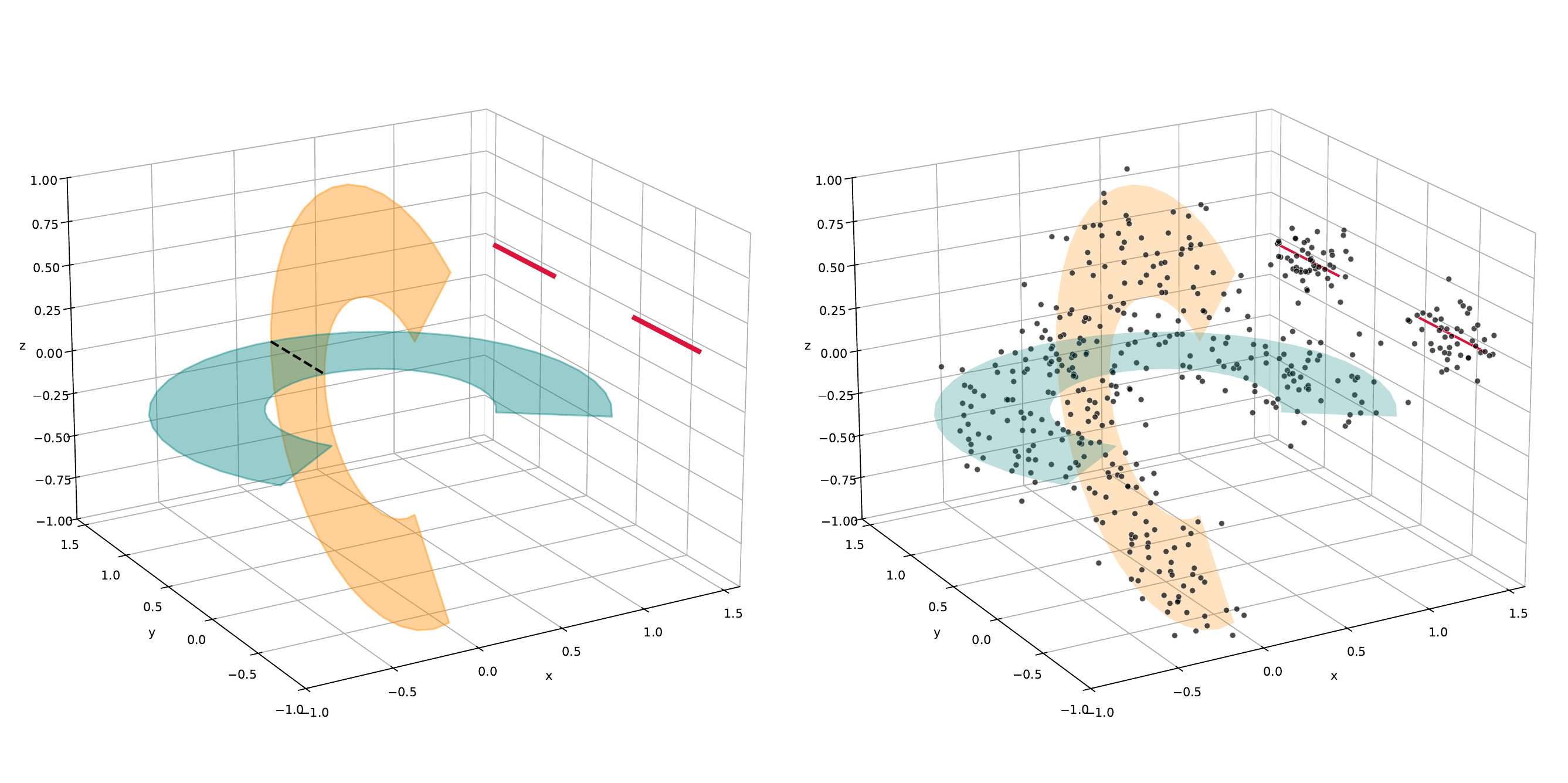}
    \caption{Example of model in $D=3$ (for visualization) with $K=3$ components: two $2-$dimensional components $G_1, G_2$ with supports $\cS_1, \cS_2$ shown in orange and teal colors (intersecting along the $1-$dimensional line segment shown with dashed line), and one $1-$dimensional component $G_3$, whose support is the union of two line segments on the common affine space (a line), shown in red. The measure $G_1$ (orange) is supported on a plane parallel to the $Y-Z$ plane, while $G_2$ (teal) is supported on a plane parallel to the $X-Y$ plane. The left panel shows the latent intrinsic low-dimensional structure, the right panel shows a scatter plot of generated observations from the overall mixture model with noise. In this example, $G_1$ is uniform on its U-shaped support, but the model allows arbitrary non-parametric $G_1\ll \cH_2$ restricted to the affine space containing the support.}
    \label{fig: non convex example model}
\end{figure}

\begin{remark}\label{remark:model}
    We make the following remarks regarding the above model:
    \begin{enumerate}[label=(\roman*)]
        \item This model can be seen as a \textit{semi-parametric mixture} model, where the mixture probabilities and component specific scale parameters $\phi$ form the parametric part, while the component measures $\cG$ form the non-parametric part. In general, such a mixture model is not identifiable -- the special structure which allows us to overcome this issue is the fact that each $G\in\cG$ is low-dimensional.
        \item We refer to $P_{\cG, \pi}=\sum_{k=1}^K \pi_k G_k$ to be the \textit{noiseless} model, since this is precisely the distribution of the latent $\eta_i$ in Equation \eqref{model}. The observations in our model takes a further convolution with respect to the noise model $\{Q_\phi\}$. The probability measure $P_{\cG,\pi}$ is singular with respect to $\cL_D$. Note that if $\phi_1=\dots=\phi_K$ and the common convolutional distribution is $Q$, the density can be expressed as $p(x)=\int q(x-\eta) dP_{\\cG,\pi}(\eta)$, and hence $P_{\cG,\pi}$ can be thought of as the latent mixing measure in this case.
        \item The overall model is \textit{not} a convolution -- each component of the mixture is a convolution of a low-dimensional measure $G_k$ with a noise-model $Q_k$. This enables a more expressive class of models, while rendering its analysis more challenging.
    \end{enumerate}
\end{remark}

Notice that the standard finite mixture model of the form $\sum_k \pi_k q_{\phi_k}(x-\mu_k)$ can also be seen as a special case of the above model with $G_k=\delta_{\mu_k}$ as the 0-dimensional latent component measure. A general example is illustrated in Figure \ref{fig: non convex example model}. We conclude this section with an important parametric special case of the above class of mixtures, which we will study further in Section \ref{sec:posterior_contraction}.
\begin{example}\label{example:polytope}
    As discussed in Example \ref{example:1}(2), for the $k-$th component, we parameterize $G_k$ by $\Theta_k=(\theta_{k,1},\dots,\theta_{k,d})$, the vertices of a convex polytope $\cS_k$ and take $G_k=(f_k)_{\#} (\mu_k)$, where $f_k:\beta\to \sum_{j=1}^d \beta_j\theta_{k,j}$, and $\mu_k$ is a probability distribution on $\Delta^{d-1}$, for example, $\mu_k=\text{Dir}(\gamma_k)$. Note that this construction ensures that $G_k$ is $(d-1)-$dimensional and is supported on a convex polytope. This defines the kernel in our general model \ref{model}, using $Q_{\phi}=\cN(\cdot|\boldsymbol{0}, \phi^2 I_D)$ as the noise model. The intuition behind such a latent measure $G_k$ stems from the fact that the signal $\eta\sim G_k$ can be seen as a (random) convex combination of the \textit{end-members} $\theta_{k,1},\dots,\theta_{k,d}$ representing some extremal behavior within the subpopulation - this relates such a latent measure to convex unmixing or archetypal analysis models.
\end{example}

\subsection{Related Models}\label{sec:related models}
There have been numerous effort at extending the mixture modeling framework, particularly to high-dimensions with nonparametric components. The proposed model is related to many such models in the existing literature and can be viewed as a semi-parametric extension to them.

\subsubsection*{Non-parametric Mixtures}
Extending finite parametric mixture model to the non-parametric case is challenging due to severe non-identifiability issues as discussed in Section \ref{sec:introduction}. Recent work focuses on extensions when observations are grouped according to the latent class assignments \cite{vandermeulen2019operator, ritchie2020consistent}, when the covariates carry certain latent structures \cite{allman2009identifiability, gassiat2016nonparametric}, when the component distributions are symmetric \cite{bordes2006semiparametric, hunter2007inference}, products of
univariate distributions \cite{hall2003nonparametric, elmore2005application} or well-separated \cite{aragam2020identifiability, aragam2023uniform}. The last two works are closest to ours and consider a general non-parametric setting of the form $p(\cdot) = \sum_k \pi_k \left(G_k(\cdot-\mu_k) \star \phi_{\sigma} \right)$, where $G_k$ are compactly supported. Note that in our model \eqref{eq:density}, each component is allowed to have different levels of noise variance, $q_{\phi}$ is allowed to be non-Gaussian and the supports are not required to be compact. \cite{aragam2023uniform} assume a strong separation between the supports of the underlying $G_k'$s (assumption D3), which is required if there are no other restriction on $G_k'$s. In this work, we explore a different avenue by restricting $G_k$ to be low-dimensional, which allows us to identify the components under much milder assumptions such as Assumption \ref{assume:A} or \ref{assume:B}. 
We note that convolutional mixtures have been studied in \cite{nguyen2013convergence} -- however, the model \eqref{eq:density} can be seen as a mixture of (continuous) mixtures, and the methods in \cite{nguyen2013convergence} cannot be directly used to establish identifiability of each component in such nested mixtures.

\subsubsection*{Mixture of Probabilistic PCA and Mixture of Factor Analyzers}
MPCCA \cite{tipping1999mixtures} and MFA \cite{mclachlan2003modelling} can be seen as extensions of probabilistic models with continuous latent variables Probabilistic PCA \cite{tipping1999probabilistic} and Factor Analysis \cite{kim1978introduction} respectively, where the idea is to use the latent variables to capture a low-dimensional subspace on which the data primarily reside, thereby combining clustering and dimensionality-reduction. Assuming a $d$ dimensional latent space and marginalizing the latent variables, MPPCA leads to the following density
$$p_{\text{MPPCA}}(x) = \sum_k \pi_k \phi(x\mid \mu_k, \Theta_k\Theta_k^\top + \sigma^2 I),$$
where $\bTheta=(\Theta_1,\dots,\Theta_K)$, $\bmu=(\mu_1,\dots,\mu_K)$. $\mu_k\in\bbR^D$ is the mean associated with the $k$-th component, $\Theta_k$ is a $D\times d$ factor matrix  and $\pi$ is the mixture probabilities.
This is a parametric special case of our model \eqref{eq:density} -- $\cN_D(\cdot|\mu, \Theta\Theta^\top + \sigma^2 I) = N_D(\cdot|\mu, \Theta\Theta^\top)\star N_D(\cdot|0, \sigma^2 I)$ (the first measure is $d-$dimensional, supported on the affine space $\mu+\text{col}(\Theta)$ discussed in Example \ref{example:1}(1)) -- where the only noticeable difference is the distribution on the continuous latent variable. 
Identifiability in such models can be directly seen from corresponding results for Gaussian mixture models, up to orthogonal transformations of $\Theta_k$'s. These models has been extended in various directions \cite{collas2021probabilistic, han2015covariance, archambeau2008mixtures, xu2023hemppcat, lee2021mixtures}, however they still mostly use Gaussian (or Student's $t$) distribution for the latent variable $\beta$, which result in an elliptic distribution for the components. In contrast, the parametrization in this work allows a relaxation from elliptic distributions and can accommodate skewness, component-specific variance and most importantly, quite general supports for the latent measure, at the cost of the intractable likelihood.

\subsubsection*{Other connections}
The current model is also related to mixed membership mixtures, non-negative matrix factorization (NMF) methods and subspace clustering. We present a brief discussion here - see Appendix \ref{app:related models} for further details. Topic models often utilize the geometric structure where documents are seen as points on a low-dimensional convex polytope (often called topic polytope) whose vertices are the topics -- such a viewpoint has been explored for both theoretical analysis as well as computational methods \cite{nguyen2015posterior, yurochkin2016geometric, yurochkin2019dirichlet, chen2023learning}. Such low-dimensional or convex structure forms the basis for NMF \cite{ding2008convex} and archetypal analysis \cite{cutler1994archetypal, javadi2020nonnegative}. Such distributions can all be seen as special cases of the generic low-dimensional distribution with convolutional noise, considered in this work. The mixture structure in our work extends such models to hierarchical versions, such as hierarchical topic models \cite{griffiths2003hierarchical, chakraborty2024learning}. The current model can also be seen as a probabilistic model for subspace clustering, providing a general and interpretable extension of methods like $K-$subspaces \cite{tseng2000nearest, zhang2009median}, generalized PCA \cite{vidal2005generalized} and spectral clustering methods like sparse subspace clustering \cite{elhamifar2013sparse}.

Based on our discussion, the model considered in this paper is a rich class of semi-parametric mixtures that unify several statistical and machine learning frameworks. While we discuss some algorithms for inference for the proposed model, the focus of this paper is to understand the class of mixture models in Equation \eqref{eq:density} in terms of identifiability of the latent structure and noise kernel parameters and characterize the rate at which the parameters can be learned using a Bayesian formulation for the special parametric case discussed in Example \ref{example:polytope}.

\section{Identifiability}\label{sec:identifiability}
In this section, we study identifiability in the class of models introduced in Section \ref{sec:model}. In particular, our goal is to identify minimal conditions on the underlying latent component measures so that up to a permutation of labels, each component can be identified and both the latent measure and the noise kernel parameter can be learned. Our approach is based on a geometric separation condition on the collection of latent component measures, and we exploit the intrinsic low-dimensional structure to derive our results.

\subsection{Identifiability in the noiseless model}

We start by focusing on the noiseless model $P_{\cG,\pi}=\sum_{k\in[K]} \pi_k G_k$, introduced in Remark \ref{remark:model}(ii), where $\pi\in\Delta^{K-1}$ are the mixture weights and the component measures $G_1,\dots,G_K\in \cP_D$ are low-dimensional. We introduce a simple geometric condition on the supports of $G_k$ which ensures that the components can be recovered uniquely from the mixture. Denote the space of all such mixtures with exactly $K$ components as 
$$\cM_\circ(K)=\left\{P_{\cG,\pi}=\sum_{k=1}^K\pi_k G_k :\pi\in\Delta^{K-1}, G_k\in \cP_D, \pi_k>0\,\forall k\in[K]\right\}.$$
Any $P\in \cM_\circ(K)$ is parametrized by $(\cG,\pi)$ where $\cG=(G_1,\dots,G_K)$ are the latent low-dimensional component measures and $\pi$ is the mixture weights, with $\pi_k>0$ ensuring that there are no trivial redundant components. The component measures are allowed to have different dimensions. Note that any $P_{\cG,\pi}$ is not absolutely continuous with respect to $\cL_D$ in general. We write $\cup_{k\geq 1} \cM_{\circ}(k)$ to denote the class of all finite mixtures with low-dimensional component measures. We use the following notion of geometric separation for identifiability.

\begin{assumption}\label{assume:A}
We call a collection of sets $\cS_1,\dots,\cS_K\subset \bbR^D$ to satisfy Assumption A if for any $k\neq k'\in[K]$, $\aff(\cS_k)$ and $\aff (\cS_{k'})$ are distinct. We call a collection of low-dimensional probability measures $G_1,\dots,G_K\in \cP_D$ to satisfy Assumption A (by slight abuse of terminology) if the collection of the corresponding supports $\cS_1,\dots,\cS_K$ satisfies Assumption A, where $\cS_k=\supp(G_k)$.
\end{assumption}

\begin{remark}
The above assumption on probability measures $G_1,\dots,G_k$ is purely geometric only on the supports of the components. Suppose $G_k$ is $d_k-$dimensional. Then for $k\neq j$, if $d_k\neq d_j$, then $\cS_k$ and $\cS_j$ could be arbitrary (no restriction aside from the dimension), even allowing $\cS_j\subset \cS_k$ -- thus the above assumption places no restriction on pairs of components of different dimensions. For $d_j=d_k$, the assumption only restricts the case where the affine subspaces are the same. Two such supports are still allowed to intersect -- for example, consider two triangular polytopes in $\bbR^3$ intersecting along a line segment.    
\end{remark}

We first demonstrate that this is a sufficient separation condition to identify the components, including the number of components and their dimensions. Denote $\cM_\circ^A(K):=\{P=\sum_k\pi_kG_k\in \cM_\circ(K): G_1,\dots,G_K \text{ satisfy Assumption } \hyperref[assume:A]{A}\}$ to be the collection of all such noiseless mixtures such that the collection of component measures satisfy the above assumption.

\begin{proposition}\label{prop:noiseless_identifiability}
Let $P_{\cG,\pi}, P_{\cG',\pi'}\in \cup_{K\geq 1} \cM_\circ(K)$ be two mixtures of low-dimensional probability measures.
\begin{enumerate}
    \item Suppose $P_{\cG,\pi}\in \cM_{\circ}^A(K)$ and $P_{\cG',\pi'}\in \cM_\circ^A(K')$, such that $P_{\cG,\pi}=P_{\cG',\pi'}$. Then $K=K'$ and there exists a permutation $\tau$ of $[K]$ such that $\pi=\tau(\pi')$ and $\cG=\tau(\cG')$.
    \item Suppose $P_{\cG,\pi}\in \cM_\circ^A(K)$ and $P_{\cG',\pi'}\in \cM_\circ(K')$ such that $P_{\cG,\pi}=P_{\cG',\pi'}$. Then $K'\geq K$. Furthermore, if $K'=K$, then there is a permutation $\tau$ of $[K]$ such that $\pi=\tau(\pi')$ and $\cG=\tau(\cG')$.
\end{enumerate}
\end{proposition}

\begin{remark} A few remarks regarding the above result:
\begin{enumerate}
    \item We note the distinction between the two parts of the above proposition. In part 1, Assumption \hyperref[assume:A]{A} is imposed on \textit{both} the mixtures. This allows identifying the number of components, and shows that the components match up to a permutation of the component labels. This part can be taken as the typical identifiability result. In part 2, note that Assumption \hyperref[assume:A]{A} is only imposed on one of the mixtures. This shows that while such a model can be represented in multiple ways (as a mixture of low-dimensional measures), the minimal representation in terms of the number of components is unique and corresponds to only a permutation of labels. 
    \item Since the collection of $G_1,\dots,G_K$ is identifiable from the mixture $P_{\cG,\pi}$, the individual dimensions $d_1,\dots,d_K$ of the component measures are also identifiable.
    \item The result also extends to the case where $K$ is allowed to be infinite (countable mixtures). 
\end{enumerate}
\end{remark}

The following example shows that it is trivial to `split' components and hence represent the model as a mixture with higher number of components -- this example also illustrates the identifiability challenge in non-parametric mixtures without additional assumption.
\begin{example}
    Let $G_k$ be a probability measure supported on $\cS_k$ with $\dim \cS_k=d$. Suppose there exists measureable sets $S_1,S_2\subset \cS_k$ such that
    $$S_1\cup S_2=\cS_k, \quad G_k(S_1)>0, \quad G_k(S_2)>0, \quad G_k(S_1\cap S_2) =0.$$
    Define probability measures $G_{k1},G_{k2}$ by restricting $G_k$ on $S_1,S_2$ and normalizing, i.e., $G_{k,i}(A)=G_k(A\cap S_i)/G_k(S_i)$ for $i=1,2$. Then
    $$\pi_k G_k = \pi_{k1}G_{k1} + \pi_{k2} G_{k2}$$
    where $\pi_{ki}=\pi_k G_k(S_i)$ for $i=1,2$. Thus, a single mixture component can be easily decomposed into two distinct components yielding a different mixture representation.
\end{example}

The proof of Proposition \ref{prop:noiseless_identifiability} is given in Appendix \ref{app:proof_prop:noiseless_identifiability}. The main ingredient of the proof is realizing that if $\dim \cS_1 > \dim \cS_2$, then since $G_k$ is absolutely continuous with respect to the corresponding Hausdorff measure, $pG_1 + (1-p)G_2|_{\aff(\cS_2)}$ coincides with the measure $G_2$, where $G|_{\cS}(A)=G(\cS\cap A)/G(A)$ is the restricted and normalized version of $G$ on $\cS$. Furthermore, two different affine spaces of the same dimension $d$ are either disjoint or their intersection is another affine space of a strictly lower dimension, and hence a null-set for $\cH_d$. This geometric property along with the absolute continuity of each $G_k$ allows us to extract the components from the mixture one by one, starting with the component having the smallest affine dimension. 

Proposition \ref{prop:noiseless_identifiability} essentially shows that for any $P\in \cup_K \cM_\circ(K)$, there is a \textit{unique minimal} representation as $P_{\cG,\pi}\in \cM_\circ^A(K)$ for some minimal $K$. The Assumption \hyperref[assume:A]{A} can be seen as a unique representation requirement, rather than a restrictive condition. This raises the question of whether alternative conditions yield similar unique representations.
The following is an example of a refinement condition.

\begin{assumption}\label{assume:B}
We call a collection of sets $\cS_1,\dots,\cS_K\subset \bbR^D$ to satisfy Assumption B if for any $k\neq k'\in[K]$, either (i) $\aff\cS_k$ and $\aff \cS_{k'}$ are distinct, or (ii) $\cS_k$ and $\cS_{k'}$ are connected sets and $\cS_k\cap \cS_{k'}=\emptyset$. Call a collection of probability measures $G_1,\dots,G_K\in \cP_D$ to satisfy Assumption B, if the supports $\cS_1,\dots,\cS_K$ satisfy Assumption B, where $\cS_k=\supp(G_k)$.
\end{assumption}

This allows different components to share the same affine space as the corresponding affine hull of its components. It separates such components by assuming that each such support is connected, and the different component supports (sharing the same affine space) are disjoint. Note that requiring $\cS_k$ to be connected is already a restriction -- for example, the measure $G$ in Example \ref{example:1}(3) is supported on $C$ which is not connected, yet $G$ is low-dimensional (but not atomic). Assuming \hyperref[assume:B]{B} in place of \hyperref[assume:A]{A} merely changes the interpretation of what we mean by a component - in the latter case, each component is associated with its own affine space, while for the former, components are potentially finer -- either they are characterized by their own affine space, or by the connected components of the measure within a single affine space. Proposition \ref{prop:noiseless_identifiability b} in the Appendix shows that both parts of Proposition \ref{prop:noiseless_identifiability} are true if we replace Assumption \hyperref[assume:A]{A} with Assumption \hyperref[assume:B]{B}. An example is illustrated in Figure \ref{fig: non convex example model} -- consider the probability measure in the left panel. When restricted to $\cup_k\cM_\circ^A(k)$, the unique minimal representation corresponds to $K=3$, while when restricted to $\cup_k \cM_\circ^B(k)$ (replace Assumption \hyperref[assume:A]{A} by \hyperref[assume:B]{B}), then the unique minimal representation corresponds to $K=4$.

\subsection{Identifiability with noise}

We turn to the identifiability problem for the model with noise. Our approach is to use the previous proposition along with the following lemma 
which shows that a mixture of a low-dimensional probability measure with one which is absolutely continuous with respect to $\cL_D$ is identifiable.

\begin{lemma}\label{lemma:identify_lower_dim_abs_cont}
        Suppose $G, G'\in \cup_k\cM_\circ(k)$ be mixtures of low-dimensional probability measures, and $F, F'$ are absolutely continuous probability measures on $\bbR^D$. Then,
        $$p G + (1-p)F = p' G' + (1-p')F' \Rightarrow p=p', G=G',F=F'.$$
    \end{lemma}

The above lemma can be seen as a direct corollary of the uniqueness of Lebesgue decomposition and allows us to split such a probability measure uniquely into a part that is absolutely continuous and one that is continuously supported on a union of low-dimensional sets. Recall the probability model $P_{\cG,\pi,\bphi}$ from Equation \eqref{eq:model measure}. Using the notations from Section \ref{sec:model}, we denote
$$\cM_{\cQ}(K):=\left\{P=\sum_{k=1}^K \pi_k \mu_k: \pi\in\Delta^{K-1}, \pi_k>0, \mu_k\in \cP_D(\cQ)\right\}$$
to be the class of mixture models with exactly $K$ components such that each component is a convolutional distribution around a low-dimensional measure. For $P=\sum_k\pi_k\mu_k\in \cM_\cQ(K)$, we parametrize it as $P_{\cG,\pi,\bphi}$ with $\mu_k=G_k\star Q_{\phi_k}$. As in the previous section, we denote $\cM^A_{\cQ}(K)$ to be the subset of mixtures $P_{\cG,\pi,\bphi}\in \cM_{\cQ}(K)$ where the underlying $G_1,\dots,G_K$ (components in $\cG$) satisfy Assumption \hyperref[assume:A]{A}. Finally, denote the characteristic function of the probability measure $Q$ by $\varphi_Q$ defined as $\varphi_Q(s)=\bbE_{X\sim Q}[\exp(is^\top X)]$ for all $s\in\bbR^D$. We assume the following condition on the class of noise kernels $\cQ$.

\begin{assumption}\label{assume:C}
    A class of zero-mean probability measures $\cQ=\{Q_{\phi}:\phi\in \Phi\}$ such that $Q\ll \cL_D$ for all $Q\in \cQ$ is said to satisfy Assumption C if the following conditions are met.
    \begin{enumerate}
        \item For any $\phi\in\Phi$, $\varphi_{Q_\phi}(t)\neq 0$ for all $t$.
        \item $\Phi$ is totally ordered.
        \item For any $\phi>\phi'$, $\varphi_{Q_\phi}/\varphi_{Q_{\phi'}}$ is a characteristic function of some probability distribution $\tilde{P}\ll \cL_D$.
    \end{enumerate}
\end{assumption}

The above assumption is satisfied by commonly used classes of noise distributions. Standard examples include the Gaussian and Cauchy.
\begin{align*}
    X\sim \cN_D(\boldsymbol{0}, \phi I), \phi>0&\Rightarrow \varphi_{Q_\phi}(t) = \exp\left(-\frac{1}{2}\phi \norm{t}_2^2\right) \\
    X \sim \text{t}_1(\boldsymbol{0}, \phi I), \phi>0 
&\Rightarrow \varphi_{Q\phi}(t) = \exp\left(-\phi \norm{t}_2\right).
\end{align*}
More generally, the symmetric $\alpha-$stable family $(\alpha\in (0,2]$) satisfies $\varphi_{Q_\phi}(t) = \exp\big(-\phi \norm{t}_2^\alpha\big),$
which interpolates between the Gaussian and Cauchy noise providing a flexible class of heavy-tailed convolution kernels. Beyond stable laws, Assumption \hyperref[assume:C]{C} also holds for many infinitely divisible distributions. For instance, a Gaussian scale mixture with Gamma mixing satisfies $\varphi_{Q_\phi}(t) = (1 + \tfrac{1}{\lambda}\norm{t}_2^2)^{-\phi}$
so that ratio of characteristic functions remains within the same functional family. More generally, any family arising from a Levy process with zero-mean and absolutely continuous marginals satisfies the assumption, with $\phi$ playing the role of time. Most of the noise kernels considered in this work are isotropic in nature, where $\phi$ captures the noise level. We do not consider this a major limitation in the flexibility of the model class since we view the nonparametric $G_k$ as the major structure that explains most of the variability in the data from the $k-$th subpopulation.

Given the above setting and assumptions, we have the following identifiability result for this class of mixture models. 

\begin{theorem}\label{thm:isotropic_identifiability}
Let $P_{\cG,\pi,\bphi}, P_{\cG',\pi',\bphi'}\in \cup_{K\geq 1} \cM_\cQ(K)$ be two finite mixture models in $\bbR^D$ consisting of components that are convolutions around low-dimensional probability measures. Suppose the noise kernel family $\cQ$ satisfies Assumption \hyperref[assume:C]{C}. Then we have the following.
\begin{enumerate}
    \item Suppose $P_{\cG,\pi,\bphi}\in \cM_{\cQ}^A(K)$ and $P_{\cG',\pi',\bphi'}\in \cM_{\cQ}^A(K')$ such that $P_{\cG,\pi,\bphi}=P_{\cG',\pi',\bphi'}$. Then $K=K'$ and there exists a permutation $\tau$ of $[K]$ such that $\pi=\tau(\pi'), \cG=\tau(\cG')$ and $\bphi=\tau(\bphi')$.
    \item Suppose $P_{\cG,\pi,\bphi}\in \cM_{\cQ}^A(K)$ and $P_{\cG',\pi',\bphi'}\in \cM_{\cQ}(K')$ such that $P_{\cG,\pi,\bphi}=P_{\cG',\pi',\bphi'}$. Then $K'\geq K$. If $K'=K$, then there exists a permutation $\tau$ of $[K]$ such that $\pi=\tau(\pi'), \cG=\tau(\cG')$ and $\bphi=\tau(\bphi')$.
\end{enumerate}
\end{theorem}

The proof of the above theorem  is provided in Appendix \ref{app:proof_thm:isotropic_identifiability}. The key idea is a recursive peeling strategy by iterating between deconvolving, and using the previously established results, taking advantage the condition on $\cQ$. Note that under the convolutional structure of the components, the characteristic function of the mixture $P_{\cG,\pi,\bphi}$ has the form $\varphi_P(s) = \sum_{k=1}^K \pi_k \varphi_{G_k}(s)\varphi_{Q_{\phi_k}}(s).$
Using the condition of $\cQ$, we can divide both sides of the equality of characteristic functions by $\varphi_{Q_{\phi_*}}$ corresponding to the minimal noise level $\phi_*$ in the collection (this corresponds to a deconvolution step) leading to a mixture where some components (precisely those with $\phi_k=\phi_*$) are low-dimensional, and others still being absolutely continuous with respect to $\cL_D$. Here we can use Lemma \ref{lemma:identify_lower_dim_abs_cont} to isolate the singular parts, and further use Proposition \ref{prop:noiseless_identifiability} to identify the corresponding components. We keep peeling such components until we have matched them all. This proof technique effectively uses the ordering via $\phi$ which introduces a filtration of smoothing strength, and identifiability is achieved via repeated extraction of the least smooth layer.

\begin{remark}
    We make the following remarks regarding Theorem \ref{thm:isotropic_identifiability}. 
    \begin{enumerate}
        \item For Laplace distribution $q_{\phi}=\text{Laplace}(\boldsymbol{0}, \phi I)$ with $\phi>0$, the c.f. is $\varphi_{\phi}(t)=1/(1+\phi\norm{t}_2^2/2)$. Although this does not directly satisfy the condition required in the theorem, see the remark after the proof in the Appendix which demonstrates how the exact same strategy can also be used for this noise distribution to obtain identifiability.
        \item Unlike Proposition \ref{prop:noiseless_identifiability}, Theorem \ref{thm:isotropic_identifiability} requires finite $K<\infty$. The important problem in the case $K=\infty$ arises from the fact that although $\inf \phi_k$ always exist, the minimum might not. Thus, in the presence of such accumulation points, the current analysis would not work.
    \end{enumerate}
\end{remark}

As in the previous section, Theorem \ref{thm:isotropic_identifiability} also generalizes to the case when Assumption \hyperref[assume:A]{A} is replaced by Assumption \hyperref[assume:B]{B}.

\subsection{Parametric special case: polytope supports}
\label{sec: polytope case}

We state some of the results from the previous section in terms of the special case considered in Example \ref{example:polytope}. In this case, we parametrize each of the component latent measure as $G=(f_\Theta)_{\#}\mu$ where $f_{\Theta}:\Delta^{d-1}\to \bbR^D$ via $f_{\Theta}(\beta)=\sum_{j=1}^d \beta_j\theta_{j}$ (here $\theta_1',\dots,\theta_d'$ are the rows of $\Theta\in\bbR^{d\times D}$ with $d<D$), and $\mu$ is a probability measure on $\Delta^{d-1}$, with $\mu\ll \cH_d$ restricted to $\Delta^{d-1}$. As discussed in \ref{sec:related models}, this model has connections to topic models, archetypal analysis, convex NMF, and others, where $\theta_1,\dots,\theta_d$ are interpreted as the topics or end-members depending on context. Note that any such $G$ is low-dimensional in the sense of definition \ref{def:low-dim}, and compactly supported with $\cS=\supp(G)$ being a convex polytope with dimension at most $d-1$, with vertices from $\{\theta_1,\dots,\theta_d\}$.

For the overall mixture, assume that each $G_k=(f_{\Theta_k})_{\#}(\mu_k)$ for $k=1,\dots,K$, with $\Theta_k$ having columns $\theta_{k,1},\dots,\theta_{k,d_k}\in\bbR^D$. Moreover, assume that each $\mu_k$ is from a parametrized family of distributions on $\Delta^{d_k-1}$, for example, $\mu_k=\text{Dir}(d_k, \gamma_k)$. Moreover, for concreteness, assume that the noise kernel is Gaussian $Q_{\sigma}\equiv \cN_D(\boldsymbol{0}, \sigma^2 I)$. Then the overall model has the following hierarchical representation
\begin{equation}
    \begin{aligned}
        z|\pi &\sim \text{Cat}(\pi) \\
        \beta | z=k &\sim \mu_k \\
        X | z=k, \beta &\sim \cN_D\left(\Theta_k'\beta, \sigma_k^2 I\right).
    \end{aligned}
\end{equation}

The density of $X$ with respect to $\cL_D$ has the following expressions
\begin{equation}\label{eq: polytope density}
    p_{\bTheta, \pi, \Sigma, \bmu}(x) = \sum_{k=1}^K \pi_k \int_{\cS} \phi_D(x \mid \eta, \sigma_k^2 I) dG_k(\eta) = \sum_{k=1}^K \pi_k \int_{\Delta^{d_k-1}} \phi_D(x \mid \Theta_k'\beta, \sigma_k^2 I) d\mu_k(\beta),
\end{equation}
where $G_k$ is as defined above with $\cS_k=\supp(G_k)$, and the parametrization is in terms of the end-points $\bTheta=(\Theta_1,\dots,\Theta_K)$, the mixture weights $\pi\in\Delta^{K-1}$, the noise levels $\Sigma=(\sigma_1^2,\dots,\sigma_K^2)$ and the underlying \textit{admixing} distributions $\bmu=(\mu_1,\dots,\mu_K)$. 

Using the results from the previous section, under appropriate conditions on $G_1,\dots,G_K$ (which effectively are conditions on the underlying polytopes $\cS_1,\dots,\cS_K$), we can identify each of the $G_k$. However, in this parametrization of $G_k$, the $\Theta_k$ and $\mu_k$ are not unique. This stems from redundancy in the number of end-points. This non-identifiability is analogous to the non-minimal representations discussed earlier. The following example illustrates this.

\begin{example}
    Suppose $P$ is the distribution of $\sum_{j=1}^d \beta_j\theta_j$ with $\beta\sim \mu$, an absolutely continuous distribution on $\Delta^{d-1}$. Suppose $\theta_1,\dots,\theta_{m}$ are the vertices of the convex hull $\cS=\conv(\theta_1,\dots,\theta_d)$, and each $\theta_j$ with $m<j\leq d$ are convex combinations of $\theta_1,\dots,\theta_m$. Let $\Theta$ be the matrix consisting of the original points $\theta_1,\dots,\theta_d$ and $\Theta^*$ be the corresponding matrix with $\theta_1,\dots,\theta_m$. Let $\theta_j=\sum_{\ell=1}^m a_{j\ell}\theta_\ell$ be the convex representation for $j=m+1,\dots,d$. Then, we can write
    $$\sum_{j=1}^d \beta_j \theta_j = \sum_{\ell=1}^m \tilde{\beta}_\ell\theta_\ell, \quad \tilde{\beta}_\ell=\beta_\ell + \sum_{j=m+1}^d a_{j\ell}\beta_j, \quad \ell\in [m].$$
    Let $T$ be the linear map which takes $\beta\mapsto \tilde{\beta}$, as given above. Writing $\mu^*=T_{\#}\mu$, it is easy to check that $\mu^*$ is absolutely continuous on $\Delta^{m-1}$. Hence, such a $P$ has equivalent expressions $P=(f_{\Theta})_{\#}\mu = (f_{\Theta^*})_{\#}(\mu^*)$. This example illustrates that for such probability measures, the minimal representation is achieved when $\theta_1,\dots,\theta_d$ are vertices of the supporting polytope. In the special case that $\theta_1,\dots,\theta_d$ are affinely independent, $G$ is exactly $(d-1)-$dimensional with support $\cS$, a simplex.
\end{example}

\begin{figure}
    \centering
    \includegraphics[clip, trim=0.5cm 1.3cm 0.5cm 3.1cm, width=0.7\linewidth]{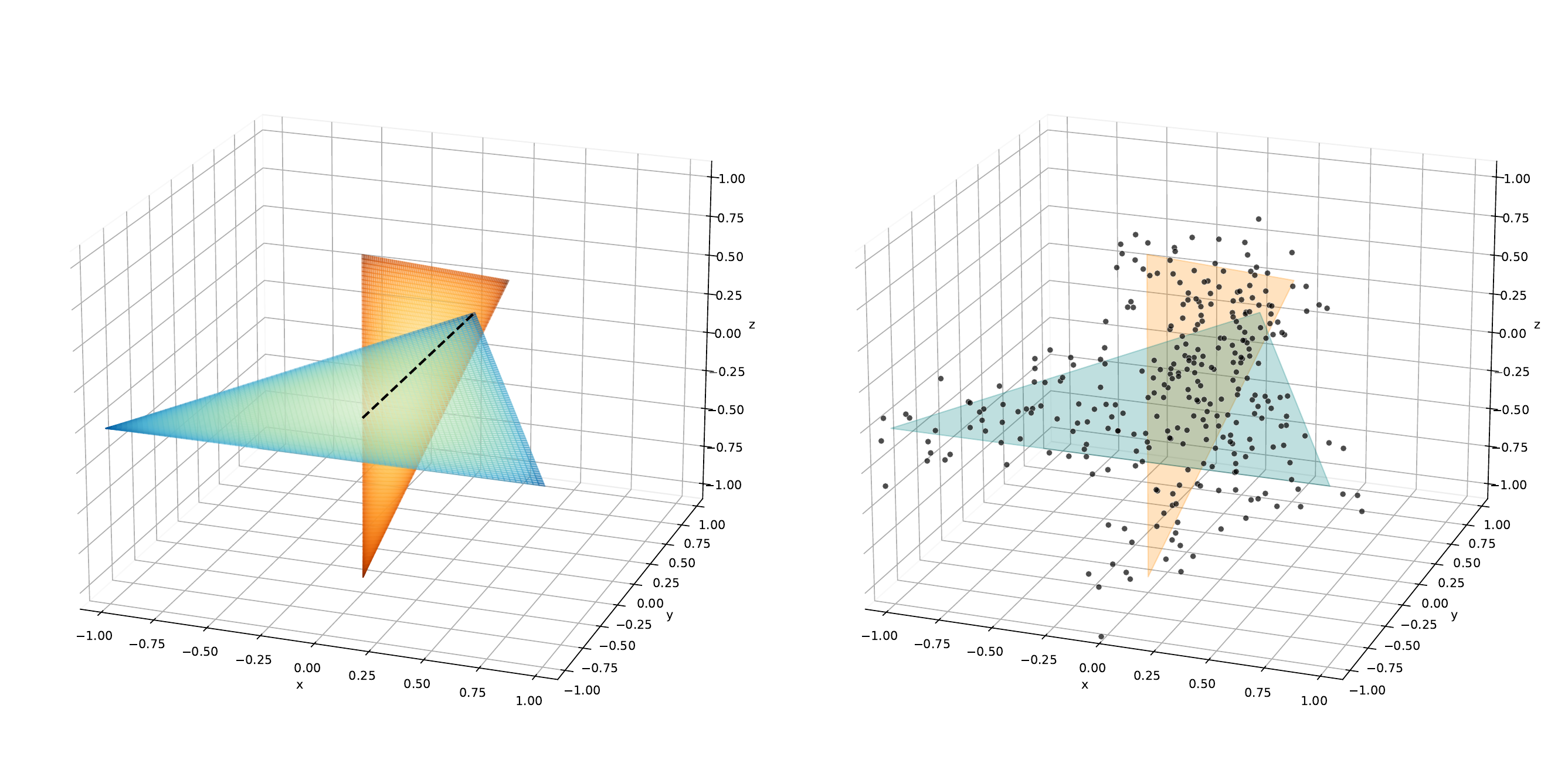}
    \caption{Example of model in $D=3$ (for visualization) with $K=2$ components, each latent measure $G_k$ is supported on a 2-dimensional polytope (triangle here) with the dashed line showing the intersection of the supports. Left panel shows the noiseless case, while the right panel shows a scatter plot of observations from from the model with Gaussian noise. The underlying $\mu_k$ is Dirichlet and its effect on $G_k$ can be seen by the color shading (left panel).}
    \label{fig: polytope example model}
\end{figure}

To remove such redundancy, we define $\Theta$ to be \textit{exposed} if the columns $\theta_1,\dots,\theta_d$ are all vertices of the polytope $\cS=\conv(\theta_1,\dots,\theta_d)$. Recall that for a convex polytope $\cS$, the collection of extreme points or vertices is defined as the set of points which cannot be expressed as the strict convex combination of two distinct points from the polytope, i.e., 
\begin{equation}\label{eq: vertices}
    \extr(\cS):=\{x\in\cS\mid \nexists \, u\neq v\in \cS \ni x = \lambda u + (1-\lambda)v \text{ for some } \lambda\in (0,1)\}.
\end{equation}
We will require that each $\Theta_k$ is exposed to remove the use of redundant end-members within a component. This condition requires that within each component, each of the parameters $\theta_{k,1},\dots,\theta_{k,d}$ in $\Theta_k$ are the vertices of the corresponding polytope $\cS_k$. This is needed since identification in this parametrization requires the additional geometric step of recovering $\Theta_k$ from $G_k$. See Definition \ref{def:exposed} for a slightly more general global definition of exposure for a collection of polytopes, as will be required for establishing the convergence rates of the parameters in Section \ref{sec:posterior_contraction}.

For the model $p=p_{\bTheta, \pi,\Sigma,\bmu}$ as above, we define the \textit{size} of the model as $s(p)=(K;(d_1,\dots,d_K))$ with $d_1\leq \dots\leq d_K$ being the ordered number of end-points in the components, i.e., after a permutation of the labels $\Theta_k\in\bbR^{d_k\times D}$. Then we can define a total order on the space of sizes as
$$(K;(d_1,\dots,d_K))>(K', (d_1',\dots,d_{K'}')) \iff K>K' \text{ or } \left(K=K'\text{ and } (d_1,\dots,d_K) > (d_1',\dots,d_K')\right)$$
where $(d_1,\dots,d_k')>(d_1,\dots,d_k')$ is the lexicographic ordering. With these notions, the discussion above, and the results from the previous section, we have the following result.

\begin{corollary}\label{cor: identifiability polytope}
    Suppose $P_{\bTheta,\pi,\Sigma, \bmu}, P_{\bTheta', \pi', \Sigma', \bmu'}$ be finite mixtures as above with $P_{\bTheta,\pi,\Sigma, \bmu}= P_{\bTheta', \pi', \Sigma', \bmu'}$.
    \begin{enumerate}
        \item Suppose \textit{both the models} satisfy (i) assumptions in Theorem \ref{thm:isotropic_identifiability}, and (ii) each $\Theta_k$ is exposed. Then up to a permutation, all the component parameters match.
        \item Suppose \textit{only } $P_{\bTheta,\pi,\Sigma,\bmu}$ satisfies the conditions in the previous part, then $s(p_{\bTheta,\pi,\Sigma,\bmu})\leq s(p_{\bTheta',\pi',\Sigma',\bmu'})$. Further, if $s(p_{\bTheta,\pi,\Sigma,\bmu})= s(p_{\bTheta',\pi',\Sigma',\bmu'})$, then up to a permutation, all the component parameters match.
    \end{enumerate}
\end{corollary}

Figure \ref{fig: polytope example model} shows a concrete example of this particular parametric subclass of the general semiparametric mixtures considered in previous sections.

\section{Convergence Rates of Parameter Estimates}\label{sec:posterior_contraction}

In this section, we restrict ourselves to a specific parametric class for the component latent measures in order to establish concrete rates of parameter estimation. In particular, we consider the class of models where each $G_k$ is supported on a convex polytope $\cS_k$, as discussed in Section \ref{sec: polytope case}. We are primarily interested in estimating all the component-specific end-members, mixture weights and noise variances and deriving contraction rates for the posterior distribution on these parameters under an appropriate Bayesian setting.

Recall the model $P_{\bTheta,\pi,\Sigma,\bmu}$ from Equation \eqref{eq: polytope density}. We simplify the setting by assuming $d_1=\dots=d_K=d$, i.e., each component has the same number of end-members. Moreover, we assume that the underlying measure $\mu$ in $G=(f_{\Theta})_{\#}\mu$ for any component are all identical, denoted as $p_\beta$, a fixed probability distribution on $\Delta^{d-1}$. We are interested in the rates for $\bTheta=(\Theta_1,\dots,\Theta_K)$, $\pi$ and $\Sigma=(\sigma_1^2,\dots,\sigma_K^2)$ in the well-specified setting where $K, d$ and $p_\beta$ are known. We assume that $p_\beta \ll \cH_d$ and $\supp(p_{\beta}) = \Delta^{d-1}$, to ensure that component latent measures $G_k$'s retain the low-dimensional structure as in the preceding sections and the support of $G_k$ is indeed a convex polytope of dimension at most $d-1$. Despite these simplifications, deriving parameter estimation rates in this class of models is highly non-trivial owing to the nested mixture structure, with a finite mixture structure on top of a (continuous) mixture for each component. 

While posterior contraction rates for density estimation can be established without much difficulty using the general theory in \cite{ghosal2017fundamentals}, rates for the parameter are significantly harder to establish. This requires understanding the map $(\bTheta,\pi,\Sigma) \mapsto P_{\bTheta,\pi,\Sigma}$. One way to achieve this is through the use of \textit{inverse bounds}, as in \cite{nguyen2013convergence}, \cite{nguyen2015posterior}, where a lower bound on the total variation distance is established by an appropriate metric on the parameter space. This allows one to transfer density estimation rates to parameter estimation under suitable conditions on the model. 

Inverse bounds for mixture of normal model has been established \cite{ho2016strong} for the mixing measure. Such an analysis extends to mixture kernels under a condition known as \textit{first-order identifiable} \citep{ho2016strong, guha2021posterior}, which is a considerably stronger condition than identifiability. 
\cite{nguyen2013convergence} established inverse bounds for \textit{any} mixing distribution convoluted with an appropriate noise kernel, which results in the following general inequality
\begin{equation}
    W_2^2(G,G')\lesssim \left(-\log d_{\TV}(p_{G}, p_{G'})\right)^{-1}, \quad \text{as } W_2(G,G')\to 0
\end{equation}
where $G,G'$ are mixing measures with bounded support and $p_G=\int \phi(\cdot|\eta)dG(\eta)$ is the infinite mixture model. In the case where $\sigma_1^2=\dots=\sigma_K^2$, as noted in Remark \ref{remark:model}(ii), the overall model can be seen as a single convolution $P_{\cG,\pi,\phi}=P_{\cG,\pi}\star Q_\phi$ and the result in the above display can be used to obtain the estimation result for the latent mixing measure $P_{\cG,\pi}$. Firstly, the use of such a general inverse bound gives a much slower rate of convergence. Secondly, convergence is obtained for the measure $P_{\cG,\pi}$, and cannot easily be extended to the underlying components $G_1,\dots,G_K$ separately. Finally, our model is not convolutional in general, and the component-specific variances add a lot of flexibility in modeling. 

It is clear that this particular setting presents its own set of challenges, where existing methods cannot be used directly. We follow the standard strategy through inverse bounds and the novelty of our approach lies in introducing a suitable geometric assumption on the underlying polytopes (see the notion of \textit{exposure} in Definition \ref{def:exposed}). This can be viewed as a type of separation between the components and seemingly a bit stronger than distinct affine spaces, as was sufficient for identifiability (a precise discussion is given later). However, this allows us to show that the mixture kernel in our model $\int \phi(x-\eta, \sigma^2 I) dG(\eta)$ is itself a first-order identifiable kernel. Note that this kernel is a continuous location mixture on its own. This allows us to derive a parametric estimation rate.

\subsection{Parametrization and Metric}
Let $K, d$ and $p_\beta$ be fixed. The model is parametrized by
$(\bTheta,\pi, \Sigma)$, where $\pi\in\Delta^{K-1}$ denotes the mixture probabilities and $\Sigma=(\sigma_1^2,\dots,\sigma_K^2)$ are the component variances. $\bTheta = \{\Theta_1,\dots,\Theta_K\}$ denotes the collection of end-members with $\theta_{k,j} \in \bbR^D$ for all $k\in [K], j\in [d]$, where the columns of $\Theta_k$ represent the end-members of component $k$ (thus, $\Theta_k\in\bbR^{D\times d}$). The model can be expressed as
\begin{equation}
\begin{aligned}\label{model:parametrization}
    z | \pi \sim \text{Cat}(\pi), \quad
    \beta \sim p_{\beta}, \quad 
    X|z=k, \beta, \bTheta, \Sigma &\sim \cN\left(\cdot \mid\Theta_k \beta, \sigma_k^2 I_D\right).
    \end{aligned}
\end{equation}
The model stipulates that for each data point, one of the $K$ components is chosen with probability given by $\pi$ and given that the component is $k$, a random vector is drawn $\eta\sim G_k=(f_{\Theta_k})_{\#} p_{\beta}$ supported on a convex polytope $\cS_k$ and a noisy version of this vector is observed $X=\eta+\epsilon$, where the noise $\epsilon$ is Gaussian with variance $\sigma_k^2$, depending on the chosen component.

For $p_\beta$, we assume that $p_\beta \ll \cH_{d-1}$ restricted to $\Delta^{d-1}$. Additionally, suppose $p_{\beta}$ is exchangeable, i.e., if $\beta\sim p_{\beta}$ then $(\beta_1,\dots,\beta_d)$ and $(\beta_{\tau(1)},\dots,\beta_{\tau(d)})$ have the same distribution for any permutation $\tau$ of $[d]$, with finite first moment. Most of the techniques used in the sequel will pass without this definition. This definition merely implies that for any component, the order of the vertices in the parametrization does not matter. If this assumption is not true, then one can simply ignore this permutation-invariance and proceed by redefining the metric appropriately. 

Denote the parameter space by $\Psi$ which contains all $\psi = (\bTheta, \pi,\Sigma)$ as discussed above with $\pi_k>0$ for all $k\in[K]$ to ensure that the model has exactly $K$ non-trivial components. The dependence on $K, d$ and $p_\beta$ is suppressed in the notation, and these are assumed fixed. For $\cT\subset \bbR^D$ and $S\subset \bbR$, let us define $\Psi_{(\cT,S)}\subset\Psi$ such that for any $(\Theta,\pi,\Sigma)\in \Psi_{(\cT,S)}$, $\theta_{kj}\in\cT$ and $\sigma_k^2 \in S$ for all $k\in[K], j\in[d]$.
Given $\psi\in \Psi$, the density takes the form
\begin{align}\label{eq:density_1}
    p_{\psi}(x) = \sum_{k=1}^K \pi_k \int_{\cS_k} \phi(x\mid \eta, \sigma_k^2 I) dG_k(\eta)=\sum_k \pi_k \int_{\Delta^{d-1}} \phi(x\mid \Theta_k^\top\beta, \sigma_k^2 I) dP_{\beta}(\beta).
\end{align}


 We consider the following pseudometric on the parameter $\psi\in \Psi$:
\begin{equation}
    \label{eq:bayes_metric}
    d(\psi, \psi') = \min_{\tau\in \bbS_K} \sum_{k=1}^K \left(d_M(\Theta_k, \Theta'_{\tau(k)}) + |\pi_k - \pi_{\tau(k)}'| + |\sigma_k^2 - \sigma_{\tau(k)}'^2|\right).
\end{equation}
In the above, $d_M(\Theta, \Theta') := \min_{\tau \in \bbS_d} \sum_{j=1}^d \norm{\theta_j - \theta'_{\tau(j)}}$ is a metric invariant to the permutation of the columns $\theta_1,\dots,\theta_d$ of $\Theta$. This invariance is required due to the exchangeability assumption on $p_\beta$, allowing us to estimate the end-members up to a permutation. 
It can be easily checked that $d$ is a pseudometric on $\Psi$, and hence also on $\Psi_{(\cT,S)}$. 

We first establish an upper bound for the KL-divergence between two different models in terms of the distance between the corresponding parameters.

\begin{proposition}\label{prop:kl_upper bound}
    For any $\psi, \psi'\in \Psi_{(\cT,S)}$ where $\cT$ and $S$ are compact with $\inf S = \underline{\sigma}^2>0$, we have
    \begin{align}
        \KL\left(P_{\psi}\Vert P_{\psi'}\right) \leq \frac{1}{2\underline{\sigma}^2}\left(d(\psi,\psi')^2 + C d(\psi,\psi'))\right),
    \end{align}
    where $C=C(\cT,S)=D + (\diam(\cT)^2 + D\diam(S))/2$ is a constant, free of $K$ and $\underline{\sigma}^2$.
\end{proposition}
The proof of the proposition is provided in Appendix \ref{app:proof_prop:kl_upper bound}. Note that the restrictions imply that the parameter space is compact and the component variances are bounded away from 0. This upper bound allows us to control entropy of the model class and prior mass conditions of the KL-ball around the truth through corresponding entropy for the parameter space and Euclidean neighborhoods around the parameter, that will be required in the posterior contraction result in Section \ref{sec:posterior_contraction_rates}. An interesting aspect of the above result is that it does not use any assumption on $\bTheta$ (such as those required for identifiability, or those to be considered in the next section). Lastly, we note that for small enough $d(\psi,\psi')=\epsilon$, the dominating term in the upper bound is the linear term.

\subsection{Inverse Bound}
In this section, we derive a key lower bound on the total variation distance by the metric between the corresponding parameters. Note that this is the opposite direction of what we showed in Proposition \ref{prop:kl_upper bound}. This step is typically hard. In our case, there are several additional challenges compared to those in the existing finite mixture setting. Firstly, the model considered in this work is a mixture model \eqref{eq:density_1}, where the kernel is 
$$f(\cdot|\Theta) = \int_{\cS} \phi(\cdot|\eta,\sigma^2 I) dG(\eta),$$
where $G$ is a probability measure on a polytope $\cS$.  The notion of first-order identifiability, which results in inverse bounds, is related to the linear independence of the kernel and its gradients with respect to the parameters. In our case, even within a component, there is a (continuous) mixture with the Gaussian location kernel. Understanding linear independence of such a kernel along with its derivative with respect to the parameters is challenging. Apart from the interaction between different components, there is interaction between parameters associated to a single component, i.e., the end-members of a single component -- this posits additional difficulty to analyze such a kernel. Our approach is to introduce a suitable and natural geometric assumption, that allows us to overcome these challenges.  We define these notions of \textit{sequential} and \textit{total exposure} related to the geometric arrangement of the component polytopes in the following.

\begin{figure}
    \centering
    \includegraphics[clip, trim=0cm 4cm 0cm 7cm, width=0.8\linewidth]{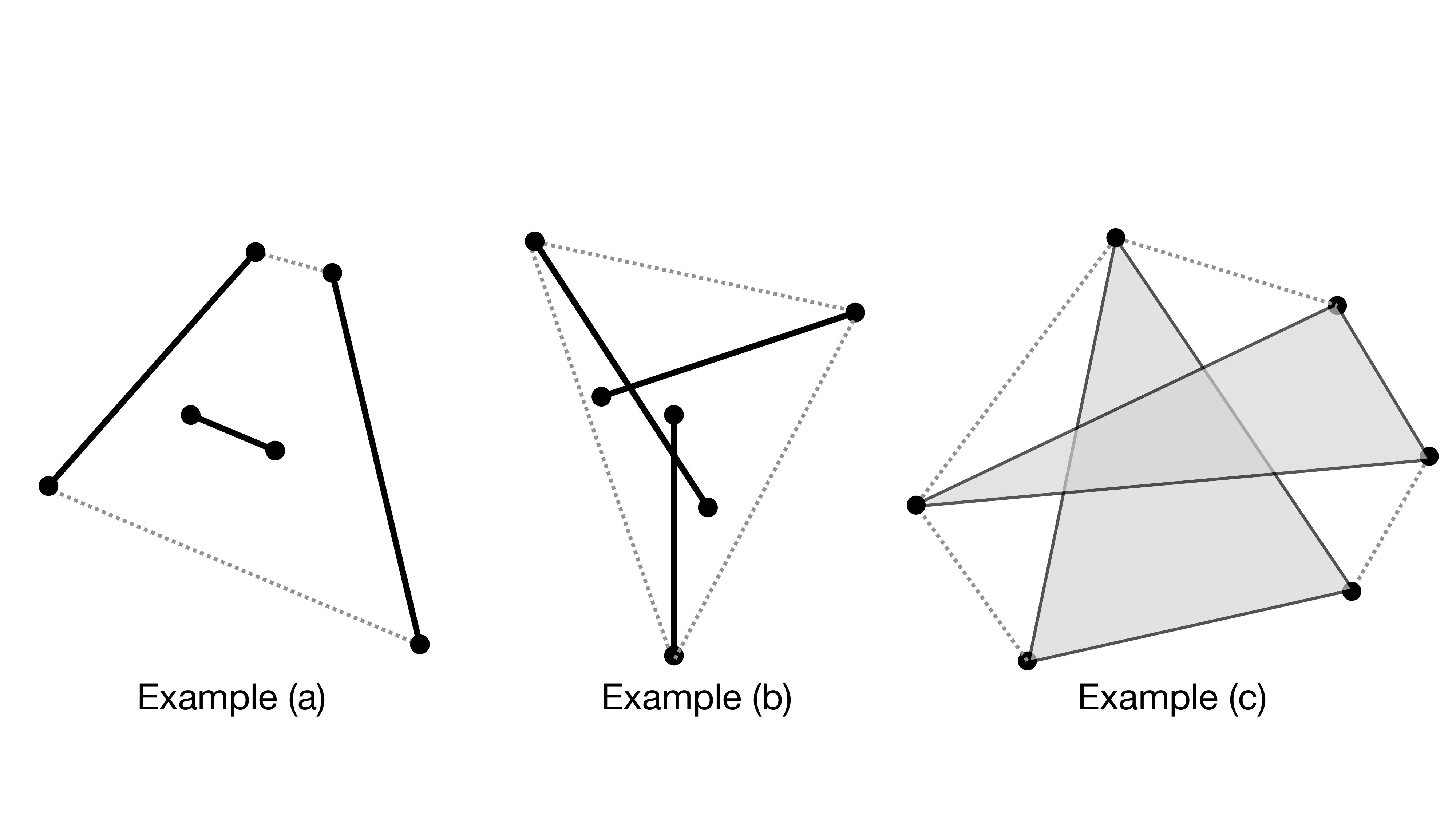}
    \caption{Examples in $\bbR^2$ illustrating total exposure definition: Examples (a) and (b) satisfy \hyperref[assume:A]{A}, but neither are totally exposed In (a), two of the three polytopes are exposed, while no polytope is exposed in (b). Example (c) is totally exposed but does not satisfy \hyperref[assume:A]{A}. Note when ambient dimension $D$ is large and component polytopes are in general position, they almost surely satisfy both \hyperref[assume:A]{A} and totally exposed.}
    \label{fig:sequentially_exposed}
\end{figure}

\begin{definition}\label{def:exposed}
 Given $\bTheta=(\Theta_1,\dots,\Theta_K)$, let $\cS_k=\conv(\theta_{k,1},\dots,\theta_{k,d})$ be the polytope corresponding to $\Theta_k$. Denote $\cC(\bTheta)=\{\theta_{k,j}\mid k\in[K], j\in [d]\}$.
\begin{enumerate}[label=(\roman*)]
    \item $\theta\in \cC(\bTheta)$ is said to be exposed with respect to $\bTheta$ if $\theta\in \extr(\conv(\cC(\bTheta)))$ and $\theta\in \cS_j\setminus \cup_{k\neq j} \cS_k$ for some $j\in[K]$.
    \item $\Theta_k$ is said to be exposed with respect to $\bTheta$ if $\theta_{k,j}$ is exposed with respect to $\bTheta$ for all $j\in[d]$.
    \item $\bTheta$ is said to be totally exposed if $\Theta_k$ is exposed with respect to $\bTheta$ for any $k\in[K]$.
\end{enumerate}
\end{definition}

The above definition extends the notion of \textit{exposed} defined in Section \ref{sec: polytope case}, which can be seen as a special case of Definition \ref{def:exposed}(i) for the case $\bTheta=(\Theta)$, corresponding to a single component -- we will refer to this special case as $\Theta$ is exposed with respect to itself, for clarity. A point $\theta\in \cC(\bTheta)$ is said to be exposed if it is not only an extreme point of the global convex hull, but also belongs uniquely to a single polytope. Thus, for a collection of polytopes, exposed points are those that are both globally extremal and locally identifiable to a specific component. If all vertices of $\Theta_k$ are exposed (with respect to the collection $\bTheta$), then we call this component exposed. Note that if $\Theta_k$ is exposed with respect to $\bTheta$, then $\theta_{k,1},\dots,\theta_{k,d}$ are all vertices of the global convex hull, and hence must be vertices of $\cS_k$. 
Total exposure requires each component to be simultaneously exposed with respect to the full collection. See Figure \ref{fig:sequentially_exposed} for an illustration.

\begin{remark}  \textit{    (Assumption \hyperref[assume:A]{A} versus exposure)}
    The assumption \hyperref[assume:A]{A} and the notions of exposure are not directly connected. Total exposure does not necessarily imply Assumption \hyperref[assume:A]{A} -- there are examples of convex polytopes with the same affine space that are totally exposed, yet have non-trivial intersection. Conversely, a collection of polytopes satisfying \hyperref[assume:A]{A} need not be totally exposed. See Figure \ref{fig:sequentially_exposed} for examples.
\end{remark}

We provide a brief discussion into the identifiability theory with such exposure assumptions. In the following, we call $\psi_0\in\Psi_{(\cT,S)}$ identifiable if for any $\psi\in\Psi_{(\cT,S)}$, $d_{\TV}(p_{\psi},p_{\psi_0})=0\Rightarrow d(\psi,\psi_0)=0$. In the current parametrization, if the supporting polytopes $\cS_1,\dots,\cS_K$ of $\psi_0$ have distinct affine spaces, then they satisfy \hyperref[assume:A]{A}. Together with assuming that each $\Theta_k$ is exposed (with respect to itself), Corollary \ref{cor: identifiability polytope} ensures that $\psi_0$ is identifiable.  The results in Section \ref{sec:identifiability} were general and applicable for nonparametric mixtures, and as such required the component supports to be separated in the sense of Assumption \hyperref[assume:A]{A}. However, for the current class of parametrized mixtures whose component latent measures have polytope supports, we can extend the identifiability result to be based on sequential exposure instead. 

\begin{lemma}\label{lemma:identifiability_sequentially exposed}
        Consider $\psi_0=(\bTheta_0, \pi_0, \Sigma_0)\in\Psi_{(\cT,S)}$ such that $\bTheta_0$ is totally exposed. Then $\psi_0$ is identifiable, i.e., for all $\psi\in\Psi_{(\cT,S)}$  we have $d_{\TV}(p_{\psi},p_{\psi_0})=0\Rightarrow d(\psi,\psi_0) =0$.
    \end{lemma}

Although Assumption \hyperref[assume:A]{A} was sufficient for identifiability purpose, it is not enough for first-order identifiability which requires a slightly stronger notion of separation. The notion of total exposure turns out to be precisely what is required to prove the following inverse bound.

\begin{theorem}\label{thm:inverse_bound}
    Let $\psi_0\in \Psi$, such that $\bTheta_0$ is totally exposed. Then we have the following.
    $$\liminf_{\substack{\psi\in \Psi \\d(\psi,\psi_0)\to 0}} \frac{d_{\TV}(P_{\psi}, P_{\psi_0})}{d(\psi,\psi_0)} > 0.$$
\end{theorem}

\begin{proof}[Proof sketch]
The proof of the inverse bound is in the Appendix \ref{app:proof_thm:inverse_bound}. We use the standard proof by contradiction by assuming the existence of a sequence $\psi_n$ such that $d(\psi_n,\psi_0)\to 0$ and the ratio  $d_{\TV}(p_{\psi_n},p_{\psi_0})/d(\psi_n,\psi_0) \to 0$ as $n\to\infty$. This leads us to the first-order identifiability condition
\begin{align*}
        \sum_{k\in[K]} b_k\int &\phi(x\mid \Theta_k^\top\beta, \sigma_k^2)dP_{\beta}(\beta) + \sum_{k\in[K]} \frac{\pi_k}{\sigma_k^2} \int \left(\sum_j \beta_j a_{kj}\right)^\top \left(x - \Theta_k^\top\beta\right) \phi(x\mid \Theta_k^\top\beta,\sigma_k^2)dP_{\beta}(\beta) \nonumber\\
        &+ \sum_{k\in[K]} \pi_k \int c_k\left(\frac{\norm{x - \Theta_k^\top\beta}^2}{2\sigma_k^4} - \frac{D}{2\sigma_k^2}\right)\phi(x\mid \Theta_k^\top\beta,\sigma_k^2) dP_{\beta}(\beta) = 0 \text{ a.s. } x\in\bbR^D
    \end{align*}
where all the parameters $\bTheta, \pi, \sigma_k$ correspond to $\psi_0$. The coefficients $\ba, \bb, \bc$ correspond to appropriate limits, such as $b_k = (\pi_k^{(n)} - \pi_k^{(0)}) / d(\psi_n, \psi_0)$, $a_{kj} = (\theta_{kj}^{(n)} - \theta_{kj}^{(0)})/d(\psi_n, \psi_0)$, etc. (using suitable subsequence, if required). This implies that not all of these coefficients can be 0 -- the remainder of the proof shows that if the above display is true, then all the coefficients must be zero which is the required contradiction. 

Dealing with the intractable integrals in the above display is the primary problem, noting that within the integral the effects of all the end-members in that component are combined together. Hence, for example, even if we show $a_{k1}=\boldsymbol{0}$, the effect of $\theta_{k1}$ is still there in the identity. Furthermore, the term $\sum_j \beta_j a_{kj}$ in the integrand poses challenges since isolating the effect of a particular $a_{kj}$ is hard. This is exactly where the geometric assumption of total exposure comes in. Under the assumption, we show that it is possible to re-parametrize in such a way that the components can be handled one at a time, and the effect of the parameters corresponding to a component can be analyzed. The exposed condition allows us to control the integral, in particular the mass in a small neighborhood around the $\theta_{k,j}$'s.
\end{proof}

    Note that identifiability of $\psi_0$ is a global condition, while the inverse bound above is of a local nature. Theorem \ref{thm:inverse_bound} ensures that given $\psi_0$ with totally exposed component polytopes, there exists $\epsilon(\psi_0)>0$ and a constant $C(\psi_0)$ such that whenever $d(\psi,\psi_0)<\epsilon$ with $\psi\in\Psi$, we have $d_{\TV}(p_{\psi},p_{\psi_0})\geq C(\psi_0) d(\psi,\psi_0)$. Following a similar argument as in Lemma 5.6 \cite{wei2020convergence}, the local nature of this condition could be made global restricting to $\Psi_{(\cT,S)}$ by leveraging compactness of the parameter space. In particular, if $\cT,S$ are compact, using $\psi\mapsto d_{\TV}(p_{\psi},p_{\psi_0})/d(\psi,\psi_0)$ is a continuous map and the fact that $\psi_0$ is identifiable (thanks to Lemma \ref{lemma:identifiability_sequentially exposed}), we can have a single constant $C(\psi_0)$, such that for any $\psi\in \Psi_{(\cT,S)}$, we have
    \begin{equation}\label{eq:lower_bound_tv}
        d_{\TV}(p_{\psi},p_{\psi_0}) \geq C(\psi_0) d(\psi,\psi_0).
    \end{equation}
Note that the above inequality implies that the parameter estimation rate cannot be worse than the density estimation rate, up to constants depending only on the ground truth model. Thus, to derive estimation rates for the parameter, we need estimation rates for the density. This is the approach we take in the following subsection.

\subsection{Posterior Contraction Rates}
\label{sec:posterior_contraction_rates}
Now we are ready to discuss rates for posterior contraction. Given i.i.d. samples $X_1,\dots,X_n$ from $P_{\psi_0}$, we place a prior $\Pi$ on the parameter space $\Psi_{(\cT,S)}$. Application of Bayes rule gives the posterior distribution of $\psi$ as
$$\Pi(\psi\in A \mid X_1,\dots,X_n) = \frac{\int_A \prod_{i\in[n]} p_{\psi}(X_i) d\Pi(\psi)}{\int_{\Psi_{(\cT,S)}}\prod_{i\in[n]} p_{\psi}(X_i) d\Pi(\psi) },$$
for any measureable set $A$. We wish to show that this posterior distribution concentrates around $\psi_0$ as $n\to\infty$. Towards this, we first show that the posterior distribution for the density estimation contracts to $p_{\psi_0}$ in Hellinger metric using standard techniques from \cite{ghosal2000convergence}. Note that the prior $\Pi$ on $\psi$ induces a prior distribution $\tilde{\Pi}$ on the model class $\{p_{\psi} : \psi\in \Psi_{(\cT,S)}\}$. Typically, a prior $\Pi$ on $\Psi_{(\cT,S)}$ is constructed by placing independent prior distributions for $\theta_{kj}\in\cT$ for all $k,j$, a prior distribution for $\pi\in\Delta^{K-1}$ independently and for each $k$, independent prior distributions for $\sigma_k^2\in S$. 

\begin{theorem}\label{thm:density_contraction}
    Let $\psi_0\in \Psi_{(\cT,S)}$ with $\cT\subset\bbR^D,S\subset R_+$ compact with $\inf S > 0$, such that the component polytopes of $\psi_0$ are totally exposed. Suppose $\Pi$ is a prior on $\Psi_{(\cT,S)}$ such that $\Pi=\Pi_{\bTheta}\times \Pi_{\pi}\times \Pi_{\Sigma}$, where $\Pi_{\bTheta}$ is a probability measure on $\cT^d$ that is absolutely continuous with a density bounded away from $0$ and $\infty$, $\pi_{\pi}$ is a probability measure on $\Delta^{K-1}$, absolutely continuous with respect to Lebesgue on $\bbR^{K-1}$, with density bounded away from 0 and $\infty$, and $\Pi_{\Sigma}$ is a probability measure on $S^K$, similarly.
    Then for sufficiently large $M$, as $n\to\infty$, we have the following:
    \begin{align}
        \Pi\left(d(\psi, \psi_0) > M\sqrt{\frac{\log n}{n}} \,\Big| \, X_1,\dots, X_n\right) \to 0 \text{ in } P_{\psi_0}^{\infty}-\text{probability}.
    \end{align}
\end{theorem}
 
\textit{Proof sketch: } The proof is provided in Appendix Section \ref{app:proof_contraction_rate} -- here we provide a high-level idea of the approach. Using the theory of posterior contraction for general densities, we aim to first establish a parametric contraction rate for the density estimation problem and then use inverse bound to translate it for parameter estimation. The former typically requires control of the model complexity through entropy numbers and control on the prior mass on a small neighborhood of the truth. The former can be accomplished in a straightforward manner thanks to the upper bound on the KL established earlier, which also gives an upper bound for the Hellinger distance between two densities from the model family, in terms of the distance between the corresponding parameters measured in terms of the metric $d$. For the other condition, the neighborhood is measured in terms of a KL-ball $\{p: d_{\KL}(p_0\Vert p)\leq \epsilon^2, K_2(p_0,p)\leq \epsilon^2\}$ where $K_2(p,q)=\bbE\left(\log (p/q) - d_{\KL}(p\Vert q)\right)^2$ -- however, it turns out controlling this quantity becomes challenging for this model. Hence, we use Theorem 8.11 \cite{ghosal2017fundamentals} which only use KL-balls of the form $\{p:d_{\KL}(p_0\Vert p)\leq \epsilon^2\}$, which can be controlled via metric $d$ because of the upper bound \ref{prop:kl_upper bound}. However, this approach also requires an upper bound on $\Pi(h(p,p_0)\leq \epsilon)$, which we control using the inverse bound since $\sqrt{2} h \geq d_{\TV} \gtrsim d$. After establishing these bounds, we just verify the conditions of Theorem 8.11 \cite{ghosal2017fundamentals}.

Note that the particular assumption on the prior is quite general and the typical construction in practice. However, the result can be easily extended to other priors, so that the prior mass condition is satisfied with respect to the target model $p_{\psi_0}$. 

\section{Numerical Study}\label{sec:numerical study}

\subsection{Algorithms}\label{sec:algorithm}
In this section, we give a brief description of the algorithms used for estimation of the parameters in the model. We first start with the case $K=1$, i.e. corresponding to a single component. In the sequel we use these methods as part of the overall algorithms for the general case $K>1$, where we use Expectation-Maximization and Metropolis-Hastings MCMC as our two algorithms.

\subsubsection{Single Component}
Let $\Theta\in \bbR^{D\times d}$ be the matrix containing the vertices of the target polytope as the columns, where we always assume $d<D$. The density of the data $X$ given $\theta, \sigma^2$ if given by
$$p(x|\Theta,\sigma^2) = \int \phi(x\mid \Theta\beta, \sigma^2 I_D)P_{\beta}(d\beta).$$
In the special case $p_{\beta}$ is a symmetric Dirichlet distribution and columns of $\Theta$ are linearly independent, a very fast geometric algorithm \textit{Voronoi Latent Admixture} (VLAD) \citep{yurochkin2019dirichlet} exists. Note that VLAD can also estimate $\alpha$ associated with the symmetric Dirichlet $p_{\beta}$ by a grid search.

For the case, $p_{\beta}$ is a Dirichlet distribution, not necessarily symmetric, a spectral algorithm can be developed utilizing the decomposition of the moment tensors of the Dirichlet distribution \citep{do2025dirichlet} up to the third order. This approach is based on \cite{anandkumar2012spectral} (Algorithm 5), but suitably modified to include the Gaussian kernel (instead of the multinomial kernel as in Latent Dirichlet Allocation). We describe this algorithm in detail in Appendix \ref{app:spectral_algorithm}. This algorithm can estimate $\balpha\in\bbR_+^d$ for the Dirichlet $p_{\beta}$, given $\bar{\alpha}=\sum_j \alpha_j$, but requires the columns of $\Theta$ to be linearly independent. Note that when $p_{\beta}$ is symmetric Dirichlet, this requires knowing $d\alpha$, which boils down to knowing $\alpha$.

A third approach involves a Gaussian approximation of the mixing measure $p_{\beta}$, assuming this is known. However, this algorithm gives up identifiability, since with a Gaussian kernel, only the column space of the covariance is identifiable, which means that the vertices can be estimated up to a rotation. We fix this by estimating an optimal rotation using K-Means clustering of the data. Compared to the spectral algorithm (which is known to be exact when the vertices are linearly independent), this algorithm works reasonably well even when the vertices form a convex polytope (not just a simplex).

The final class of algorithms we consider are Markov Chain Monte Carlo (MCMC) methods, under a fully Bayesian setting with the parameters associated with $p_{\beta}$ (i.e., the Dirichlet parameter $\alpha$) unknown. We utilize two algorithms -- the first one augments the data with the latent variables $\beta$ and develops a Metropolis Hastings within Gibbs sampler noting that conditional on $\beta$, the parameters $\Theta,\sigma^2$ has closed form posterior under traditional Bayesian regression setting (e.g. with a Normal-Inverse-Gamma prior), while the second approach marginalizes over the latent variables and uses a pseudo-marginal Metropolis Hastings \citep{andrieu2009pseudo} algorithm (in particular, the Grouped Independent Metropolis Hastings of \cite{beaumont2003estimation}). They are discussed in the Appendix \ref{app:MCMC_algorithm_single}.

\subsubsection{Multiple Components}
\label{subsec:multiple_components}
The previous section demonstrates that estimating the vertices under such a model is challenging even without the mixture (i.e., $K=1$). In this section, we discuss an algorithm for the case $K>1$. Apart from the algorithm described here, we also use other algorithms including two EM-based methods and a Grouped Independent Metropolis Hastings sampler, all described in Appendix \ref{app:multiple_component_algorithms}.

We consider an approach using a Gaussian mixture model as an approximation to the original model. The main challenge with the original model is due to the intractable integral over the latent variable
$$p(x|\pi,\bTheta,\sigma^2) = \sum_k \pi_k \int \phi(x\mid \Theta_k\beta, \sigma_k^2 I) P_{\beta}(d\beta).$$
In this approach, we thus approximate the model uniformly by Gaussian mixture models as follows. Let $Q^{(M)}$ be the (discrete) uniform distribution on a uniformly spaced grid of points $\beta_1,\dots,\beta_M$ in $\Delta^{d-1}$. These $\beta_1,\dots,\beta_M$ are fixed and not considered as random. Given these, we can construct a mixture of Gaussian distributions with $KM$ components, where the $(k,j)-$th component is $\cN(\cdot\mid \Theta_k\beta_j,\sigma_k^2 I)$ with mixture probability $\pi_k \nu_j/M$, where $\nu_j=p_{\beta}(\beta_j)$. Let us denote this model as $P^{(M)}$, which has the following likelihood
\begin{equation}\label{eq:em_approx}
    p^{(M)}(x\mid \pi,\bTheta,\sigma^2) = \sum_{k\in[K], j\in[M]} \frac{\pi_k \nu_j}{M} \phi(x\mid \Theta_k\beta_j,\sigma_k^2 I).
\end{equation}
Note that as $M\to\infty$, we have $\|p^{(M)} - p\|_{\infty} \to 0$. We therefore employ EM algorithm on this approximate model $p^{(M)}$ (with suitable level of $M$ based on computational resource). We next discuss the E and M steps of this algorithm, which closely follows the EM algorithm for the Gaussian mixture model, with the restrictions on the component parameters now.

\vspace{1em}
\textbf{E-Step:} Suppose we have $\xi_t=(\pi{(t)},\bTheta{(t)}, \sigma^2(t))$ as the current parameters. In the E-step, we compute the expected value of the latent cluster variable, which in turn is the weight associated to data point $X_i$ for each of the components in the mixture. In this case, it takes the following form

\begin{equation}\label{eq:E-step approx model}
    w_{i,(k,j)}(t+1) := \bbP_{z|x,\xi_t}(Z_i=(k,j)) \propto \pi_k(t)\nu_j \phi(x_i\mid \Theta_k(t)\beta_j, \sigma_k^2(t) I)
\end{equation}
We need to compute this for each $i\in[n]$ and $(k,j)\in [K]\times [M]$. The weights are scaled so that $\sum_{k,j} w_{i,(k,j)}(t+1) = 1$ for each $i$.

\vspace{1em}
\textbf{M-step:} In this step, the parameters are updated based on maximizing the expected complete log-likelihood, where the expectation is taken with respect to the conditional distribution of the latent cluster allocation variables $Z_1,\dots,Z_n$. The complete log-likelihood is given by
$$\ell_c(\xi|\bX) = \sum_{i\in[n]} \sum_{k,j} 1(Z_i=(k,j)) \left[\log(\pi_k) + \log \phi(X_i\mid \Theta_k\beta_j, \sigma_k^2 I)\right],$$
which gives the conditional expectation as
\begin{equation}
\begin{aligned}
    \bbE_{\bZ|\bX,\xi_t}[\ell_c(\xi|\bX)] &= \sum_{i,j,k} w_{i,(k,j)}(t+1) \left[\log(\pi_k) -\frac{D}{2}\log(2\pi\sigma_k^2) - \frac{1}{2\sigma_k^2}\norm{X_i - \Theta_k\beta_j}_2^2\right] \\
    &= \sum_{i,j,k} w_{i,(k,j)} \log(\pi_k) -\frac{D}{2}\sum_{i,j,k} w_{i,(k,j)}\log(2\pi\sigma_k^2) - \sum_{i,j,k} \frac{w_{i,(k,j)}}{2\sigma_k^2} \norm{X_i - \Theta_k\beta_j}_2^2,
\end{aligned}
\end{equation}
where we dropped the $t+1$ from the second line. Note that these weights depend on the current values of the parameters. However, in the M-step, we maximize the above display (given the weights) for the parameters $\pi,\bTheta,\sigma^2$ to obtain the updates for the next iteration. For notational convenience, denote
$w_{\cdot k\cdot} = \sum_{i,j} w_{i, (k,j)}.$
The updates take the following closed form solutions
\begin{equation}
\begin{aligned}\label{eq:M-step approx model}
    \pi_k(t+1) &\propto w_{\cdot k \cdot}, \quad \text{ subject to }  \sum_k \pi_k(t+1)=1 \\
    \Theta_k(t+1) &= \left(\sum_{i,j} w_{i,(k,j)} X_i\beta_j^\top \right)\left(\sum_{i,j} w_{i,(k,j)} \beta_j\beta_j^\top\right)^{-1}\\
    \sigma_k^2(t+1) &= \frac{1}{Dw_{\cdot k \cdot}}\sum_{i,j} w_{i,(k,j)}\norm{X_i - \Theta_k(t+1)\beta_j}_2^2.
\end{aligned}
\end{equation}
The overall algorithm starts with some initial choice for parameters $\xi_0$ and iterating between the E-step (equation \eqref{eq:E-step approx model}) and M-step (equation \eqref{eq:M-step approx model}) till convergence. To monitor convergence, we can keep track of the current value of the evidence lower bound, which can be computed by plugging in the optimal parameters in the expression for $\bbE_{\bZ|\bX,\xi_t}[\ell_c(\xi|\bX)]$, to get
$$\cE_{t+1} = \sum_k w_{\cdot k\cdot} \log (w_{\cdot k \cdot}) - \sum_k w_{\cdot k \cdot}\log (n) - \frac{D}{2}\sum_k w_{\cdot k \cdot}\log(2\pi \sigma_k^2(t+1)) - \frac{nD}{2}.$$
 In the implementation, we use $\beta_1,\dots,\beta_M$ drawn from a uniform distribution at the start of the algorithm -- however, it might be better to use different distribution (concentrated near the boundary) and replacing Equation \eqref{eq:em_approx} with an importance sampling step to make more efficient use of the $M$ latent variables.

 \subsection{Simulations and Results}\label{sec:simulations}
\begin{table}
    \centering
\begin{tabular}{ p{2cm} p{9cm} p{2cm}  }
 \hline
 Name & Brief description  & Reference\\
 \hline
 VLAD   & Geometric Algorithm \cite{yurochkin2019dirichlet} &   -\\
 Spectral &  Spectral Algorithm  \cite{anandkumar2012spectral} & \ref{app:spectral_algorithm}\\
 Gaussian & Approximate $p_{\beta}$ by Gaussian &  \ref{app:approx_gaussian_algorithm}\\
 MCMC (aug)  & MCMC (Gibbs) algorithm by augmenting with latent variables $\beta$ &  \ref{app:MCMC_algorithm_single}\\
 MCMC (mar)  & MH algorithm by marginalizing latent $\beta$ &  \ref{app:MCMC_algorithm_single}\\
 \hline
\end{tabular}
\caption{Algorithms used for the Single Component Simulations}
\label{table:algorithms_single}
\end{table}
\begin{figure}
    \centering
    \includegraphics[width=1.\linewidth]{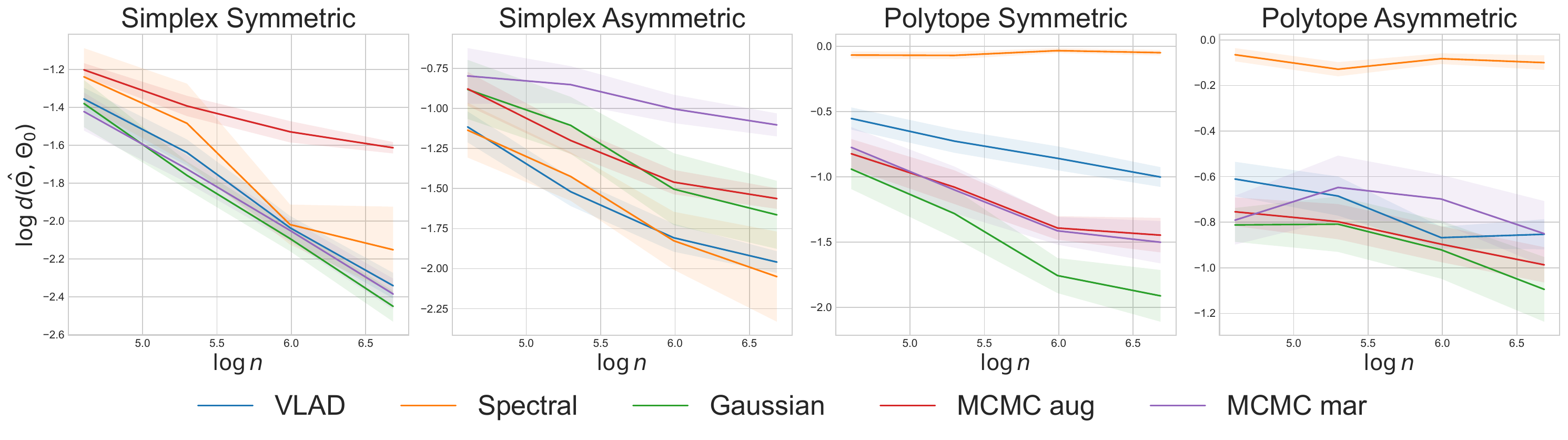}
    \caption{Simulation Results in a Single Component Setting}
    \label{fig:single_component_result}
\end{figure}

\subsubsection*{Single Component}
The first set of simulations deals with the case of a single component $K=1$, using the various algorithms discussed above. We set $D=20$ for these simulations and noise level $\sigma=0.4$. For the ground-truth component, we consider two cases - $d=3$ with affinely independent  vertices (hence the component polytope is a simplex) and $d=4$ with non-affinely independent vertices (these are chosen such that the component polytope is a 2-dimensional object with the 4 vertices exposed, i.e., a convex quadrilateral). In each of these cases, we consider the underlying mixing distribution $p_{\beta}$ to be either a symmetric Dirichlet distribution (with $\alpha=0.8$ in the simplex case and $\alpha=0.5$ in the polytope case) or an asymmetric Dirichlet distribution (with $\alpha=(0.2, 0.6, 1.0)$ in the simplex case and $\alpha=(0.3, 0.4, 1.0, 0.6)$ in the polytope case).  These settings are shown in Figure \ref{fig:single_component_settings} in the Appendix, along with the estimated polytopes using each of the methods in a single experiment with sample size 400. We use the five algorithms discussed above (shown in Table \ref{table:algorithms_single}) and evaluate the performance based on 20 repeated experiments for each setting for $n\in\{100,200,400,800\}$ -- the performance is measured in terms of estimating the ground-truth polytope using $d(\hat{\Theta},\Theta_0)$ where $d=d_M$ as in Equation \eqref{eq:bayes_metric}. Figure \ref{fig:single_component_result} shows the results from these experiments. We note that for the polytope case, the spectral algorithm does not work well, this is due to the fact that it requires the vertices to be linearly independent. Apart from that, the fast methods VLAD and spectral algorithm work very well in the simplex case. In the polytope case, the MCMC algorithms perform better (note that all other algorithms use some information about the underlying $\alpha$, which both the MCMC algorithms treat this as unknown). The Gaussian approximation algorithm seems to perform well in all the cases. It is worth noting that the slopes of the plots in the simplex case  are close  to the theoretical parametric estimation rates, while for the polytope case, they are slightly slower - this is perhaps attributable to the asymmetric Dirichlet placing much lower mass around one of the vertices (see setting in  Figure \ref{fig:single_component_settings}).

\subsubsection*{Multiple Components}
\begin{table}
    \centering
\begin{tabular}{ p{3cm} p{8cm} p{2cm}  }
 \hline
 Name & Brief description  & Reference\\
 \hline
 Geometric (Geom)   & MPPCA + VLAD &   \ref{subsec:geometric_multiple_components}\\
 Gaussian (Gaussian) &  Approximate $p_{\beta}$ by Gaussian with EM framework & \ref{subsec:gaussian_approx_multiple_components}\\
 MCMC (MCMC) &Grouped Independent MH Sampler marginalizing $\beta$ &  \ref{subsec:mcmc_multiple_components}\\
 EM $M$  &EM for Approximate Model using $M$ fixed draws of $\beta$ &  \ref{subsec:multiple_components}\\
 \hline
\end{tabular}
\caption{Algorithms used for the Multiple Component Simulations}
\label{table:algorithms}
\end{table}

\begin{figure}
    \centering
    \includegraphics[width=0.9\linewidth]{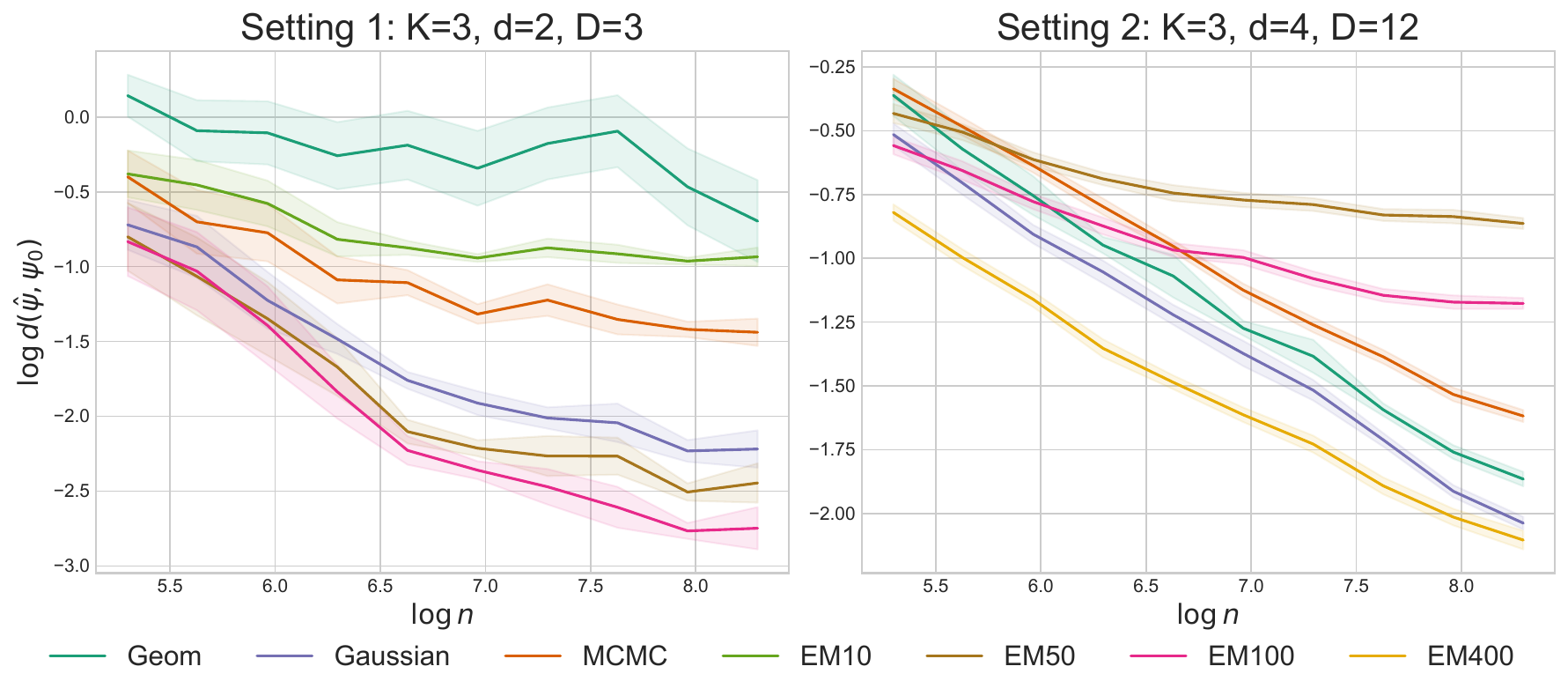}
    \caption{Simulation Results in General Setting}
    \label{fig:multiple_components_results}
\end{figure}

In our next set of simulations, we consider the general $K$ case with multiple components. We consider two settings. In each case, we consider 10 values of $n$ equally spaced (in log scale) between 200 and 4000 and consider 50 repeated experiments for setting, using 5 types of algorithms as discussed in the previous section, tabulated below. The performance of the algorithms is measured in terms of the metric $d$ defined in Equation \eqref{eq:bayes_metric}

In Setting 1, we set $K=3, d=2, D=3$ where each component is a line-segment in three-dimensions. However, the ground-truth components are chosen in a way such that there is intersection between every pair of them (see left Figure \ref{fig:multiple_components_setting}). The mixture probabilities are $0.3, 0.3, 0.4$ (roughly equal weights), with the component noise levels $\sigma_1=0.12, \sigma_2=0.2, \sigma_3=0.07$ differing across the components. The underlying mixing distribution is chosen as a symmetric Dirichlet distribution with $\alpha=1.0$.

In Setting 2, we set $K=3, D=12, d=4$ and generated the ground-truth by drawing the vertices randomly from a Gaussian distribution in $\bbR^{12}$. This setting, shown in Figure \ref{fig:multiple_components_setting} (right), involves two components having intersection, while the third component is separated from these. The mixture probabilities are generated from a symmetric Dirichlet distribution with high concentration to ensure all components have fairly high mass, with the component noise levels $\sigma_1\approx 0.1, \sigma_2\approx 0.2, \sigma_3\approx 0.6$ also generated randomly. The underlying mixing distribution is chosen as a symmetric Dirichlet distribution with $\alpha=0.75$.

The results from the experiments are shown in Figure \ref{fig:multiple_components_results} as $\log d$ versus $\log n$ plots (the lines represent the mean performance over 50 repetitions and the confidence bands illustrated as color bands). Firstly, we note that the EM algorithm for the approximate model (with sufficiently high $M$) performs the best in both settings. However, we note that with smaller values of $M$, the performance worsens as sample size increases. For Setting 2, most of the algorithms work pretty well and the convergence rate matches the theoretical parametric rate. For Setting 1, the situation is more complex, owing to several local modes in the log likelihood, with algorithms getting stuck at some of these modes (see right Figure \ref{fig:setting3_local_mode} in the Appendix). This issue has already been identified and studied in regards to Mixtures of Factor Analyzers \cite{ueda2000smem, ghahramani1999variational} with solutions involving a split-merge mechanism. We do not consider such split-merge, rather run each algorithm with multiple initializations and choose the best based on log likelihood. For comparison, the Geometric algorithm (Geom) is run using only a single initialization. We notice the substantially worse performance of this in Setting 1 (compared to the others with multiple initializations) illustrating this effect of local optima. However, in Setting 2, a single initialization works almost comparably -- this can be attributed to the fact that in higher (ambient) dimensions, the components are intuitively more separated and as such, a suitably chosen initialization most likely eliminates this issue. The MCMC algorithm does not perform well (at least in Setting 1) and has a very low acceptance rate -- this is explained by the random walk proposal, which is not optimal particularly in high dimensions, where the components are low-dimensional. This can be improved by considering proposals guided by gradients (Langevin-type proposals), which we defer to future work. We also consider another larger setting (Setting 3, given in the Appendix) with $K=10, D=200, d=25$, and the result is shown in ref, where we only used the Gaussian algorithm and the result demonstrates a parametric estimation rate again.

\subsubsection*{Model Selection}
\begin{figure}
    \centering
    \includegraphics[width=0.6\linewidth]{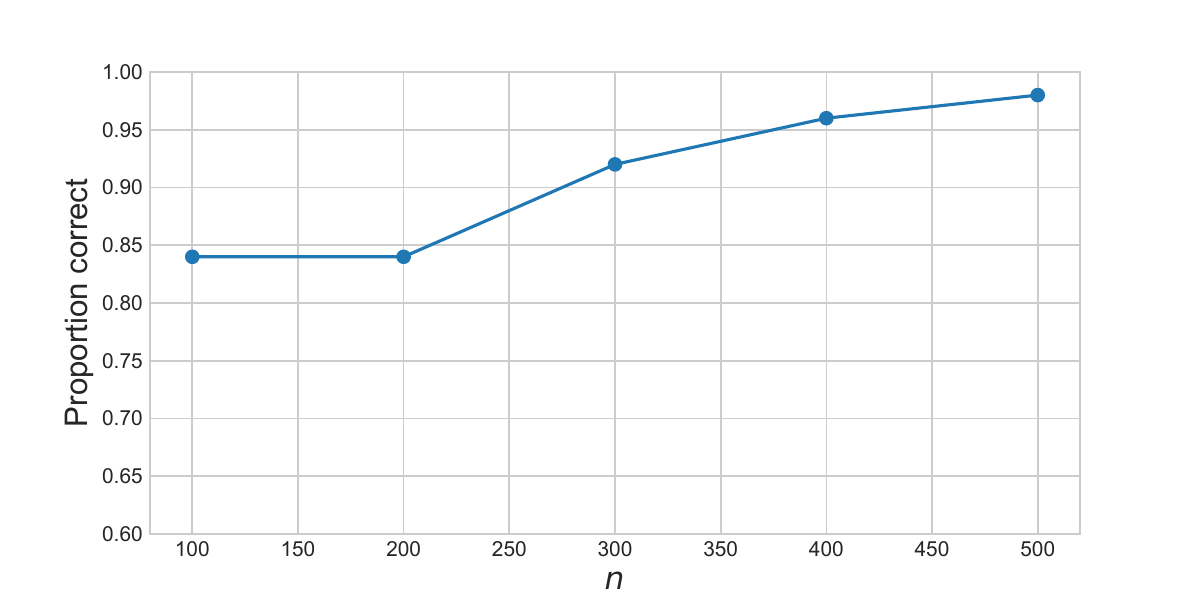}
    \caption{Results of Model Selection using BIC for Setting 2}
    \label{fig:model_selection}
\end{figure}
In the next set of simulations, we consider the task of model selection to shed light on the identifiability results in Section \ref{sec:identifiability}. We use the same setting as Setting 2 in the previous section. For each $n\in\{100,200,300,400,500\}$, we fit the model using a grid of 25 values for $K$ and $d$ (recall the true $K_*=3, d_*=4$), taking all combinations of $K\in\{K_*-2,K_*-1,K_*,K_*+1, K_*+2\}$ and similarly for $d$. In these experiments, we use the EM algorithm for the approximate model with $M=200$. We pick the model based on the lowest Bayesian Information Criteria (BIC) \cite{schwarz1978estimating} and check if it matches the ground-truth. Note that while computing BIC, the log likelihood of the model is required, which is approximated via a Monte Carlo (for the intractable integral). For each $n$, this experiment is repeated 50 times and the proportion of times BIC chooses the true model is computed. The result is shown in Figure \ref{fig:model_selection}, where we see that it indeed captures the true $K,d$.

Apart from these experiments, we also perform simulation studies to understand the performance in the case where the underlying $p_{\beta}$ is mis-specified. Interestingly, the results show that the model captures the components when $p_{\beta}$ is Dirichlet and $\alpha$ is mis-specified and even in the case, where the underlying $p_{\beta}$ is Gaussian but we fit with a Dirichlet. These results are provided in the Appendix.

\section{Conclusion}
In this paper, we consider the mixture of effectively low-dimensional convolutional measure components roughly supported on multiple affine subspaces.  In the first part of the paper, we consider a semi-parametric model where we demonstrate identifiability in very general setting, exploiting the low-dimensional nature of the latent mixing measure. Subsequently, we focus on a particular parametrization and discuss connections of our model to several other important latent variable models in the literature. We establish parametric estimation rate for such class of models by developing a novel inverse bound extending beyond finite mixtures, by again exploiting the geometric nature of such models. We believe the methods presented here deepen our understanding of latent variable models, going beyond the commonly analyzed Gaussian mixture model and can be applied more generally in situations involving continuous latent variables. We also conduct thorough simulation studies to support our theoretical findings, and in the process, develop new algorithmic tools to deal with such models.

The paper opens up a lot of avenues for future research. From a methodological perspective, the identifiability results in this paper open up the possibility of exploring non-parametric mixtures under low-dimensional latent measure structure, which can be implemented by placing a non-parametric prior (e.g., Dirichlet process mixture) on $p_{\beta}$. Theoretically, the next step involves understanding the estimation rate when the mixing distribution $p_{\beta}$ is unknown. In terms of algorithm, there are many directions one can pursue. In this paper, we discussed several algorithms but more effort is required to understand these deeply and make them scalable for large datasets. In particular, the EM algorithm for the approximate model shows potential to be useful for other cases, and can be improved by understanding how well the approximation works, choosing $M$ efficiently and also implementing the EM via mini-batches for large datasets. For the MCMC algorithm, as discussed, a better candidate could be Langevin or Hamiltonian MCMC, although care needs to be taken to accommodate the intractable likelihood (and hence the gradients), a feature quite common in recent Bayesian models.

\bibliographystyle{imsart-nameyear} 
\bibliography{ref}   

\newpage
\begin{appendix}

\section{Additional Details for Section \ref{sec:model}}

\begin{lemma}
    Let $G$ be a probability measure with support $\cS\subset\bbR^D$ such that $G$ is absolutely continuous with respect to the $d-$dimensional Hausdorff measure $\cH_d$ on $\cA=\aff \cS$, with $\dim \cA = d$. Then $\cS$ has Hausdorff dimension $d$.
\end{lemma}

\begin{proof}
   Since $\cS \subset \cA$ and $\dim(\cA)=d$, it follows that $\dim_H(\cS) \le d$, where $\dim_H$ is the Hausdorff dimension. By assumption, $G \ll \mathcal{H}_d |_{\cA}$, so there exists a Borel measurable function $f \ge 0$ on $\cA$ (density) such that
\[
G(B) = \int_{B \cap \cA} f \, d\mathcal{H}_d
\quad \text{for all Borel sets } B.
\]
Since $G$ is a probability measure,
$\int_{\cA} f \, d\mathcal{H}_d = 1$. Define the set $E := \{ x \in \cA : f(x) > 0 \}$. 
Then $\mathcal{H}_d(E) > 0$, since otherwise the above integral would be zero. Since any set with positive $\mathcal{H}^d$-measure has Hausdorff dimension at least $d$, hence $\dim_H(E) \ge d$.
Since $E \subset \cA$ and $\dim(\cA)=d$, we conclude that $\dim_H(E) = d$.

Next, we relate $E$ to the support $\cS = \supp(\mu)$. We claim that
\[
\mathcal{H}_d(E \setminus \cS) = 0.
\]
Indeed, if $x \notin \cS$, then there exists an open neighborhood $U$ of $x$ such that $G(U)=0$. Hence
\[
0 = G(U) = \int_U f \, d\mathcal{H}_d,
\]
which implies that $f=0$ $\mathcal{H}_d$-almost everywhere on $U$. Therefore $x \notin E$, up to a $\mathcal{H}_d$-null set, proving the claim.

It follows that $\mathcal{H}_d(\cS) \ge \mathcal{H}_d(E) > 0$, and hence
$\dim_H(\cS) \ge d$. Combining the bounds, we get the result.
\end{proof}

\subsection{Related Models}
\label{app:related models}
\subsubsection*{Mixed-Membership Mixtures like Topic Models}
Another interesting connection of this work is with mixed-membership mixture models, with topic models such as Latent Dirichlet Allocation \cite{blei2012probabilistic} being a prime example. In such models, it is assumed that we have access to $m$  observations $X_i=(X_{1i},\dots,X_{im})$ from each of $i=1,\dots,n$ groups, however in contrast to mixture product models, observations within a group are not assumed independent. Rather, the following model is used
\begin{equation}
    \begin{aligned}
        \beta \sim p_{\beta}, \quad X | \beta, \Theta \sim \sum_{j=1}^d \beta_j f(\cdot|\theta_j).
    \end{aligned}
\end{equation}
Thus, conditionally each group follows a mixture model with kernel $f$. When $p_{\beta}$ is Dirichlet and $f$ is multinomial (with $\theta_1,\dots,\theta_d\in\Delta^{D-1}$), the model becomes the famous Latent Dirichlet Allocation (LDA) \cite{blei2003latent}.  Note that in this case, as $m\to\infty$, by Central Limit Theorem, the conditional distribution of the mean of group $i$, $\bar{X}_i$, can be approximated by a Gaussian distribution $\bar{X}_i |\beta \simeq \cN(\Theta\beta, \Sigma(\beta)/m)$, where $(\Sigma(\beta))_{ij} = -\beta_i\beta_j$ if $i\neq j$ and $(\Sigma(\beta))_{ii} = \beta_i(1-\beta_i)$, which is a probability measure supported on $\cS=\conv(\theta_1,\dots,\theta_d)$ - this geometric connection has been exploited in \cite{nguyen2015posterior}. For large $m$, when the covariance is small, this can be approximated as $\bar{X}_i \sim \cN(\Theta\beta, \sigma^2 I)$ for small $\sigma^2$. If we only observe the group means (and not the individual observations), then $\beta\sim p_{\beta}, X|\beta,\Theta\sim \cN(\Theta\beta,\sigma^2 I_D)$ is a reasonable model, which is precisely the component distribution for our parametrization in \eqref{model:parametrization}. The mixture layer on top is to capture further heterogeneity in the data, which arises from hierarchical extensions of the topic model, e.g. \cite{griffiths2003hierarchical}, \cite{chakraborty2024learning}.

\subsubsection*{Non-Probabilistic Subspace Clustering}
The problem of subspace clustering is to cluster data lying on a union of low-dimensional subspaces. Such methods find numerous applications in image processing \cite{hong2006multiscale}, computer vision, e.g. image segmentation \cite{yang2008unsupervised}, motion segmentation and temporal video segmentation \cite{kanatani2001motion}. See \cite{parsons2004subspace},\cite{elhamifar2013sparse} for a brief review. Since data in a subspace are often distributed not elliptically around a centroid, it is important to generalize beyond methods taking advantage of spatial proximity only. There are various subspace clustering algorithms (MPPCA and MFA can be considered probabilistic models for this problem) -- iterative (such as K-subspaces \cite{tseng2000nearest} and median K-flats \cite{zhang2009median}), algebraic (such as Generalized PCA \cite{vidal2005generalized}) and spectral clustering-based methods, of which sparse subspace clustering (SSC) is an important example. SSC tries to reconstruct each observation as a combination of other points in the dataset, $x_i = \sum_{j\neq i} c_j x_j$ and posits sparsity constraints on $C$ (the matrix of the coefficients) and finally, uses spectral clustering applied to a similarity matrix obtained from $C$ \cite{elhamifar2013sparse}. An alternate way of viewing the model \eqref{model:parametrization} is that we are trying to obtain a sparse representation of the observations $x \approx \sum_k 1(z=k)\sum_j \beta_{kj} \theta_{kj}$
where each $\beta_k=(\beta_{k1},\dots,\beta_{kd})\in\Delta^{d-1}$ (convex representation not in terms of other data points, but a small number of end-point parameters) and the indicator variable ensures that only one of the $K$ components is active (sparsity). 

The use of convex polytopes as the supports for the component distributions on different subspaces provides a more interpretable representation of the clusters. Such ideas are explored recently \cite{lawless2022interpretable}, \cite{lawless2023cluster}. Thus, from a clustering point of view, the parametrized model considered in this work can be considered as a probabilistic model for the subspace clustering problem, going beyond elliptic component distributions and providing interpretable clustering results. 

\subsubsection*{Other connections}
The individual component distributions are also connected to methods in spectral unmixing, where the problem is to reconstruct a signal as a weighted sum of some end-member representative signals. There is extensive literature on this problem and notable algorithms include Non-negative Matrix Factorization methods (NMF) (specifically sparse NMF and convex NMF \cite{ding2008convex}) and archetypal analysis \cite{cutler1994archetypal}, \cite{javadi2020nonnegative}. These methods also seek for a representation $x\approx \sum_j \beta_j \theta_j$, or equivalently, a factorization of the form $X=B\Theta^\top$, where $X$ is the data matrix, $B$ is the matrix containing the weights (convex coefficients) and $\Theta$ is the $D\times d$ matrix containing the archetypes as the columns. In our case, we use a probabilistic version of this, where the rows of $B$ are considered latent variables and the mixing $p_{\beta}$ controls the distribution of the observations within the convex hull of the archetypes. In fact, this model is similar to the Dirichlet Simplex Nest considered in \cite{yurochkin2019dirichlet}, where the focus was on a scalable geometric algorithm in the case $p_{\beta}$ is a symmetric Dirichlet distribution. The current model extends such models by adding another layer of hierarchy (leading to the mixture structure).

\section{Proofs in Section \ref{sec:identifiability}}

\subsection{Proof of Proposition \ref{prop:noiseless_identifiability}}
\label{app:proof_prop:noiseless_identifiability}

We need the following simple lemma, whose proof is obvious and hence omitted.
\begin{lemma}
Suppose $\cA,\cA'$ are affine spaces and $\cS$ is low-dimensional.
\begin{enumerate}
    \item $\cS\cap\cA$ is either empty or another low-dimensional set.
    \item $\cA\cap\cA'$ is either empty or another affine space.
    \item If $\dim \cA=\dim\cA'$, then $\cA\cap\cA'$ either has strictly lower dimension or $\cA=\cA'$.
    \item If $\dim\cA<\dim\cA'$, then either $\cA\cap\cA'$ has dimension strictly less than $\dim\cA$ or $\cA\subset \cA'$.
\end{enumerate}
\end{lemma}
Now, we come to the proof of the above proposition.

\begin{proof}[Proof of Part 1]
Let $\cS_k = \supp(G_k)$ and $\cA_k = \aff \cS_k$, with $d_k = \dim \cS_k$. Define $\cS_{k'}'$ and $\cA_{k'}'$ analogously for $G_{k'}'$. By Assumption~\ref{assume:A}, the affine spaces $\{\cA_k\}_{k\in[K]}$ are pairwise distinct, and similarly for $\{\cA_{k'}'\}_{k'\in[K']}$. Recall that $G$ is continuously supported on $\cS=\supp(G)$ means that $G\ll \cH_d$ restricted to $\cA=\aff(\cS)$ with $d=\dim(\cA)$. Lastly, for a probability measure $G$ and $B\subset \bbR^D$ such that $G(B)>0$, restricted and normalized measure $G|_{B}$ is defined via $G|_{B}(A):=G(B\cap A)/G(B)$.

\medskip

\noindent
\textbf{Step 1: Equality of minimal dimensions.}
Let $d_{\min} = \min_k d_k$ and $d_{\min}' = \min_{k'} d_{k'}'$. Suppose $d_{\min} < d_{\min}'$. Let $i$ be such that $d_i = d_{\min}$. Then $P(\cS_i) \ge \pi_i > 0$. However, for every $k'$, since $d_{k'}' > d_{\min}$, the set $\cS_i \subset \cA_i$ lies in a strictly lower-dimensional affine space than $\cA_{k'}'$, so $G_{k'}'(\cS_i)=0$. Thus $P'(\cS_i)=0$, contradicting $P=P'$. Hence $d_{\min} = d_{\min}'$.

\medskip

\noindent
\textbf{Step 2: Identification of a minimal component.}
Fix $i$ such that $d_i = d_{\min}$ and let $\cA = \cA_i$. We claim that
\[
P|_{\cA} = G_i.
\]
Indeed, for any $k \neq i$:
\begin{itemize}
    \item If $d_k > d_{\min}$, then $\cA$ is strictly lower-dimensional than $\cA_k$, so $G_k(\cA)=0$.
    \item If $d_k = d_{\min}$, then by Assumption~\ref{assume:A}, $\cA_k \neq \cA$, so $\cA_k \cap \cA$ is strictly lower-dimensional, and hence $G_k(\cA)=0$.
\end{itemize}
Thus $P(\cA) = \pi_i$ and
\[
P|_{\cA} = \frac{\pi_i G_i}{\pi_i} = G_i.
\]

Applying the same argument to $P'$, there exists a unique $j$ such that $\cA_j' = \cA$. Uniqueness follows from Assumption \hyperref[assume:A]{A}. Similarly
\[
P'|_{\cA} = G_j'.
\]
Since $P=P'$, we conclude
\[
G_i = G_j', \qquad \pi_i = P(\cA) = P'(\cA) = \pi_j'.
\]

\medskip

\noindent
\textbf{Step 3: Inductive removal.}
Remove the matched components $(\pi_i, G_i)$ and $(\pi_j', G_j')$ from $P$ and $P'$ and renormalize. The remaining mixtures still satisfy Assumption~\ref{assume:A}. Repeating the above argument, we successively match and remove components.

Since both mixtures are finite and at each step one component is removed, the procedure terminates after finitely many steps. Therefore $K=K'$ and there exists a permutation $\tau$ of $[K]$ such that
\[
G_k = G_{\tau(k)}', \qquad \pi_k = \pi_{\tau(k)}'.
\]
\end{proof}

\begin{proof}[Proof of part 2]
    \noindent
    \textbf{Step 1: Lower bound of $K'$.} Let $d_{\min}=\min_k d_k$, and choose $i$ such that $d_i=d_{\min}$. Let $\cA=\cA_i$. As in the proof of part 1 above, we have $P|_{\cA}=G_i$ and $P(\cA)=\pi_i>0$. Since $P=P'$, it follows that $P'(\cA)=\pi_i>0$.

    We claim that there exists at least one $j$ such that $\cA_j'=\cA$. Indeed,suppose for contradiction that $\cA_j' \neq \cA$ for all $j$. Then for each $j$:
    \begin{itemize}
    \item either $\dim \cA_j' > d_{\min}$, in which case $\cA \subset \cA_i$ has strictly smaller dimension and hence $G_j'(\cA)=0$,
    \item or $\dim \cA_j' \le d_{\min}$ but $\cA_j' \neq \cA$, in which case $\cA_j' \cap \cA$ has strictly smaller dimension, again implying $G_j'(\cA)=0$.
\end{itemize}
Thus $P'(\cA)=0$, a contradiction. Hence there exists $j$ such that $\cA_j' = \cA$.

\medskip

\noindent
\textbf{Step 2: Identification of one component.}
Restricting to $\cA$, we obtain
\[
P|_{\cA} = G_i, \qquad P'|_{\cA} = \sum_{j: \cA_j' = \cA} \frac{\pi_j'}{P'(\cA)} G_j'.
\]
Since $P|_{\cA} = P'|_{\cA}$ and $P'(\cA)=\pi_i$, it follows that $G_i$ is a mixture of $\{G_j': \cA_j' = \cA\}$.

Since $G_i$ is supported on $\cA$ and absolutely continuous with respect to $\cH_{d_{\min}}$, this implies that there must be at least one such $j$ with $\pi_j' > 0$. In particular, at least one component of $P'$ is associated to $G_i$.
\medskip

\noindent
\textbf{Step 3: Iteration.}
Repeat the above argument for each component of $P_{\cG,\pi}$. Since the affine spaces $\{\cA_k\}_{k\in[K]}$ are distinct, the corresponding sets $\{\cA_k\}$ are disjoint up to lower-dimensional intersections, and each contributes positive mass under $P$. Thus, for each $k\in[K]$, there must exist at least one $j$ such that $\cA_j' = \cA_k$. Since distinct $\cA_k$'s cannot share the same $\cA_j'$, this implies that $K' \ge K$.

\medskip

\noindent
\textbf{Step 4: Equality case.}
If $K'=K$, then each $\cA_k$ must correspond to exactly one $\cA_j'$. Hence there exists a permutation $\tau$ of $[K]$ such that $\cA_k = \cA_{\tau(k)}'$. Restricting to each $\cA_k$ as in Part 1, we obtain
\[
G_k = G_{\tau(k)}', \qquad \pi_k = \pi_{\tau(k)}'.
\]

This completes the proof.
\end{proof}

\subsection{Proof of Lemma \ref{lemma:identify_lower_dim_abs_cont}}
\begin{proof}
    Let $P=pG + (1-p)F$ and $P'=p'G'+(1-p')F'$. Since each $\cS_k$ has dimension strictly less than $D$, it follows that $\cH_D(\cS_k)=0$ and hence, $\cL_D(\cS_k)=0$. Therefore, $\cS=\cup_k \cS_k$ satisfies $\cL_D(\cS)=0$ and thus $G$ is singular with respect to $\cL_D$. Similarly, $G'$ is singular with respect to $\cL_D$, while $F, F'$ are absolutely continuous with respect to $\cL_D$.

    Now, $P$ admits Lebesgue decomposition $P=P_{\text{sing}}+P_{\text{ac}}$, where $P_{\text{sing}}=pG$ and $P_{\text{ac}}=(1-p)F$ and similarly for $P'$. Since $P=P'$, by the uniqueness of Lebesgue decomposition with respect to $\cL_D$, we must have
    $$P_{\text{sing}}=P'_{\text{sing}}, \quad P_{\text{ac}}=P'_{\text{ac}}.$$
    Hence, $pG=p'G'$ and $(1-p)F=(1-p')F'$. If $p>0$, then $p'>0$. Since $G, G'$ are probability measures, by normalization from $pG=p'G'$, we get $p=p'$ and $G=G'$. Consequently, $(1-p)F=(1-p')F'$ implies $F=F'$. The case $p=0$ is trivial since it implies $p'=0$.
\end{proof}

\subsection{Proof of Theorem \ref{thm:isotropic_identifiability}}
\label{app:proof_thm:isotropic_identifiability}
\begin{proof}[Proof of part 1]
    Since $P_{\cG,\pi,\bphi} = P_{\cG',\pi',\bphi'}$, taking the characteristic function we get 
     $$\sum_{k\in[K]} \pi_k \varphi_{G_k}(t)\varphi_{\phi_k}(t) = \sum_{k\in[K']} \pi_k' \varphi_{G_k'}(t)\varphi_{\phi_k'}(t)\quad \forall\, t\in\bbR^D,$$
     where $\varphi_{G_k}$ is the characteristic function for $G_k$ and $\varphi_{\phi}(t)$ is the characteristic function of $q_{\phi}$. By assumption, $\varphi_{\phi_1}(t)/\varphi_{\phi_2}(t)=\psi_{\phi_1,\phi_2}(t)$ is a characteristic function of an absolutely continuous distribution on $\bbR^D$ for every $\phi_1>\phi_2$.
     
     Without loss of generality, suppose $\phi_1\leq \phi_2\leq\cdots\leq \phi_K$ and similarly $\phi_1'\leq \phi_2'\leq\cdots\leq \phi_{K'}'$. If $\phi_1<\phi_1'$, then dividing the above display by $\varphi_{\phi_1}(t)$, we obtain
     $$\sum_{k=1}^{\ell} \pi_k\varphi_{G_k}(t) + \sum_{k=\ell+1}^{K} \pi_k \varphi_{G_k}(t)\psi_{\phi_k, \phi_1}(t) = \sum_{k\in[K']} \pi_k' \varphi_{G_k'}(t)\psi_{\phi_k',\phi_1}(t)\quad \forall\, t\in\bbR^D$$
     where $\ell\in[K]$ is such that $\phi_k-\phi_1 > 0$ for all $k>\ell$, $\phi_k=\phi_1$ for $k\in[\ell]$ and by assumed ordering, $\phi_k'>\phi_1$ for all $k\in[K']$. Thus, the RHS of the above display is the characteristic function of a distribution in $\bbR^D$, which is absolutely continuous (every component is convolution with an absolutely continuous distribution), while the LHS has at least one part which is not absolutely continuous (the part corresponding to the first $\ell$ components). By the Lemma \ref{lemma:identify_lower_dim_abs_cont} and the fact that $\pi_k>0$ for all $k$, this implies that $\sum_{k\in[\ell]} \pi_k G_k$ is the zero measure, which is false. By changing roles, we can similarly disprove the case $\phi_1>\phi_1'$, which implies that $\phi_1=\phi_1'$.

     Now suppose $k=1,\dots,m$ of the components in $P$ have $\phi_1$ as the convolution kernel parameter while components $k'=1,\dots,m'$ of $P'$ also have $\phi_1$, while all other components have parameter strictly greater than $\phi_1$. Then, dividing by $\varphi_{\phi_1}$ as before, we get
     \begin{align}\label{eq:identifiability_cf}
         \sum_{k=1}^m \pi_k\varphi_{G_k}(t) + \sum_{k=m+1}^K \pi_k \varphi_{G_k}(t)\psi_{\phi_k,\phi_1}(t) = \sum_{k=1}^{m'} \pi_k'\varphi_{G_k'}(t) + \sum_{k=m'+1}^{K'} \pi_k' \varphi_{G_k'}(t)\psi_{\phi_k', \phi_1}(t)
     \end{align}
     which, transformed into space of measures, give
     \begin{align}\label{eq:identifiability_cf2}
         \left(\sum_{k=1}^m \pi_k\right)\sum_{k=1}^m \frac{\pi_k}{\sum_{k=1}^m \pi_k} G_k + \left(1-\sum_{k=1}^m \pi_k\right)F = \left(\sum_{k=1}^{m'} \pi_k'\right)\sum_{k=1}^{m'} \frac{\pi_k'}{\sum_{k=1}^{m'} \pi_k'}  G_k' + \left(1-\sum_{k=1}^{m'} \pi_k'\right)F'
     \end{align}
     where $F, F'$ are absolutely continuous measures. By assumption, all $\cS_k, \cS_k'$ have dimension less than $D$, hence by Lemma \ref{lemma:identify_lower_dim_abs_cont}, we know that $\sum_{k=1}^m \pi_k = \sum_{k=1}^{m'}\pi_k'$ and 
     $$\sum_{k=1}^m \tilde{\pi}_k G_k = \sum_{k=1}^{m'} \tilde{\pi}_k' G_k'$$
     where $\tilde{\pi}_k = \pi_k / \sum_{k=1}^m \pi_k$. Now, using Assumption \hyperref[assume:A]{A} and by our result in Proposition \ref{prop:noiseless_identifiability} (part 1), we have $m=m'$ and the components and the mixture probabilities match, up to a permutation. Finally, we also have (after eliminating these components and normalizing):
     $$\sum_{k=m+1}^K \tilde{\pi}_k \varphi_{G_k}(t)\varphi_{\phi_k}(t) = \sum_{k=m+1}^{K'} \tilde{\pi}_k' \varphi_{G_k'}(t)\varphi_{\phi_k'}(t)\quad \forall\, t\in\bbR^D$$
     where $\tilde{\pi}_k = \pi_k/(1 - \sum_{j\in[m]}\pi_j)$ and similarly for $\tilde{\pi}_k$. We can repeat the argument using the next minimum $\phi_k$. In each step, we take out equal number of components from both. Since we cannot be left out with any components on either while the other is empty, we conclude that $K=K'$. Finally, at each step, we prove for every $\phi\in \{\phi_1,\dots,\phi_K\}$, there exist the same number of components in both $P,P'$ with precisely $\phi$ as the scale parameter and these components match between $P$ and $P'$ up to a permutation. Thus, combining all these, we conclude that the proof is complete.
\end{proof}

\begin{proof}[Proof of part 2]
     As in the previous proof, let $\phi_{\min}=\min_{k\in[K]} \phi_k$, which has to match $\phi_{\min}'=\min_{k\in[K']} \phi_k'$ by the same argument as in the previous part. We group the components of $P_{\cG, \pi,\bphi}$ into (i) those with scale $\phi_{\min}$ indexed by $I$ and (ii) the rest. So, $I=\{k\in[K]: \phi_k=\phi_{\min}\}$ and $I^c=[K]\setminus I$. Similarly, define $I'$. Dividing the characteristic function identity with $\varphi_{\phi_{\min}}(t)$, we obtain the identity
    $$\sum_{k\in I} \pi_k \varphi_{G_k}(t) + \sum_{k\in I^c}\pi_k \varphi_{G_k}(t)\psi_{\phi_k,\phi_{\min}}(t) = \sum_{k\in I'} \pi_k' \varphi_{G_k'}(t) + \sum_{k\in (I')^c}\pi_k' \varphi_{G_k'}(t)\psi_{\phi_k',\phi_{\min}}(t).$$
    By Assumption \hyperref[assume:C]{C}, all the terms $\psi_{\phi_k,\phi_{\min}}$ correspond to convolution with an absolutely continuous distribution since for all $k\in I^c$, $\phi_k>\phi_{\min}$ and similarly for $(I')^c$. Hence, the decomposition splits into (i) finite mixture of low-dimensional measures, and (ii) an absolutely continuous component.

     By Lemma \ref{lemma:identify_lower_dim_abs_cont}, the singular (non absolutely continuous) parts of both sides must match. This gives
     $$\sum_{k\in I} \tilde{\pi}_k G_k = \sum_{k\in I'} \tilde{\pi}_k' G_k', \quad \tilde{\pi}_k,$$
     where $\sum_{k\in I} \pi_k = \sum_{k\in I'} \pi_k'$ (matching the total mass), and $\tilde{\pi}$ and $\tilde{\pi}'$ are the normalized weights. The equality in the above display is exactly a mixture of low-dimensional measures without noise. Hence Proposition \ref{prop:noiseless_identifiability} (part 2) applies under Assumption \hyperref[assume:A]{A}, yielding that $|I|\leq |I'|$ and if equality holds (i.e., $|I|=|I'|$), then up to a permutation $\tau$, $(\pi_k, G_k)_{k\in I} = (\pi_{\tau(k)}', G_{\tau(k)}')_{k\in I}$. 

     Once we match these components of $I$ with those of $I'$, remove these from both mixtures and renormalize the remaining parts. The residual model still satisfies the same structural assumptions, with the minimum scale now strictly larger than $\phi_{\min}$. We repeat the argument inductively. Since the mixture $P_{\cG, \pi, \bphi}$ has finitely many components, the process terminates in finitely many steps. 

     Combining all the above arguments, we have shown that for each noise level, the components with this noise either match (up to a permutation) or $P_{\cG',\pi',\bphi'}$ must have more number of components. This proves that $K'\geq K$, and in the case $K=K'$, the equality has to be true for each noise level in the inductive argument, resulting in a single permutation $\tau$ such that $(\pi_k, G_k, \phi_k) = (\pi_{\tau(k)}', G_{\tau(k)}', \phi_{\tau(k)}')$ for all $k$.
\end{proof}

\begin{remark}\label{remark: laplace}
    A note for Laplace distribution for the noise model, i.e., $q_{\phi}=\text{Laplace}(\boldsymbol{0}, \phi I)$ for $\phi>0$. We know that the characteristic function of the Laplace distribution is $\varphi_{\phi}=1/(1+\phi\norm{t}_2^2/2)$. This gives
    $$\psi_{\phi_1,\phi_2}(t)=\frac{\varphi_{\phi_1}(t)}{\varphi_{\phi_2}(t)} = \frac{1 + \frac{1}{2}\phi_2\norm{t}_2^2}{1 + \frac{1}{2}\phi_1\norm{t}_2^2} =\frac{\phi_2}{\phi_1} + \frac{1 - \phi_2/\phi_1}{1+\frac{1}{2}\phi_1\norm{t}_2^2}.$$
    Note we are interested in the case $\phi_1>\phi_2$. Taking the inverse Fourier transform, we get the distribution whose characteristic function is $\psi_{\phi_1,\phi_2}(t)$, given by
    $$P_{\phi_1,\phi_2}:=\frac{\phi_2}{\phi_1} \delta_{\boldsymbol{0}} + \left(1- \frac{\phi_2}{\phi_1}\right)\text{Laplace}(\boldsymbol{0}, \phi_1 I_D)$$
    which is a mixture of a discrete (point mass at $\boldsymbol{0}$) and Laplace distribution. This does not satisfy our assumption of $\varphi_{\phi_1}/\varphi_{\phi_2}$ being the c.f. of an absolutely continuous distribution on $\bbR^D$ -- however, this characterization can also be used with our proof technique under slight modification. Note that $\varphi_{G_k}\psi_{\phi_k,\phi_1}$ (as appears in Equation \eqref{eq:identifiability_cf}) is the c.f. of $X+Y$, where $X\sim G_k$ and $Y\sim P_{\psi_k,\psi_1}$ independently. This is the probability distribution $\left(\phi_1/\phi_k\right) G_k + \left(1 - \phi_1/\phi_k\right) G_k\star \text{Laplace}(\boldsymbol{0}, \phi_1 I_D)$
    which turns the left-hand side of Equation \eqref{eq:identifiability_cf2} into the following
    $$\underbrace{\sum_{k=1}^m \pi_kG_k + \sum_{k=m+1}^K \frac{\pi_k \phi_1}{\phi_k}G_k}_{\text{low-dim}} + \underbrace{\sum_{k=m+1}^K \pi_k\left(1-\frac{\phi_1}{\phi_k}\right)G_k\star \text{Laplace}(\boldsymbol{0}, \phi_1 I_D)}_{\text{abs. cont.}}$$
    and similarly for the right-hand side. Applying Lemma \ref{lemma:identify_lower_dim_abs_cont} gives
    \begin{align*}
    \sum_{k=1}^m \pi_kG_k + \sum_{k=m+1}^K \frac{\pi_k \phi_1}{\phi_k}G_k &= \sum_{k=1}^{K'} \frac{\pi_k'\phi_1}{\phi_{k}'}G_k' &\text{low-dim part}\\
    \sum_{k=m+1}^K  \pi_k\left(1-\frac{\phi_1}{\phi_k}\right)G_k\star \text{Laplace}(\boldsymbol{0}, \phi_1 I_D) &=\sum_{k=1}^{K'}  \pi_k'\left(1-\frac{\phi_1}{\phi_k'}\right)G_k'\star \text{Laplace}(\boldsymbol{0}, \phi_1 I_D) &\text{cont. part}
    \end{align*}
    The first equality of measures along with Proposition \ref{prop:noiseless_identifiability} shows that $K=K'$ and $G_1,\dots,G_K$ are same as $G_1',\dots,G_K'$ up to a permutation. Without loss of generality, say $G_k=G_k'$ for all $k$. Then, we also get
    $$\pi_k = \pi_k'\phi_1/\phi_k', k=1,\dots, m \text{ and } \pi_k/\phi_k = \pi_k'/\phi_k', k=m+1,\dots,K.$$
    The second equality, noting that each convolution is with the same Laplace distribution gives (convert to characteristic function and cancel off the part corresponding to this Laplace part and reverse back to distribution)
    $$\sum_{k=m+1}^K  \pi_k\left(1-\frac{\phi_1}{\phi_k}\right)G_k =\sum_{k=1}^{K'}  \pi_k'\left(1-\frac{\phi_1}{\phi_k'}\right)G_k'$$
    which is not possible (left-side has fewer components), unless $\phi_1'=\phi_1$. In that case, there must be exactly $m$ components with $\phi_1'=\dots=\phi_m'=\phi_1$. For these components, we immediately obtain $G_k=G_k'$ and $\pi_k=\pi_k'$. Then, we remove this component and proceed similar to the proof of the theorem.
\end{remark}

\subsection{Additional result with Assumption \hyperref[assume:B]{B}}
\label{app: relax assumption A}

In this section, we show that similar conclusion can be reached by using Assumption \hyperref[assume:B]{B} in place of Assumption \hyperref[assume:A]{A}. In particular, we allow two components to have supports with the same affine hulls, subject that their supports are disjoint, and each support is a connected set. 
\begin{theorem}\label{thm:isotropic_identifiability b}
Suppose $P_{\cG,\pi,\bphi} = P_{\cG',\pi',\bphi'}$ with $K, K'$ components respectively where $\pi\in\intr (\Delta^{K-1})$ and $ \pi'\in\intr (\Delta^{K'-1})$, supports of measures in $\cG$ (and also those in $\cG'$) satisfy assumption \hyperref[assume:B]{B} and have dimensions strictly lower than $D$. Assume the noise kernel family $\cQ$ satisfies Assumption \hyperref[assume:C]{C}. Then we have $K=K'$, and all parameters match up to a permutation of labels.
\end{theorem}

We use the corresponding version of Proposition \ref{prop:noiseless_identifiability}, modified for Assumption B.

    \begin{proposition}\label{prop:noiseless_identifiability b}
Suppose $\sum_{k\in[K]} \pi_k G_k = \sum_{k\in[K']} \pi_k' G_k'$ where both $\{\cS_k\}_{k\in[K]}$ and $\{\cS_{k}'\}_{k'\in[K']}$ satisfy Assumption \hyperref[assume:B]{B} and $\pi\in \intr \Delta^{K-1}$ and $\pi'\in\intr \Delta^{K'-1}$, then $K=K'$ and there exists a permutation $\tau$ of $[K]$ such that $G_k=G_{\tau(k)}'$ and $\pi_k=\pi_{\tau(k)}'$.
\end{proposition}

\begin{proof}
Based on Proposition \ref{prop:noiseless_identifiability}, we can reduce the problem to identification within each affine hull. In particular, let $\cA_\ell$ be an affine space such that at least one of the components $\cS_k$ has $\cA_\ell$ as the affine hull. Let $\cI_\ell=\{k:\aff\cS_k = \cA_\ell\}$, and similarly $\cI_\ell'$, be the indices of the components corresponding to this affine space. The problem now boils down to identifiability of the components in the mixture
$$\sum_{k\in \cI_\ell} \pi_k G_k = \sum_{k\in \cI_\ell'} \pi_k' G_k'.$$

Looking at the corresponding supports, the support of the LHS of the above display consists of $|\cI_\ell|$ connected components, while that of the RHS consists of $|\cI_\ell'|$ connected components by Assumption \hyperref[assume:B]{B}. This implies $|\cI_\ell|=|\cI_\ell'|$. By disjointness, for any $k\in\cI_\ell$, there is exactly one $k'\in\cI_\ell'$ corresponding to $G_k$, and since the probability mass on $\cS_k=\cS_{k'}'$, the weights must match. This argument gives a one-to-one matching of the components within this affine space. Combining over all the affine spaces, the proposition follows.
\end{proof}

\begin{proof}[Proof of Theorem \ref{thm:isotropic_identifiability b}]
This follows using the same argument as in the proof of Theorem \ref{thm:isotropic_identifiability}, replacing Assumption \hyperref[assume:A]{A} by Assumption \hyperref[assume:B]{B}, and using Proposition \ref{prop:noiseless_identifiability b}, in place of Proposition \ref{prop:noiseless_identifiability}. Furthermore, results corresponding to Remark \ref{remark: laplace} also extend to this setting.
\end{proof}

It is worth noting that similarly the minimal representation type result as in Theorem \ref{thm:isotropic_identifiability} (part 2) can also be stated using Assumption \hyperref[assume:B]{B}.

\section{Proofs in Section \ref{sec:posterior_contraction}}

\subsection{Proof of Proposition \ref{prop:kl_upper bound}}
\label{app:proof_prop:kl_upper bound}
\begin{proof}
    Suppose $d(\psi,\psi')=\epsilon$. This implies that (possibly after relabeling of the components) that 
    $\sum_k \left(d_M(\Theta_k, \Theta_k') + |\pi_k - \pi_k'| + |\sigma_k^2 - \sigma_k'^2|\right) = \epsilon.$
    Using the definition of the metric $d_M$ (and a potential further relabeling of the rows within $\Theta_k'$'s),
    for every $k\in[K], j\in [d]$, we have $\norm{\theta_{kj} - \theta_{kj}'}<\epsilon$. We argue that this implies $W_2(G_k, G_k') \leq \epsilon$. We have $W^2_2(G_k,G_k') = \inf_{\Pi} \int \norm{x-y}^2 d\Pi(x,y)$ where the infimum is taken over all couplings $\Pi$ of $G_k$ and $G_k'$. Construct a specific deterministic coupling: Draw $\beta\sim p_{\beta}$ and let $X=\Theta_k\beta$ and $Y=\Theta_k'\beta$ (sharing the same $\beta$), consider $\Pi$ to be the distribution of $(X,Y)$. It is clear that it is a valid coupling. This can be also viewed in terms of Monge maps: Let $T:\cS_k \to \cS_k'$ be the map which sends $\sum_j \alpha_j \theta_{kj} = x\mapsto y=\sum_j \alpha_j \theta_{kj}'$ and then $\Pi = (\text{Id}, T) \# G_k$. Under this coupling, for every $\alpha\in \Delta^{d-1}$, $\norm{x-y} = \norm{\sum_j \alpha_j \theta_{kj} - \sum_j \alpha_j \theta_{kj}'} = \norm{\sum_j \alpha_j (\theta_{kj} - \theta_j')}\leq \sum_j \alpha_j \norm{\theta_{kj} - \theta_{kj}'} < \epsilon$. Thus, for any $(x,y)\sim \Pi$, we have $\norm{x-y}^2 < \epsilon^2$, which gives that for this specific coupling $\int \norm{x-y}^2d\Pi (x,y) < \epsilon^2$, which shows that $W_2(G_k, G_k') < \epsilon$.

    Let $Q_k=\int \phi(x\mid \eta) dG_k(\eta)$ denote the $k$-th component distribution for $P_{\psi}$. We use properties of the KL-divergence to upper bound the KL between $p_{\psi}$ and $p_{\psi'}$ in terms of a convex sum of the $W_2^2$ between pairs of components. In particular, using the joint convexity of KL-divergence, for any coupling $q$ of $\pi$ and $\pi'$ and any coupling $\Pi_{k,k'}$ of $(G_k, G_{k'}')$ (for each $k,k'$) we have
    \begin{align*}
        \KL(P_{\psi}, P_{\psi'}) &= \KL\left(\sum_k \pi_k Q_k, \sum_k \pi_k' Q_{k}'\right) \\
        &\leq \sum_{k,k'} q_{k,k'} \KL(Q_k, Q_{k'}') \\
        &= \sum_{k,k'} q_{k,k'} \KL\left(\int \phi(\cdot\mid \eta,\sigma_k^2 I) dG_k(\eta), \int \phi(x\mid \eta, \sigma_{k'}'^2 I) dG_{k'}'(\eta)\right) \\
        &\leq \sum_{k,k'} q_{k,k'} \int \KL\left(\phi(\cdot|\eta,\sigma_k^2 I), \phi(\cdot|\eta',\sigma_{k'}'^2 I)\right) \Pi_{k,k'}(d\eta,d\eta') \\
        &= \frac{1}{2}\sum_{k,k'} q_{k,k'} \int \left(\frac{1}{\sigma_{k'}'^2}\norm{\eta - \eta'}^2 + D\frac{\sigma_k^2}{\sigma_{k'}'^2} - D + D(\log \sigma_{k'}'^2 - \log \sigma_k^2) \right)\Pi_{k,k'}(d\eta, d\eta') \\
        &\leq \frac{1}{2 }\sum_{k,k'} q_{k,k'} \left[\frac{1}{\sigma_{k'}'^2} W_2^2(G_k, G_{k'}') + D\left(\log \sigma_{k'}'^2 - \log \sigma_k^2 + \frac{\sigma_k^2 - \sigma_{k'}'^2}{\sigma_{k'}'^2}\right)\right]\\
        &\leq \frac{1}{2}\sum_{k,k'} q_{k,k'} \left[\frac{1}{\sigma_k'^2} W_2^2(G_k, G_{k'}') + D\left(\left|\log \sigma_{k'}'^2 - \log \sigma_k^2\right| + \frac{\left|\sigma_k^2 - \sigma_{k'}'^2\right|}{\sigma_{k'}'^2}\right)\right] \\
        &\leq \frac{1}{2 \underline{\sigma}^2}\sum_{k,k'} q_{k,k'} \left[W_2^2(G_k, G_{k'}') + D\left|\sigma_k^2 - \sigma_{k'}'^2\right|\right].
    \end{align*}
    where the third-from-last line follows because of the definition of $W_2$ (using the specific coupling $\Pi_{k,k'}$) and the prior line holds for any coupling $\Pi_{k,k'}$. The last line follows from the fact that $\log x$ on $[\underline{\sigma}^2, \infty)$ is Lipschitz continuous with Lipschitz constant $1/\underline{\sigma}^2$. Now, for every $k$, $W_2^2(G_k, G_k') <\epsilon^2$, while for $k\neq k'$, we use the trivial upper bound $W_2^2(G_k, G_{k'}') \leq \diam(\cT)^2$. For the variances, $|\sigma_k^2 - \sigma_k'^2|<\epsilon$ for every $k$ and for $k\neq k'$, $|\sigma_k^2 - \sigma_{k'}'^2| < \diam(S)$. 
    
    Choose $q_{k,k}=\min \{\pi_k, \pi_{k}'\}$. Recall, $\sum_k |\pi_k - \pi_k'|<\epsilon$. For a proper coupling, $q$ (view as a $K\times K$ matrix) satisfies $\sum_k q_{k,k'}=\pi_{k'}'$  (columns add to $\pi'$) and $\sum_{k'} q_{k,k'} = \pi_k$ (rows add to $\pi$). Now, for each $k$, either $q_{k,k}=\pi_k$ (if $\pi_k\leq \pi_k'$), else $q_{k,k}=\pi_{k}'$. In the former case, it enforces $q_{k,k'}=0$ for all $k'\neq k$ (i.e., the $k$th row is 0 apart from the diagonal) and $\sum_{j\neq k} q_{jk} = \pi_k'-\pi_k$ (i.e., the $k$th column adds to the difference aside from the diagonal entry). A similar thing occurs in the latter case. Thus, in either case, if we add the $k$th row and $k$th column (ignoring the diagonal entry), we get
    $\sum_{j\neq k} q_{kj} + \sum_{j\neq k} q_{jk} = |\pi_k - \pi_k'|$.
    Repeating for every $k$, we end up counting each non-diagonal entry twice, which gives
    $$2\sum_{k\neq k'} q_{k,k'} = \sum_k |\pi_k - \pi_k'| < \epsilon.$$
     Thus, the above bound implies that
    \begin{align*}
        \KL(P_{\psi},P_{\psi'}) &\leq \frac{1}{2 \underline{\sigma}^2}\sum_{k,k'} q_{k,k'} \left[W_2^2(G_k, G_{k'}') + D\left|\sigma_k^2 - \sigma_k'^2\right|\right] \\
        &\leq \frac{1}{2\underline{\sigma}^2}\left[\sum_k q_{k,k} \left(\epsilon^2 + D\epsilon\right) + \sum_{k\neq k'} q_{k,k'} \left(\diam(\cT)^2+D\diam(S)\right)\right] \\
        &\leq\frac{1}{2\underline{\sigma}^2} \left[\epsilon^2 + \underbrace{\left(D + \frac{\diam(\cT)^2+D\diam(S)}{2}\right)}_{C(\cT,S)}\epsilon\right]
    \end{align*}
    which concludes the proof.
\end{proof}

\subsection{Proof of Lemma \ref{lemma:identifiability_sequentially exposed}}\label{app:proof_lemma_identifiability_sequentially exposed}

\begin{proof} in Section 
\ref{sec:identifiability}, we can only look at the noiseless version and use the same technique via characteristic functions to identify the component variances. Thus, let $P=\sum_k \pi_k G_k$ corresponding to $\psi_0=(\bTheta,\pi,\Sigma)$ with $\bTheta=(\Theta_1,\dots,\Theta_K)$, where $G_k=(f_{\Theta_k})_{\#}p_{\beta}$, with $f_{\Theta_k}:\beta \mapsto \Theta_k\beta\in \cS_k$, the $k$th component polytope. We show that the only $Q=\sum_k \pi_k' G_k'$ with $K$ components so that $P=Q$ has to satisfy $(\pi_k, G_k) = (\pi_k', G_k')$ up to a permutation.

Let $\cS=\cup_k \cS_k$ be the support of $P$. Note that since the polytopes in $P$ are totally exposed, it means that all the $Kd$ end-member parameters $\theta_{k,j}$ are vertices of $\conv(\cS)$. Thus, the collection of all $\{\theta_{k,j}\}$ is identifiable from the support -- we now discuss how we can identify which of these belong to the same component. By the definition of exposed, each vertex of $\cC$ uniquely belongs to one component of $P$. Let $\theta\in \extr(\conv(\cS))$ be one of the end-members -- without loss of generality assume $\theta\in \cS_1$. Since each $\cS_k$ is compact and $\theta$ is separated from $\cup_{k\neq 1} \cS_k$, we know that for sufficiently small $\epsilon>0$, the vertex figure $B_{\epsilon}(\theta) \cap \cS \subset \cS_1\setminus\cup_{k\neq 1} \cS_k$ -- the rays of this tangent cone correspond to edges of $\cS_1$. Thus, in addition to the vertex $\theta$, we can identify these extreme rays from $\theta$, each ray giving the direction from $\theta$ to other end-members in $\cS_1$, using the polytope structure of $\cS_1$. Let $r_1,\dots,r_k$ be $k$ extreme rays identified from a small vertex cap around $\theta$. We know that there must be a end-member of $\cS_1$ along each of the rays $\{\theta + tr_j:t>0\}$. Let $v$ be the intersection of such a a ray (extended) with the $\conv(\cS)$. We claim that $v$ must be a end-member of $\cS_1$. If not (this includes both cases that $v$ is not an end-member corresponding to any component or $v$ is a end-member of a different component), then there exists a vertex $\theta^*$ of $\cS_1$ that is strictly in the line segment joining $\theta$ and $v$ (since vertices from different components must be distinct), violating the fact that all end-members of $\cS_1$ are exposed. Thus, starting from $\theta$, we can identify $k$ other vertices $\theta_1,\dots,\theta_k$ that are all end-members of $\cS_1$. We can iterate this argument for each of these $\theta_k$s, and continue identifying vertices in $\cS_1$ until we have identified all $d$ of these, since the end-members of $\cS_1$ form a connected graph via edges. Note that there might be $\theta,\theta'$ in the collection of vertices of $\conv(\cS)$ such that the entire line segment joining them is in $\cS$, but they belong to different components (see Figure \ref{fig:sequentially_exposed}(c) for an example) -- however, the additional information from a vertex cone in a small neighborhood of a vertex provides more information that just the vertex, it captures the local edge-connectivity across end-members locally within components. See Figure \ref{fig: total exposure identifiable proof} for an illustration of the construction.

Once we have identified $d$ points from the vertices of $\conv(\cS)$ $\theta_1,\dots,\theta_d$ that all belong to the same component, we effectively have identified $\cS_1$. Unlike the proof of Proposition \ref{prop:noiseless_identifiability}, we cannot simply restrict the measure $P$ to this polytope to identify the component, since different components are allowed to be on the same affine space with non-trivial overlap of the interiors (Figure \ref{fig: polytope example model} or the illustration above). Since $p_\beta$ is known, we know $G_1$ (since $G_1=(f_{\Theta_1})_{\#} p_\beta$ -- then, we can identify $\pi_1$ simply as $\lim_{\epsilon\downarrow 0} P(B_{\epsilon}(\theta_1))/G_1(B_\epsilon(\theta_1))$. 

\begin{figure}
    \centering
    \includegraphics[width=0.5\linewidth]{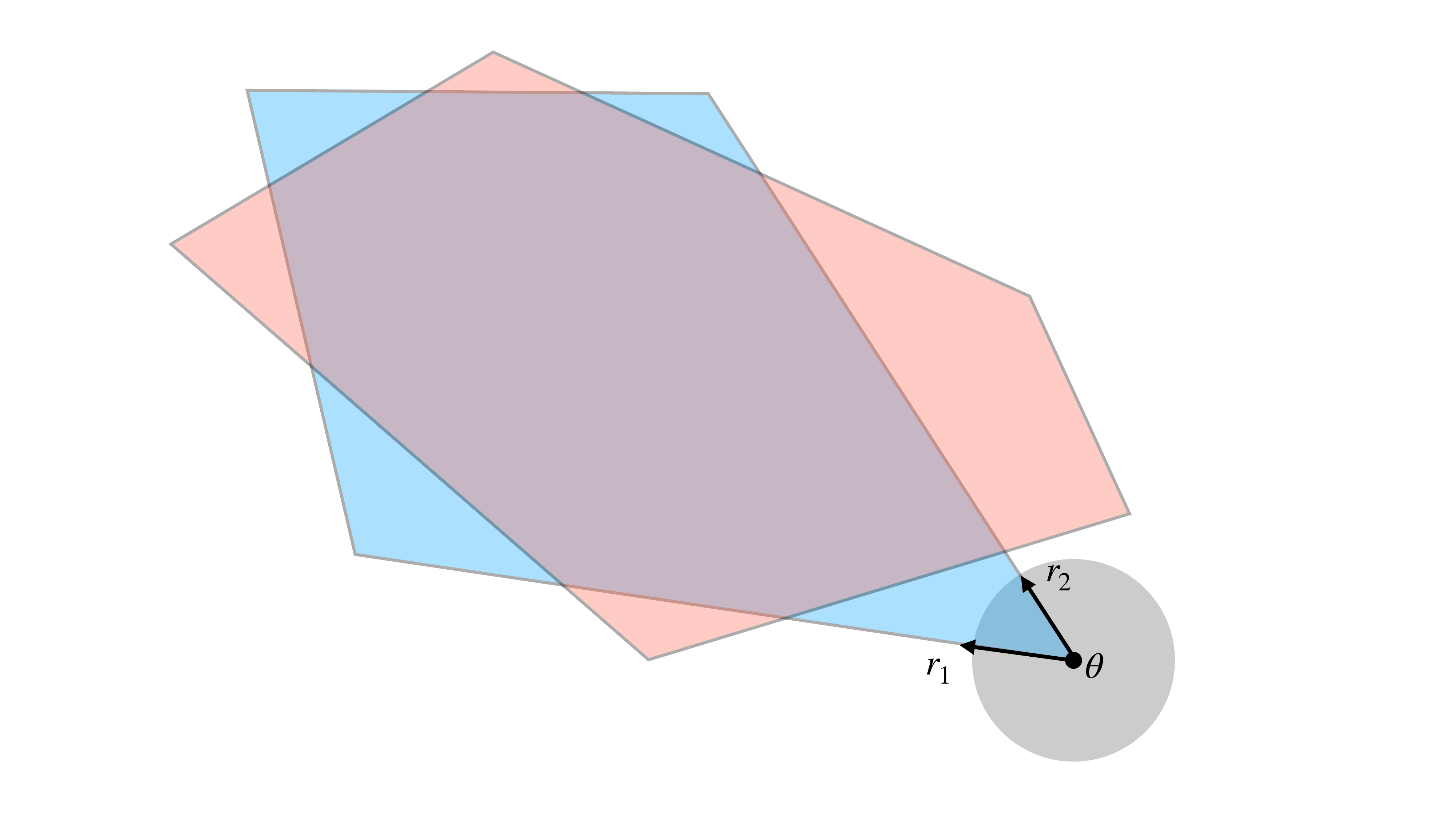}
    \caption{Illustration of the proof: Starting from $\theta$, one of the vertices of the overall convex hull, we can identify \textit{extreme rays}, $r_1, r_2$ here, emanating from $\theta$ so that vertices from the same component must lie on these rays extended. Note that the direction towards \textit{all} other vertices of the same component cannot be had from the local geometry around $\theta$ -- however, this provides a relation on the collection of $Kd$ vertices of the overall convex hull: $\theta\sim \theta'$ iff $\theta$ lies on one of the extreme rays from $\theta'$ and $\theta'$ lies on one of the extreme rays from $\theta$. If we construct a graph with the $Kd$ vertices and construct edges based on this relation, then the partition corresponding to the connected components correspond to the partition into different component polytopes.}
    \label{fig: total exposure identifiable proof}
\end{figure}
Once this component (the vertices and mixture probabilities) are recovered, remove this component, i.e., take $\tilde{P}=P - \pi_1 G_1$ (one can re-normalize to work with probability measures) and repeat the same process until all the components are recovered. Note that since $K$ is known, apart from the order of recovering the components, this process leads to a unique collection of components, precisely $\psi_0$. Any other partition of the $Kd$ vertices of $\conv(\cS)$ into $K$ components will violate the structural constraints, as discussed above.
\end{proof}

\subsection{Proof of Inverse Bound (Theorem \ref{thm:inverse_bound})}
\label{app:proof_thm:inverse_bound}

We first provide a brief overview of the proof.

\begin{proof}
We break this into smaller steps.
\textbf{Step 1: } If the result is false, there exists a sequence $\{\psi_n\}$ in $\Psi_{(\cT,S)}$ such that the following holds:
\begin{align}
    d(\psi_n,\psi_0)&\to 0 \\
    d_{\TV}(p_{\psi_n}, p_{\psi_0})/d(\psi_n,\psi_0) &\to 0
\end{align}
as $n\to\infty$. For any $n\geq 0$, let $\psi_n = (\bTheta^{(n)}, \pi^{(0)}, \Sigma^{(0)})$. Since $d(\psi_n,\psi_0)\to 0$, after a potential permutation of the components (and vertices within a component), we must have $\theta_{kj}^{(0)}\to\theta_{kj}^{(0)}$ for all $k,j$, $\pi_k^{(n)}\to\pi_k^{(0)}$ and $\sigma_k^{(n)}\to\sigma_k^{(0)}$ for all $k$. Using a subsequence, if required, suppose
\begin{align}
    \frac{\theta_{kj}^{(n)}-\theta_{kj}^{(0)}}{d(\psi_n,\psi_0)} &\to a_{kj} \in [-1,1]^D,\\
    \frac{\pi_{k}^{(n)}-\pi_{k}^{(0)}}{d(\psi_n,\psi_0)} &\to b_{k} \in [-1,1],\\
    \frac{(\sigma_{k}^2)^{(n)}-(\sigma_{k}^2)^{(0)}}{d(\psi_n,\psi_0)} &\to c_{k} \in [-1,1]
\end{align}
for all $k,j$ involved, where $\ba, \bb, \bc$ are constants. By definition of $d$, we have $\sum_{k,j}\norm{a_{kj}} + \sum_k |b_k| + \sum_k |c_k| = 1$, whence, not all of these coefficients can be 0. Note, by Fatou's lemma,
\begin{align*}
    0 &= \lim_{n\to \infty} \frac{d_{\TV}(p_{\psi_n},p_{\psi_0})}{d(\psi_n, \psi_0)} \\
    &= \lim_{n\to\infty} \int_{\bbR^D} \frac{\left|p_{\psi_n}(\bx) - p_{\psi_0}(\bx)\right|}{d(\psi_n,\psi_0)} d\bx \\
    &\geq \int_{\bbR^D} \lim_{n\to\infty} \left|\frac{p_{\psi_n}(\bx)-p_{\psi_0}(\bx)}{d(\psi_n,\psi_0)}\right| d\bx
\end{align*}
which implies that
\begin{equation}\label{eq:inv_proof1}
    \lim_{n\to\infty} \frac{p_{\psi_n}(\bx)-p_{\psi_0}(\bx)}{d(\psi_n,\psi_0)} = 0, \text{ a.s. } \bx\in\bbR^D.
\end{equation}

\vspace{1em}
\textbf{Step 2: } Note that
$$p_{\psi}(\bx) = \sum_k \pi_k \underbrace{\int \phi_D(\bx|\Theta_k\beta,\sigma_k^2I) dP(\beta)}_{f(\bx|\xi_k)}$$
where the integral is over $\Delta^{d-1}$ and we write $P_{\beta}$ as $P$ and $\xi_k=(\Theta_k,\sigma_k^2)$. We argue that differentiating the above with respect to the parameters can be done term-wise by differentiating under the integral. Note that by the Mean Value Theorem
$$f(\bx|\xi_k') - f(\bx|\xi_k) = \int \left(\phi(\bx|\Theta'\beta,\sigma'^2I) - \phi(\bx|\Theta\beta,\sigma_k^2I\right) dP(\beta) = \int (\xi'-\xi)^\top \nabla_{\xi} \phi(\bx|\tilde{\xi}(\beta)) dP(\beta)$$
for some $\tilde{\xi}(\beta)$, a convex combination of $\xi$ and $\xi'$ (depending on the choice of $\beta$). For any such choice, $\tilde{\xi}=(\tilde{\Theta},\tilde{\sigma}^2)$, where $\tilde{\Theta}\beta\in \conv(\cT)$ and $\tilde{\sigma}^2\in \conv(S)$. Furthermore, 
\begin{align*}
    \nabla_{\theta_j} \phi(\bx|\Theta\beta, \sigma^2 I) &= \frac{\beta_j(\bx - \Theta\beta)}{\sigma^2}\phi(\bx|\Theta\beta,\sigma^2I) \\
    \nabla_{\sigma^2} \phi(\bx|\Theta\beta, \sigma^2 I) &= \left(\frac{\norm{\bx - \Theta\beta}^2}{2\sigma^4} - \frac{D}{2\sigma^2}\right)\phi(\bx|\Theta\beta,\sigma^2I),
\end{align*}
which gives
\begin{align*}
    \norm{\nabla_{\theta_j} \phi(\bx|\Theta\beta, \sigma^2 I)} &\leq \beta_j \sup_{\eta\in\cT, \nu^2\in S} \norm{\bx-\eta}\phi(\bx|\eta, \nu^2 I) =: \beta_j M_{\ell}(\bx) \\
    \left|\nabla_{\sigma^2} \phi(\bx|\Theta\beta, \sigma^2 I)\right| &\leq \sup_{\eta\in\cT, \nu^2\in S} \left(\frac{\norm{\bx - \eta}^2}{2\nu^4} + \frac{D}{2\nu^2}\right)\phi(\bx|\eta,\nu^2I) =: M(\bx)
\end{align*}
where $M_{\ell}(\bx)<\infty$ and $M(\bx)<\infty$ are measureable functions and the supremums are attained in both above as $\cT,S$ are compact, and $S$ is bounded away from $0$. Moreover, $\int \beta_j M_{\ell}(\bx) dP(\beta) = M_{\ell}(\bx) \int \beta_j dP(\beta)<\infty$ (since first moment of $P_{\beta}$ exists) and similarly $\int M(\bx) dP(\beta) = M(\bx)$. Hence, by Dominated Convergence Theorem, with $\xi'\to \xi=(\Theta,\sigma^2)$,
$$\nabla_{\xi} f(\bx|\xi) = \int \nabla_{\xi} \phi(\bx|\Theta\beta,\sigma^2 I) dP(\beta).$$
Now,
\begin{align}
    p_{\psi_n}(\bx) - p_{\psi_0}(\bx) &= \sum_k \pi_k^{(n)} f(\bx|\xi_k^{(n)}) - \sum_k \pi_k^{(0)} f(\bx|\xi_k^{(0)}) \\
    &= \sum_k \left(\pi_k^{(n)} - \pi_k^{(0)}\right)f(\bx|\xi_k^{(n)}) + \sum_k \pi_k^{(n)}\left(f(\bx|\xi_k^{(n)}) - f(\bx|\xi_k^{(0)})\right).
\end{align}
Dividing by $d(\psi_n,\psi_0)$ and taking the limit $n\to\infty$, we get
\begin{equation}
    \lim_{n\to\infty} \frac{p_{\psi_n}(\bx) - p_{\psi_0}(\bx)}{d(\psi_n,\psi_0)} = \sum_k b_k f(\bx|\xi_k^{(0)}) + \sum_k \pi_k^{(0)} \left(\ba_k^\top\nabla_{\Theta}f(\bx|\xi_k^{(0)}) + c_k\nabla_{\sigma^2} f(\bx|\xi_k^{(0)})\right).
\end{equation}
From here onwards, we drop the superscript $(0)$. Plugging in the forms of the gradients, and combined with Equation \eqref{eq:inv_proof1}, we end up with the following first-order identifiability condition
\begin{equation}\label{eq:first-order identity}
\begin{aligned}
    \sum_{k\in[K]} b_k\int &\phi(\bx\mid \Theta_k^\top\beta, \sigma_k^2)dP_{\beta}(\beta) + \sum_{k\in[K]} \frac{\pi_k}{\sigma_k^2} \int \left(\sum_j \beta_j a_{kj}\right)^\top \left(\bx - \Theta_k^\top\beta\right) \phi(\bx\mid \Theta_k^\top\beta,\sigma_k^2)dP_{\beta}(\beta) \\
        &+ \sum_{k\in[K]} \pi_k \int c_k\left(\frac{\norm{\bx - \Theta_k^\top\beta}^2}{2\sigma_k^4} - \frac{D}{2\sigma_k^2}\right)\phi(\bx\mid \Theta_k^\top\beta,\sigma_k^2) dP_{\beta}(\beta) = 0 \text{ a.s. } \bx\in\bbR^D.
\end{aligned}
\end{equation}
The remainder of the proof involves showing that the coefficients $\ba, \bb, \bc$ are all 0, which would lead to a contradiction.

\vspace{1em}
\textbf{Digression: }\textbf{Reparametrization and change of coordinates:} We make two observations. For the first one, we note that if we marginalize Equation \eqref{eq:first-order identity} over all $x_2,\dots,x_D$ leaving only $x_1$ (we call this \textit{projection to the first coordinate}), we end up with
    \begin{align}\label{eq:first-order identity 1dim}
        \sum_{k\in[K]} b_k\int &\phi(x_1\mid \btheta_{k\cdot 1}^\top\beta, \sigma_k^2)dP_{\beta}(\beta) + \sum_{k\in[K]}\frac{\pi_k}{\sigma_k^2} \int \left(\sum_j \beta_j a_{kj1}\right)(x_1 - \btheta_{k\cdot 1}^\top \beta)\phi(x_1\mid \btheta_{k\cdot 1}^\top\beta, \sigma_k^2)dP_{\beta}(\beta) \nonumber\\
        &+ \sum_{k\in[K]} \pi_k \int c_k \left(\frac{(x_1 - \btheta_{k\cdot 1}^\top\beta)^2}{2\sigma_k^4} - \frac{1}{2\sigma_k^2}\right)\phi(x_1\mid \btheta_{k\cdot 1}^\top\beta, \sigma_k^2)dP_{\beta}(\beta) = 0\quad \text{a.s. } x_1\in\bbR
    \end{align}
    where $\btheta_{k\cdot j}=(\theta_{k1j},\theta_{k2j},\dots,\theta_{kdj})^\top\in\bbR^d$ is the $j$th column of $\Theta_k$. This is because $\phi(x\mid \Theta_k^\top\beta, \sigma_k^2 I) = \prod_{j=1}^D \phi(x_j\mid \btheta_{k\cdot j}^\top\beta,\sigma_k^2)$ and $\norm{x - \Theta_k^\top\beta}^2 = \sum_{j=1}^D (x_j - \btheta_{k\cdot j}^\top \beta)^2$. Furthermore, 
    \begin{align*}
        \int \phi(x|\mu,\sigma^2) dx  = 1, \quad \int (x-\mu)\phi(x|\mu,\sigma^2) = 0, \quad\int (x-\mu)^2\phi(x|\mu,\sigma^2) = \sigma^2.
    \end{align*}
    Note that Equation \eqref{eq:first-order identity 1dim} is for almost every $x_1\in \bbR$ (1-dim) and can be rewritten as (combining like terms)
    \begin{align}\label{eq:first-order-1dim rewrite}
        \int \sum_k \left[\frac{\pi_k c_k}{2\sigma_k^4} (x_1 - \btheta_{k\cdot 1}^\top\beta)^2 + \frac{\pi_k}{\sigma_k^2}(\sum_j \beta_j a_{kj1}) (x_1 - \btheta_{k\cdot 1}^\top\beta) + \left(b_k - \frac{c_k \pi_k}{2\sigma_k^2}\right)\right] &\phi(x_1\mid \btheta_{k\cdot 1}^\top\beta, \sigma_k^2)dP_{\beta}(\beta) \nonumber\\
        &=0 \quad \text{a.s. } x_1\in\bbR.
    \end{align}
    \begin{figure}
    \centering
    \includegraphics[clip, trim=0cm 2cm 1cm 8cm, width=0.9\linewidth]{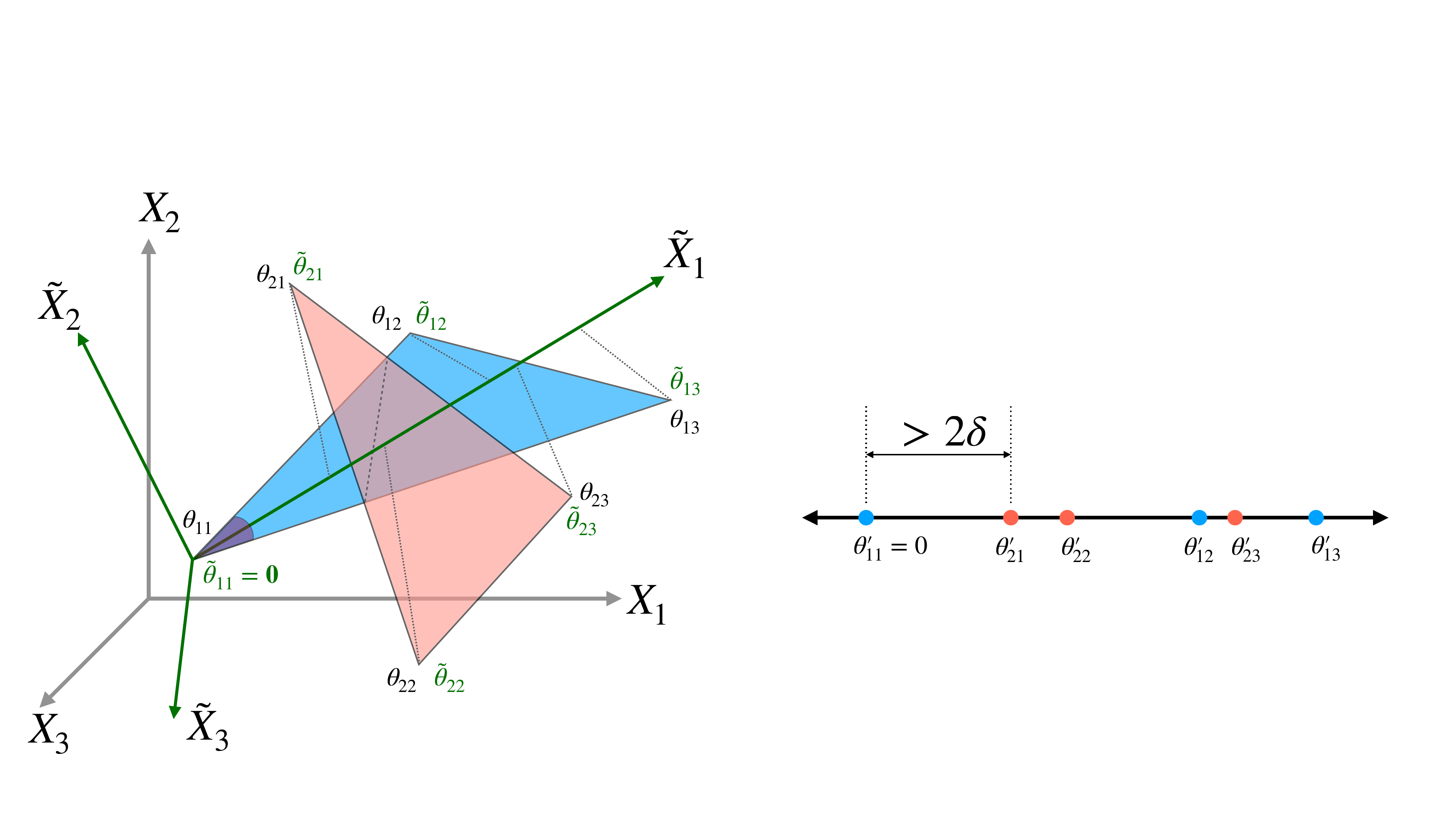}
    \caption{Extracting the effect of coefficient $a_{11}$ corresponding to vertex $\theta_{11}$ in the proof of the inverse bound (case $K=2, d=3, D=3$: first a change of coordinates from $(X_1,X_2,X_3)$ to $(\tilde{X}_1,\tilde{X}_2, \tilde{X}_3)$ using translation and orthogonal rotation only -- in this new system $\theta_{kj}$ becomes $\tilde{\theta}_{kj}$ and $a_{kj}$ becomes $\tilde{a}_{kj}$. If $a_{11}\neq\boldsymbol{0}$, such a coordinate change is possible ensuring $\tilde{a}_{111}\neq 0$. Vertex $\theta_{11}$ is an exposed point, hence the coordinate change can be made also ensuring that $\tilde{\theta}_{11}$ is origin in new system and $\tilde{\theta}_{kj1}>2\delta$ for all other vertices -- this is shown in the second figure, where $\theta_{kj}'=\tilde{\theta}_{kj1}\in\bbR$. The point of this transformation is that the effect of $a_{11}$ through the term $\sum_j \beta_j\theta_{kj}$ can be controlled by marginalizing to this $\tilde{X}_1$ axis and noting that for sufficiently small cap around $\theta_{11}$ (purple cap in left figure), the corresponding $\sum_j \beta_j \tilde{a}_{1j1}$ can be ensured to have the same sign as $\tilde{a}_{111}\neq 0$. Furthermore, all other vertices are bounded away from this $\theta_{11}'$ which allows one to extract the effect of this $\tilde{a}_{111}$ from the identity, which leads to showing that it must be 0.}
    \label{fig:enter-label}
\end{figure}
    Th next observation is regarding a change of coordinates in the ambient space $\bbR^D$. If we change the coordinates to $\tilde{\bx} = \bv + O\bx$, where $\bv\in\bbR^D$ is a translation and $O$ is orthogonal (where $\bx$ is the original coordinate), then this only changes the vertices $\bTheta$. Due to the fact that the Jacobian of this coordinate change is 1, we have $p_{\tilde{\Psi}}(\tilde{\bx}) = p_{\Psi}(\bx)$, where $\tilde{\Psi}$ is the parameter in the new system. Finally, we also have $d(\Psi_1, \Psi_2) = d(\tilde{\Psi}_1, \tilde{\Psi}_2)$ since the coordinate-change preserves Euclidean distances. Thus, in the new system, $d_{\TV}(P_{\tilde{\Psi}}, P_{\tilde{\Psi}_0})/d(\tilde{\Psi}, \tilde{\Psi}_0)$ has the same lim inf. Such a change of coordinates also preserves the quantities $\pi, \sigma^2$ and $\bb, \bc$. However, it affects $\ba$ by the rotation:
    $$\frac{\tilde{\theta}_{kj}^{(n)} - \tilde{\theta}_{kj}^{(0)}}{d(\tilde{\Psi}, \tilde{\Psi}_0)} = O \frac{\theta_{kj}^{(n)}-\theta_{kj}^{(0)}}{d(\Psi, \Psi_0)} \to O a_{kj}.$$
    Now, if $\theta_{kj}$ is an exposed vertex in the original system, there is a coordinate change as discussed before such that $\theta_{kj}$ is the origin and all other $\theta_{k'j'}'$s have their first coordinate strictly positive (in the new system). This can be seen as follows. Since $\theta_{kj}$ is an exposed vertex, there is a ray originating from $\theta_{kj}$ that goes through the interior of the convex hull of all the vertices. Treat this ray as the positive direction of the first axis in the new system with $\theta_{kj}$ as the new origin and build any orthonormal coordinate system using this specific axis. Moreover, it is easy to see that one can ensure that $\tilde{a}_{kj}$ (in the new system) is non-zero if $a_{kj}$ is non-zero. We call such a coordinate change as problem preserving.
    
    \vspace{1em}
    \textbf{Step 3: } Let the components be renamed such that 
    $$\sigma_1^2 \geq \sigma_2^2 \geq \dots \geq \sigma_K^2.$$
    We start with the first component and argue that the coefficients connected to this are all 0, in particular, $c_1=0$, $a_{1j}=0$ for all $j$ and $b_1=0$. We do this in the three sub-steps below. Note that once we are done, we can then move to the second component and repeat the argument. We only require while dealing with component $k$, that the vertices of $\cS_k$ are exposed with respect to the remaining components. 

    \vspace{0.5em}
    \textbf{Step 3(a): Coefficient $c$} If $c_1\neq 0$, let us suppose $c_1>0$. Pick any vertex of the first component, say $\theta_{11}$ and perform a problem preserving transformation to make $\theta_{11}$ the origin and project on the first coordinate to obtain (see Equation \eqref{eq:first-order-1dim rewrite})
    \begin{align}
        \sum_k \int \underbrace{\left[\frac{\pi_k c_k}{2\sigma_k^4} (x - \btheta_{k}^\top \beta)^2 + \frac{\pi_k}{\sigma_k^2}(\sum_j \beta_j a_{kj}) (x - \btheta_{k}^\top\beta) + \left(b_k - \frac{c_k \pi_k}{2\sigma_k^2}\right)\right]}_{g(x,\beta,\psi_k)}&\phi(x\mid \btheta_{k}^\top\beta, \sigma_k^2)dP_{\beta}(\beta) \nonumber\\
        &=0 \quad \text{a.s. } x\in\bbR,
    \end{align}
    where we dropped writing the $1$ in the last coordinate (to indicate the projection). Thus, in truth, related to the original parameter, $\btheta_k = (\tilde{\theta}_{k11},\tilde{\theta}_{k21},\dots,\tilde{\theta}_{kd1})$, where $\tilde{\theta}$ represents the change of coordinates. Note that by change of coordinates, in this current new system, $\theta_{11}=0$ and all other $\theta_{kj}>2\delta$ (say). Divide the above identity by $x^2 \exp\left\{-\frac{x^2}{2\sigma_1^2} + \frac{2\delta x}{2\sigma_1^2}\right\}$, we get
    \begin{equation}\label{eq:proof_inverse2}
        \sum_k \int \frac{g(x,\beta,\psi_k)}{x^2} \exp \left\{x^2\left(\frac{1}{2\sigma_1^2}-\frac{1}{2\sigma_k^2}\right) + x\left(\frac{\btheta_k^\top\beta}{2\sigma_k^2} - \frac{2\delta}{2\sigma_1^2}\right) - \frac{(\btheta_k^\top\beta)^2}{2\sigma_k^2}\right\} dP(\beta)= 0, \text{ as } x\in \bbR.
    \end{equation}
    For any choice of $\beta$, $g(x,\beta,\psi_k)/x^2 \to \pi_k c_k/2\sigma_k^2$ pointwise as $x\to -\infty$. 
    
    For all $k$, where $\sigma_k^2< \sigma_1^2$, the first term inside the exponential is negative. For all such $k$, as $x\to-\infty$, the $x^2$ term (with negative coefficient) inside the exponential dominates, and the integral goes to 0. For all other $k$ (where $\sigma_k^2=\sigma_1^2$), the integral becomes
    \begin{equation}
        \int \frac{g(x,\beta,\psi_k)}{x^2} \exp \left\{ \frac{x}{2\sigma_1^2}\left(\btheta_k^\top\beta - 2\delta\right) - \frac{(\btheta_k^\top\beta)^2}{2\sigma_k^2}\right\} dP(\beta)
    \end{equation}
    First, note that for any $k\neq 1$, $\theta_k^\top\beta > 2\delta$ for any $\beta$ (by construction). Hence, for any $k\neq 1$ with $\sigma_k^2=\sigma_1^2$, the above integral also goes to 0 as $x\to -\infty$.
    
    Finally, let us focus on $k=1$ term. Dividing the domain of integration into $\{\beta: \btheta_k^\top\beta < \delta\}$, $\{\beta: \delta < \btheta_k^\top\beta < 2\delta\}$ and $\{\beta: \btheta_k^\top\beta > 2\delta\}$, we see that on the third set, taking $x\to-\infty$, the integral goes to 0. On the second, as $x\to\infty$, the limit is positive, while on the first, the limit is $\infty$ (note that $g(x,\beta,\psi_1)/x^2 \to \pi_kc_k/2\sigma_k^2 >0$ for this component). Hence, as $x\to -\infty$, equation \eqref{eq:proof_inverse2} is violated if $c_1>0$, which forces $c_1=0$.

    \vspace{0.5em}
    \textbf{Step 3(b): Coefficient $a$} We have $c_1=0$. Let $a_{11}\neq 0$ (in original system), then since $\theta_{11}$ is exposed, there is a problem preserving transformation so that $\tilde{\theta}_{11}$ is the origin and $\tilde{a}_{111}$ is non-zero in new system. Without loss of generality, suppose $\tilde{a}_{111}>0$. Projecting to the first-coordinate (and modifying notations as before), we get
    \begin{align}\label{eq:inverse_bound_a}
        \sum_k \int \underbrace{\left[\frac{\pi_k c_k}{2\sigma_k^4} (x - \btheta_{k}^\top \beta)^2 + \frac{\pi_k}{\sigma_k^2}(\sum_j \beta_j a_{kj}) (x - \btheta_{k}^\top\beta) + \left(b_k - \frac{c_k \pi_k}{2\sigma_k^2}\right)\right]}_{g(x,\beta,\psi_k)}&\phi(x\mid \btheta_{k}^\top\beta, \sigma_k^2)dP_{\beta}(\beta) \nonumber\\
        &=0 \quad \text{a.s. } x\in\bbR
    \end{align}
    where we have $c_1=0$ and $a_{11}>0$. Choose $\delta>0$ sufficiently small, so that $\theta_{1j}>2\delta$ for all $j$ and $\theta_{kj}>2\delta$ for all $k\neq 1$. Additionally, choose $\delta$ small enough so that for any $\beta\in\Delta^{d-1}$ with $\btheta_1^\top\beta<2\delta$, $\sum_j \beta_j a_{1j} > 0$. Now, we divide by $x \exp\left\{-\frac{x^2}{2\sigma_1^2} + \frac{2\delta x}{2\sigma_1^2}\right\}$ to obtain
    \begin{equation}\label{eq:proof_inverse3}
        \sum_k \int \frac{g(x,\beta,\psi_k)}{x} \exp \left\{x^2\left(\frac{1}{2\sigma_1^2}-\frac{1}{2\sigma_k^2}\right) + x\left(\frac{\btheta_k^\top\beta}{2\sigma_k^2} - \frac{2\delta}{2\sigma_1^2}\right) - \frac{(\btheta_k^\top\beta)^2}{2\sigma_k^2}\right\} dP(\beta)= 0, \text{ as } x\in \bbR.
    \end{equation}
    As before, we evaluate the limit as $x\to - \infty$. For all $k>1$, either $\sigma_k^2<\sigma_1^2$ (in which case, the $x^2$ term in the exponent dominates with a negative coefficient) or $\sigma_1^2=\sigma_k^2$ (in which case, since all $\theta_{kj}>2\delta$ for $k\neq 1$, $\theta_k^\top\beta > 2\delta$ for all $\beta$, in which case the leading term in the exponent is $x$ with a positive coefficient and hence in either cases, $x\to - \infty$ implies that the integrals vanish. So, now let us focus on $k=1$. This term now looks like
    $$\int \left[\frac{\pi_1}{\sigma_1^2}\left(\sum_j \beta_j a_{1j}\right)\left(1 - \frac{\btheta_k^\top\beta}{x}\right) + \frac{b_1}{x}\right] e^{\frac{x}{2\sigma_1^2}\left(\btheta_1^\top \beta - 2\delta\right) - \frac{(\btheta_1^\top\beta)^2}{2\sigma_1^2}}dP(\beta).$$
    Again dividing the domain of integration into three parts $\{\beta: \btheta_1^\top\beta < \delta\}$, $\{\beta: \delta < \btheta_1^\top\beta < 2\delta\}$ and $\{\beta: \btheta_1^\top\beta > 2\delta\}$, the integral over the last part goes to 0 as $x\to - \infty$ using the same argument (since $\theta_{12},\dots,\theta_{1d}>2\delta$). Now, for the region $\{\beta: \delta <\btheta_1^\top\beta<2\delta\}$, by construction $\sum_j \beta_ja_{1j}>0$ and hence, the entire part in the braces in the integrand is positive (note that this part goes to $\pi_1\sum_j \beta_ja_{1j}/\sigma_1^2$ pointwise as $x\to-\infty$), and as a result, the whole integral is positive in this domain. Lastly, for $\{\beta: \btheta_1^\top\beta <\delta\}$, the part in the braces again is positive, but the term in the exponential now explodes since $\btheta_1^\top\beta - 2\delta < -\delta$, hence $e^{x(\btheta_1^\top\beta - 2\delta)} > e^{-x\delta}$ for $x<0$. Thus, taking $x\to\infty$, ensures that the integral over this region goes to $\infty$. Combining all the parts, we see that the integral for $k=1$ goes to $\infty$ as $x\to-\infty$, which contradicts equation \eqref{eq:inverse_bound_a}. Hence, we conclude that $a_{11}=\boldsymbol{0}$ in the original system. We can repeat the same argument for all other $j=2,\dots,d$, to conclude that $a_{11},\dots,a_{1d}=\boldsymbol{0}$.

    \vspace{0.5em}
    \textbf{Step 3(c): Coefficient $b$} Now, we look at $b_1$, noting that $c_1=0$ and all $a_{1j}=\boldsymbol{0}$. If $b_1\neq 0$, we can follow the same line of argument as before - this time, instead, dividing by $\exp\left\{-\frac{x^2}{2\sigma_1^2} + \frac{2\delta x}{2\sigma_1^2}\right\}$. All terms $k>1$ can be dealt with similarly. For $k=1$, we do not have the $c$ and $\ba$ terms and a similar argument as before leads to a contradiction, forcing $b_1=0$.
\end{proof}

\subsection{Proof of Theorem \ref{thm:density_contraction}}\label{app:proof_contraction_rate}
We essentially use Theorem 8.11 from \cite{ghosal2017fundamentals}. We discuss the proof in several steps. Based on Proposition \ref{prop:kl_upper bound}, we know that for any $\psi, \psi'$, $\KL(P_{\psi}\Vert P_{\psi'}) \lesssim \epsilon^2 + C \epsilon$, where $\epsilon=d(\psi,\psi')$. Hence, for sufficiently small $\epsilon$, we have
\begin{equation}\label{eq:hellinger_bound_d}
h^2(p_{\psi}, p_{\psi'}) \leq \frac{1}{2}\KL(P_{\psi}\Vert P_{\psi'}) \lesssim \epsilon^2+C\epsilon.
\end{equation}

\begin{proof}
    We split the proof into parts for clarity.
    
\vspace{1em}
\textbf{Part 1: Control on $K$}
 Recall the model class is $\cP=\{p_{\psi}: \psi:\Psi_{(\cT,S)}\}$ and the prior $\Pi$ on $\Psi_{(\cT,S)}$ induces a prior $\tilde{\Pi}$ on $\cP$. Fix a parameter $\psi_0\in\Psi_{(\cT,S)}$ satisfying the conditions in the statement of the theorem. We need to control the prior mass on the set $B_0(\psi_0,\epsilon)=\{p\in \cP : \KL(p_{\psi_0}\Vert p)\leq \epsilon^2)$.
Consider $\cA = \{\psi\in \Psi_{(\cT,S)}: \norm{\theta_{kj}-\theta_{kj}^{(0)}}<\epsilon, |\pi_k - \pi_k^{(0)}|<\epsilon, |\sigma_k^2 - (\sigma_k^{(0)})^2|<\epsilon \text{ for all } k,j\}$. By condition on the prior, $\Pi(\cA)\geq \tilde{c}\epsilon^{KdD+2K-1}$. Now, by definition of $d$ for any $\psi\in \cA$,
$$d(\psi,\psi_0) \leq C_1 \epsilon \Rightarrow \KL(P_{\psi_0}\Vert P_{\psi}) \lesssim \epsilon^2 + C\epsilon$$
where $C_1=Kd+2K$ and $C$ is another constant. Thus, $\cA\subset \{\psi:\KL(P_{\psi_0}\Vert P_{\psi})\lesssim \epsilon^2 + C\epsilon\}$. Thus,
$$\Pi(B_0(\psi_0, \epsilon^2+C\epsilon)) \gtrsim \epsilon^{KdD+2K-1}.$$

\vspace{1em}
\textbf{Part 2: Model Entropy} It is easy to check that
$$N(\epsilon, \Psi_{(\cT,S)}, d) \lesssim N(\epsilon, \cT, \norm{\cdot})^{Kd}\times N(\epsilon,S, |\cdot|)^K\times N(\epsilon, \Delta^{K-1}, \norm{\cdot}_1) \lesssim \epsilon^{-(KdD + 2K-1)},$$
assuming $\cT$ has dimension same as the ambient space i.e. $D$. Now, let $\psi_1,\dots,\psi_H$ be an optimal cover for $\Psi_{(\cT,S)}$ of radius $\epsilon$ in the metric $d$, i.e. $H=N(\epsilon, \Psi_{(\cT,S)}, d)$. We argue that $p_{\psi_1},\dots,p_{\psi_H}$ is a cover for $\cP$ in the metric $h$ of appropriate radius. Indeed, given any $p_{\psi}\in \cP$, there exists a $\psi_h$ in the original cover so that $d(\psi,\psi_h)<\epsilon$. This implies by Equation \eqref{eq:hellinger_bound_d} that $h^2(p_{\psi}, p_{\psi_h}) \leq C(\epsilon^2 + \epsilon)=\delta$, for some constant $C$. Thus, $\{p_{\psi_1},\dots,p_{\psi_H}\}$ form a $\delta^2-$cover for $\cP$ in the Hellinger metric. So,
$$N(\sqrt{\delta}, \cP, h) \leq N(\epsilon, P_{(\cT,S)}, d) \Rightarrow \log N(\delta, \cP, h) \lesssim \log(1/\epsilon)$$

\vspace{1em}
\textbf{Part 3: Prior Mass around $\psi_0$}  Thanks to Equation \eqref{eq:lower_bound_tv}, for all $\psi\in\Psi_{(\cT,S)}$, we have
\begin{align}
    \sqrt{2}h(p_{\psi},p_{\psi_0}) \geq d_{\TV}(p_{\psi},p_{\psi_0}) \geq C(\psi_0)d(\psi,\psi_0).
\end{align}
Hence,
\begin{align}
    \left\{\psi\in\Psi_{(\cT,S)}: h(p_{\psi},p_{\psi_0}) \leq \epsilon\right\} \subset \left\{\psi\in\Psi_{(\cT,S)}: d(\psi,\psi_0) \leq \frac{\epsilon}{C(\psi_0)}\right\}
\end{align}
Thus, for any $j$, 
\begin{align}
    \Pi\left(h(p_{\psi},p_{\psi_0}) \leq 2j\epsilon\right) &\leq \Pi\left(d(\psi,\psi_0) \leq \frac{2j\epsilon}{C(\psi_0)}\right) \\
    &\lesssim (j\epsilon)^{KDd+2K-1},
\end{align}
via the condition on the prior $\Pi$.

\vspace{1em}
\textbf{Part 4: Density Contraction} We verify the conditions of Theorem 8.11 \cite{ghosal2017fundamentals} using the results we established above. In particular, we show with $\epsilon_n=M\sqrt{\log(n) / n}$ (for sufficiently large $M$), the following holds
\begin{align}
    \frac{\Pi(j\epsilon_n\leq h(p_{\psi},p_{\psi_0})\leq 2j\epsilon_n]}{\Pi(B_0(\psi_0,\epsilon_n))} &\lesssim \exp\left(c_1 jn\epsilon_n^2\right) \\
    \sup_{\epsilon\geq \epsilon_n} \log N\left(\xi\epsilon, \{p\in\cP: h(p,p_{\psi_0})\leq 2\epsilon\}, h\right)&\lesssim n\epsilon_n^2.
\end{align}
For the first, we use the results from part (1) and part (3) above. In particular, by part (3) we have
$$\Pi(j\epsilon\leq h(p_{\psi},p_{\psi_0})\leq 2j\epsilon) \leq \Pi(h(p_{\psi},p_{\psi_0})\leq 2j\epsilon) \lesssim (j\epsilon)^{p_*}$$
where $p_*=KDd+2K-1$, the number of free parameters in the model. Also, part (1) shows that
$$\Pi(B_0(\psi_0, \epsilon))\gtrsim \epsilon^{2p_*}$$
when $\epsilon$ is sufficiently small (note that $\epsilon_n\downarrow 0$, hence for sufficiently large $n$) since the dominant term in the upper bound of the KL is the linear term. Note that $n\epsilon_n^2 = M^2 \log (n)$. Combining the two displays above, we get
$$\frac{\Pi(j\epsilon_n\leq h(p_{\psi},p_{\psi_0})\leq 2j\epsilon_n]}{\Pi(B_0(\psi_0,\epsilon_n))} \lesssim j^{p_*}\epsilon_n^{-p_*} \lesssim \exp(c_1 j n \log(n))$$
for all $n\geq n_0$ for some sufficiently large $n_0$.

For the second condition regarding the local entropy, we can use part (2) to get
\begin{align*}
    \sup_{\epsilon\geq \epsilon_n} \log N\left(\xi\epsilon, \{p\in\cP: h(p,p_{\psi_0})\leq 2\epsilon\}, h\right)&\leq \log N(\xi\epsilon_n, \cP, h) \lesssim n\epsilon_n^2
\end{align*}
since $\log(1/\epsilon_n) = \log (1/M) + \frac{1}{2}\log(n/\log(n)) \leq M^2\log n =n\epsilon_n^2$ again for all large enough $n$ so that the dominant term in the upper bound of $h$ is the linear term. Thus, both the conditions of the stated theorem are verified and the result follows.

\vspace{1em}
\textbf{Part 5: Parameter Contraction} Thanks to the previous part, we have established
$$\Pi\left(h(p_{\psi}, p_{\psi_0})\geq M\epsilon_n \mid X_1,\dots,X_n\right)\to 0  \text{ in } P_{\psi_0}^{\infty}-\text{probability}\text{ as } n\to \infty.$$
Based on part (3) in the proof, $\sqrt{2}h(p_{\psi},p_{\psi_0}) \geq C(\psi_0) d(\psi,\psi_0)$ for some constant $C$ depending only on $\psi_0$, which gives
$$\{\psi: d(\psi,\psi_0)\geq M'\epsilon_n\} \subset \{h(p_{\psi},p_{\psi_0})\geq M\epsilon_n\}$$
with $M'=\sqrt{2}M / C(\psi_0)$. The posterior probability of the right-side of the above display goes to 0, from part (4), hence we conclude that the posterior probability of the left-hand-side also goes to 0, i.e.
$$\Pi\left(d(\psi,\psi_0)\geq M'\epsilon_n \mid X_1,\dots,X_n\right)\to 0  \text{ in } P_{\psi_0}^{\infty}-\text{probability}\text{ as } n\to \infty.$$
\end{proof}

\section{Algorithms}

\subsection{Single Component}
\subsubsection{Spectral Algorithm}
\label{app:spectral_algorithm}
We first recall the moments of Dirichlet distribution. Let $P\equiv\text{Dir}(\balpha)$, where $\balpha\in\bbR_+^d$. Define $\overline{\alpha}=\sum_j \alpha_j$ and $\tilde{\alpha}=\balpha / \overline{\alpha} \in\Delta^{d-1}$. Let $\beta\sim P$, then the moment tensors up to order 3 are given by

\begin{align*}
    \bM^{(1)} = \bbE[\beta] &\Rightarrow M^{(1)}_j = \bbE[\beta_j] = \tilde{\alpha}_j \\
    \bM^{(2)} = \bbE[\beta\beta^\top] &\Rightarrow M^{(2)}_{j_1,j_2} = \bbE[\beta_{j_1}\beta_{j_2}] = \frac{\delta_{ij}\tilde{\alpha}_i + \overline{\alpha}\tilde{\alpha_i}\tilde{\alpha}_j}{\overline{\alpha}+1} \\
    \bM^{(3)} = \bbE[\beta\beta^\top\otimes \beta] &\Rightarrow M^{(3)}_{j_1,j_2,j_3} = \bbE[\beta_{j_1}\beta_{j_2}\beta_{j_3}]
\end{align*}
Note that the second moment can be written as
\begin{equation}\label{eq:2nd_moment_decomposition}
    \bM^{(2)} = \frac{1}{\bar{\alpha}+1}\left[\diag(\tilde{\alpha}) + \overline{\alpha} \bM^{(1)}(\bM^{(1)})^\top\right] 
\end{equation}
Using the moment decomposition in equation (10) \citep{do2025dirichlet}, we have the following form for $\bM^{(3)}$:
\begin{equation}
   \overline{\alpha}^{[3]} M^{(3)}_{j_1,j_2,j_3} = 2\overline{\alpha} \tilde{\alpha}_{j_1} \delta_{j_1,j_2,j_3} + \overline{\alpha}^{[2]}\overline{\alpha} \left(M^{(1)}_{j_1} M^{(2)}_{j_2,j_3}+M^{(1)}_{j_2} M^{(2)}_{j_3,j_1}+M^{(1)}_{j_3} M^{(2)}_{j_1,j_2}\right) - 2\overline{\alpha}^3 M^{(1)}_{j_1}M^{(1)}_{j_2}M^{(1)}_{j_3}.
\end{equation}
This can be expressed as a matrix, taking the action of $\bM^{(3)}$ on a vector $w\in \bbR^d$, i.e., 
\begin{equation}
    \begin{aligned}\label{eq:third_moment_dir}
        \bM^{(3)}(w) &= \bbE[\beta\beta^\top \innerprod{\beta,w}] \\
        &=\frac{2}{\bar{\alpha}(\bar{\alpha}+1)(\bar{\alpha}+2)} \diag(\alpha)\diag(w) - \frac{2\bar{\alpha}^2}{(\bar{\alpha}+1)(\bar{\alpha}+2)}\innerprod{\bM^{(1)},w} \bM^{(1)}(\bM^{(1)})^\top \\
        &\hspace{5em} + \frac{\bar{\alpha}}{\bar{\alpha}+2}\left[\bM^{(2)}w(\bM^{(1)})^\top + \bM^{(1)}(\bM^{(2)}w)^\top + \innerprod{\bM^{(1)}, w} \bM^{(2)}\right].
    \end{aligned}
\end{equation}

Let us now look at the moment tensors of a Gaussian distribution. Let $Y\sim \cN_D(\mu,\Sigma)$ with components $Y=(Y_1,\dots,Y_D)$, then we have
\begin{align*}
    \bbE[Y] = \mu &\Rightarrow \bbE[Y_i] = \mu_i\\
    \bbE[YY^\top] = \Sigma + \mu\mu^\top &\Rightarrow \bbE[Y_{i_1}X_{i_2}] = \Sigma_{i_1,i_2} + \mu_{i_1}\mu_{i_2}
\end{align*}
and for the third moment, we know that odd moments of centered Gaussian are 0. Thus, starting with $\bbE[(Y_i- \mu_i)(Y_j-\mu_j)(Y_j-\mu_k)$, little algebra yields
\begin{equation}
    \bbE[Y_{i_1}Y_{i_2}Y_{i_3}] = \left(\mu_{i_1}\Sigma_{i_2 , i_3} + \mu_{i_2}\Sigma_{i_3 , i_1} + \mu_{i_3}\Sigma_{i_1 , i_2}\right) + \mu_{i_1}\mu_{i_2}\mu_{i_3}.
\end{equation}
This can be expressed succintly as the action of the moment tensor on a vector. Given $v\in\bbR^D$, we are interested in $\bbE[YY^\top\innerprod{Y,v}]$ (which is a $D\times D$ matrix). Using the above expression, this has the following form
\begin{equation}
    \bbE[YY^\top\innerprod{Y,v}] = \Sigma v \mu^\top + \mu v^\top \Sigma + \left(\Sigma + \mu\mu^\top\right)\mu^\top v.
\end{equation}

Now, let us start discussing the moments of our model. Under the model, with $p_{\beta}$ is Dirichlet with parameter $\balpha$, we have $X|\beta \sim \cN(\Theta\beta, \sigma^2 I)$. Thus, we have
\begin{align}
    \bbE[X] &= \Theta \bM^{(1)} \\
    \bbE[XX^\top] &= \Theta \bM^{(2)}\Theta^\top + \sigma^2 I \\
    &= \frac{1}{\overline{\alpha}+1}\Theta\diag(\tilde{\alpha})\Theta^\top + \frac{\bar{\alpha}}{\bar{\alpha}+1}(\Theta\bM^{(1)})(\Theta \bM^{(1)})^\top + \sigma^2 I \nonumber\\
    &\Rightarrow \Theta \diag(\tilde{\alpha})\Theta^\top = (\overline{\alpha}+1)(\bbE[XX^\top]-\sigma^2 I) - \bar{\alpha} \bbE[X]\bbE[X]^\top. \label{eq:second_moment_dir}
\end{align}
where the second moment used the decomposition \eqref{eq:2nd_moment_decomposition}. For the third moment, we can again look at its action on a vector $v\in \bbR^D$, i.e., $\bbE[XX^\top \innerprod{X,v}]$. Conditionally, we have
\begin{align*}
    &\bbE[XX^\top\innerprod{X,v}|\beta] = \sigma^2 \left( v \beta^\top \Theta^\top + \Theta\beta v^\top\right) + \left(\Theta\beta\beta^\top\Theta^\top + \sigma^2 I_D\right) \beta^\top\Theta^\top v \\
    &\Rightarrow \bbE[XX^\top \innerprod{X,v}] = \sigma^2\left(v \bbE[X]^\top + \bbE[X]v^\top\right) + \sigma^2 \bbE[X]^\top v I_D + \Theta \bbE[\beta\beta^\top \innerprod{\beta, \Theta^\top v}]\Theta^\top.
\end{align*}
 To compute the expectation involving the third moment of $\beta$ above, we can use equation \eqref{eq:third_moment_dir}, with $w=\Theta^\top v\in\bbR^d$, to get
 \begin{align*}
     \bbE[XX^\top\innerprod{X,v}] &= \sigma^2\left(v (\bM^{(1)})^\top + \bM^{(1)}v^\top + \innerprod{\bM^{(1)}, v} I_D\right) \\
     & \quad+ \frac{2}{\bar{\alpha}(\bar{\alpha}+1)(\bar{\alpha}+2)}\Theta\left(\diag(\Theta^\top v)\diag(\alpha)\right)\Theta^\top \\
     & \quad+\frac{\bar{\alpha}}{\bar{\alpha}+2}\Theta \bM^{(2)}\Theta^\top (\Theta\bM^{(1)})^\top v - \frac{2\bar{\alpha}^2}{(\bar{\alpha}+1)(\bar{\alpha}+2)} \Theta \bM^{(1)} (\Theta\bM^{(1)})^\top (\Theta \bM^{(1)})^\top v \\
     & \quad+\frac{\bar{\alpha}}{\bar{\alpha}+2}\Theta\bM^{(2)}\Theta^\top v (\Theta\bM^{(1)})^\top + \frac{\bar{\alpha}}{\bar{\alpha}+2}\Theta \bM^{(1)} v^\top \Theta \bM^{(2)}\Theta^\top.
\end{align*}

Define
\begin{align*}
    \mu &= \bbE[X]=\Theta \bM^{(1)} \\
    S =& \bbE[XX^\top]-\sigma^2 I = \Theta\bM^{(2)}\Theta^\top,
\end{align*} we can express the above as
\begin{align*}
\bbE[XX^\top\innerprod{X,v}] &= \sigma^2\left(v \mu^\top + \mu v^\top + \innerprod{\mu, v}  I_D\right) + \frac{2}{\bar{\alpha}(\bar{\alpha}+1)(\bar{\alpha}+2)}\Theta\left(\diag(\Theta^\top v)\diag(\alpha)\right)\Theta^\top \\
  & \quad - \frac{2\bar{\alpha}^2}{(\bar{\alpha}+1)(\bar{\alpha}+2)} \mu\mu^\top \innerprod{\mu,v} + \frac{\bar{\alpha}}{\bar{\alpha}+2} \left(\innerprod{\mu,v} S  + S v\mu^\top + \mu v^\top S\right)
 \end{align*}

Thus, letting
\begin{align}
    \text{Pairs} &:= S - \frac{\bar{\alpha}}{\bar{\alpha}+1} \mu\mu^\top, \label{eq:pairs}\\
    \text{Triples}(v) &:= \bbE[XX^\top\innerprod{X,v}] - \sigma^2\left(v \mu^\top + \mu v^\top + \innerprod{\mu, v}  I_D\right) \label{eq:triples} \\
    &\quad -\frac{\bar{\alpha}}{\bar{\alpha}+2}\left(\innerprod{\mu,v} S  + S v\mu^\top + \mu v^\top S\right) + \frac{2\bar{\alpha}^2}{(\bar{\alpha}+1)(\bar{\alpha}+2)} \innerprod{\mu,v} \mu\mu^\top. \nonumber
\end{align}

Note that the above definitions yield the following diagonal structures
\begin{align*}
    \text{Pairs} &= \frac{1}{\bar{\alpha}(\bar{\alpha}+1)}\Theta \diag(\alpha) \Theta^\top \\
    \text{Triples}(v) &= \frac{2}{\bar{\alpha}(\bar{\alpha}+1)(\bar{\alpha}+2)}\Theta \diag(\Theta^\top v)\diag(\alpha)\Theta^\top,
\end{align*}
which exactly matches the pairs and triples in Lemma 4.3 \cite{anandkumar2012spectral}. Hence, we can now use the same technique in Algorithm 3 \cite{anandkumar2012spectral}, or the empirical version (Algorithm 5) using the empirical moments to estimate $\mu, \text{Pairs}$ and $\text{Triples}$. Note that the forms of the pairs and triples is different from \cite{anandkumar2012spectral}, owing to the Gaussian kernel considered here. Since $\Theta\bM^{(2)}\Theta^\top$ has rank $d<D$, we can estimate $\sigma^2$ directly, by using a SVD on the data matrix directly and use this estimate in all the steps of the algorithm.

\subsubsection{Approximate Algorithm using Gaussian Approximation}
\label{app:approx_gaussian_algorithm}
We employ the following heuristic algorithm. By approximating $p_{\beta}$ with a Gaussian distribution (matching the first two moments), i.e. $p_{\beta}\approx \cN(\mu_0,\Sigma_0)$, the resulting marginal distribution for the data becomes $\cN(\Theta \mu_0, \Theta\Sigma_0\Theta ^\top+\sigma^2 I)$. By estimating $\sigma^2$ using a SVD on the data matrix, we can plug it in for moment matching to get the following equations:
\begin{align*}
    \Theta\mu_0 &= \bar{X} \\
    \Theta\Sigma_0\Theta^\top &= \Cov(X) - \hat{\sigma}^2 I = USU^\top \text{ (eig)}.
\end{align*}
Let $A=\tilde{U}\tilde{S}^{1/2}$ taking the top $d-1$ eigenvalues and the corresponding eigenvectors from the above decomposition. From the above (note that $\Sigma_0$ has rank $d-1$), we can solve
$$\tilde{\Theta} \begin{bmatrix} A_0 &  \mu_0\end{bmatrix} = \begin{bmatrix} A & \bar{X}\end{bmatrix}$$
where $A_0=\tilde{V}\tilde{S}_2^{1/2}$ corresponding to the top $d-1$ eigenvalues and eigenvectors of $\Sigma_0=VS_2V^\top$. The above display is solved for $\tilde{\Theta}$. Let $\tilde{\theta}_1,\dots,\tilde{\theta}_d$ denote the columns. However, to resolve the rotational invariance for this Gaussian approximation, we employ the following heuristic correction. Center the data by subtracting the mean and let $w_1,\dots,w_d$ be the cluster centers for a K-Means clustering (with $d$ clusters) applied to the centered data, and denote $W$ to be the $D\times d$ matrix containing these as the columns. We minimize $\norm{\hat{\Theta} O - W}_{F}$ (the Forbenius norm) over $d\times d$ orthogonal matrices $O$ (this is done numerically using the solution for orthogonal Procrustes problem), where the columns of $\hat{\Theta}$ contain $\tilde{\theta}_j - \tilde{\Theta}\mu_0$ (i.e., the centered version). Essentially, we center to bring the problem down to the appropriate linear subspace and then the orthogonal matrix estimates a suitable rotation in this subspace to match the estimated vertices with a K-Means clustering, noting that such a clustering would likely find centers that are aligned towards the vertices (this idea is taken from VLAD \cite{yurochkin2019dirichlet}). The final estimate is then computed as $\Theta=\tilde{\Theta}O$.

\subsubsection{MCMC Algorithms for Single Component}
\label{app:MCMC_algorithm_single}
We implement two MCMC algorithms for this problem with a single component. For both of these algorithms, we treat $\alpha\in\bbR^d_+$ (parameter of the underlying Dirichlet distribution) as unknown and include its updates.

\textbf{MCMC with augmented latent variables:} In this algorithm, we introduce the continuous latent variables $\beta_1,\dots,\beta_n$ corresponding to the data $X_1,\dots,X_n$. The posterior distribution is given by

$$p(\bbeta, \psi, \alpha|\bX) \propto \text{pr}(\pi,\sigma^2,\Theta,\alpha)\times \prod_{i=1}^n p(\beta_i|\alpha)p(X_i|\Theta,\beta_i,\sigma^2)$$
where the first term is the prior for the parameters and $p(\beta|\alpha)$ is the Dirichlet density. Conditional on $\bbeta$, the posterior distributions for $\Theta, \sigma^2$ have closed form expressions under the Normal-Inverse Gamma prior for $(\Theta,\sigma^2)$, i.e., $\sigma^2\sim \text{Inv Gamma}(a_0/2, b_0/2)$, $\theta_j\sim \cN(\boldsymbol{0}, \sigma_0^2 I)$ independently for $j\in[d]$. These Gibbs updates are given by $p(\sigma^2|\text{others})\equiv \text{Inv Gamma}(a_n/2, b_n/2)$ and $p(\Theta|\text{others})\equiv \cM\cN (\mu_n, \Sigma_n, I)$ (matrix normal distribution), where
\begin{align*}
    a_n &= a_0 + n, \quad b_n = b_0 + \sum_i \norm{X_i - \Theta\beta_i}^2 \\
    \mu_n &= (\bX^\top \bbeta) \Sigma_n / \sigma^2,\quad  \Sigma_n = \left(\frac{1}{\sigma_0^2} I_d + \frac{1}{\sigma^2} \bbeta^\top\bbeta\right)^{-1}
\end{align*}
where $\bbeta$ is $n\times d$ matrix containing the $\beta_1,\dots,\beta_n$. These updates are precisely that for the multivariate Bayesian regression, since conditional on $\beta$, the joint has the exactly same structure -- recall $p(X_i|\Theta,\beta_i,\sigma^2) = \phi_D(X_i|\Theta\beta_i, \sigma^2 I)$. For $\alpha$ and $\bbeta$, we use Metropolis-Hastings steps. For $\alpha$, we use a random walk proposal. For $\beta_j'$s, we use a Dirichlet proposal with parameter $B\times \beta_j$ (i.e., centered at the current $\beta$), where the hyperparameter $B$ controls the concentration (similar to the noise parameter in a Gaussian random walk proposal).

\textbf{MCMC with latent variables marginalized:} In this case, we use a Grouped Independent Metropolis Hastings MCMC for updating the parameters. This has the same form as the MCMC described in Section \ref{app:multiple_component_algorithms} below, with $K=1$. The only other difference is that we keep $\alpha$ as a learnable parameter and use random-walk Metropolis-Hastings for this.

\subsection{Multiple Components}\label{app:multiple_component_algorithms}
We implement three algorithms - the first one is a EM-type algorithm based on the Gaussian approximation (see ...), the second one is MCMC (with latent variables marginalized) and finally, an EM algorithm for an approximating model. In these cases, we assume that the underlying $P_{\beta}$ is a known distribution (in particular, a symmetric Dirichlet distribution with known $\alpha$). As the baseline, we employ \textit{Geometric Algorithm} - which consists of first using Mixture of Probabilistic PCA to isolate the affine subspaces and then based on the clustering, fitting VLAD on each component.

\subsubsection{MCMC with Grouped Independent Metropolis Hastings}\label{subsec:mcmc_multiple_components}
For parameters $\xi=(\pi,\bTheta,\sigma^2)$, the target distribution is
$$p(\xi|\bX) \propto  \underbrace{p(\xi)}_{\text{prior}}\times \underbrace{\left(\prod_{i=1}^n \sum_k \pi_k \int \phi(X_i\mid \Theta_k\beta,\sigma_k^2 I) P_{\beta}(d\beta)\right)}_{\text{likelihood } \cL}.$$
The Monte Carlo within Metropolis algorithm proceeds as a regular Metropolis-Hastings MCMC algorithm, but approximates the intractable likelihood via a Monte Carlo approximation
$$\cL(\xi|\bX) \approx \prod_i \sum_k \sum_{j\in[M]} \frac{\pi_k}{M} \phi(X_i\mid \Theta_k \beta_{i,j},\sigma_k^2 I)=:\tilde{\cL}(\xi|\bX,\bbeta)$$
where $\beta_{i,j}$ are iid samples from $P_{\beta}$ for $i\in[n], j\in[M]$ (denoting this entire collection of latent variables as $\bbeta$). This approximation is used while computing the Metropolis-Hastings ratio. In particular, given current state $\xi(t)$, a proposal is made $\tilde{\xi}\sim q(\xi'|\xi)$ and the ratio is computed as
$$r = \frac{p(\tilde{\xi})\tilde{\cL}(\tilde{\xi}|\bX,\tilde{\bbeta})}{p(\xi_t)\tilde{\cL}(\xi_t|\bX,\bbeta)}\times \frac{q(\xi_t|\tilde{\xi})}{q(\tilde{\xi}|\xi_t)}$$
using independent samples $\bbeta, \tilde{\bbeta}$ to compute the approximate likelihood terms. Thus, at each iteration, two independent copies of $nM$ samples from $P_{\beta}$ are required for this.

In the Grouped Independent MH, the latent variables $\beta$ are reused and updated together with the parameters. This can be seen as a Markov chain on $(\xi,\bbeta)$ jointly (recall the usage of $\bbeta$ to denote the entire collection $\{\beta_{i,j}:i\in[n],j\in[M]\}$). Suppose, the chain is at $(\xi_t, \bbeta_t)$ at some point of time. A new proposal is made as $\tilde{\xi}\sim q(\xi'|\xi)$ and $\tilde{\bbeta}\sim P_{\beta}$ ($nM$ iid copies). Then the MH ratio is computed similarly as before
$$r = \frac{p(\tilde{\xi})\tilde{\cL}(\tilde{\xi}|\bX,\tilde{\bbeta})}{p(\xi_t)\tilde{\cL}(\xi_t|\bX,\bbeta_t)}\times \frac{q(\xi_t|\tilde{\xi})}{q(\tilde{\xi}|\xi_t)}.$$
If this proposal is accepted, then set $(\xi_{t+1},\bbeta_{t+1})=(\tilde{\xi}, \tilde{\bbeta})$, else we retain the same state. Note the difference from the previous case -- here, we only draw one collection of $nM$ samples per iteration. If the proposal is rejected, we keep the latent variables $\bbeta_t$. For implementation purpose, note that it is not required to store the samples in $\bbeta_t$, but only the likelihood term $\tilde{\cL}(\xi_t|\bX,\bbeta_t)$.

The proposal distribution $q(\xi'|\xi)$: For $\Theta_1,\dots,\Theta_K$ and $\sigma_1^2,\dots,\sigma_K^2$ we use a Gaussian random walk proposal (the proposal variances are hyperparameters of the MCMC). Since $\sigma_k^2$ is required to be positive, we can take the absolute value of the proposal. For the mixture probabilities, since $\pi$ is constrained to be on $\Delta^{K-1}$, the usual Gaussian random walk does not work. Instead, we can use a random walk using a singular Gaussian distribution on the affine hull of $\Delta^{K-1}$ as the proposal. In such a case, the ratio of $q(\xi_t|\tilde{\xi})/q(\tilde{\xi}|\xi_t)$ is 1 and can be ignored.

Choice of prior: While the above algorithm is not specific to any choice of prior, for concreteness, we use the following prior structure
$$p(\xi) = \text{Dir}_{\pi_0}(\pi) \times \prod_{k,j} \cN(\theta_{kj}\mid \boldsymbol{0}, \sigma_0^2I)\times \prod_k \text{Exp}(\sigma_k^2 \mid \lambda_0)$$
where $\pi_0\in\bbR_+^K, \sigma_0^2>0$ and $\lambda_0>0$ are prior hyperparameters for the MCMC.

\subsubsection{EM with Gaussian Approximation}\label{subsec:gaussian_approx_multiple_components}
We use the Gaussian approximation of $p_{\beta}$ (discussed in \ref{app:approx_gaussian_algorithm}) under a EM framework, to develop an efficient algorithm, similar to the Generalized EM algorithm for Mixture of Probabilistic PCA model \cite{tipping1999mixtures}. Using the Gaussian approximation, the model becomes $\sum_k \pi_k \cN(x|\Theta_k\mu_0, \Theta_k\Sigma_0\Theta_k^\top + \sigma_k^2)$. Given current parameters $\bTheta,\pi,\Sigma$, in the E-step, we compute the weights
$$w_{ik} \propto \pi_k\phi(X_i|\Theta_k\mu_0, \Theta_k\Sigma_0\Theta_k^\top+\sigma_k^2 I_D)$$
normalized so that $\sum_k w_{ik}=1$ for each $i$. In the M-step, we use the same estimation method as discussed in \ref{app:approx_gaussian_algorithm}, replacing the moments by weighted moments, i.e., for component $k$, use $\bar{X}_{w_k} =\sum_i w_{ik}X_i$ in place of $\bar{X}$ and $\sum_i w_{ik}(X_i-\bar{X}_{w_k})(X_i-\bar{X}_{w_k})^\top$ in place of $\Cov(X)$. Since the E-step stays the same irrespective of any rotations in $\Theta$ (since the mean and covariance are same), we only perform the optimal matching rotation at the last iteration (upon convergence).

\begin{figure}
    \centering
    \includegraphics[width=0.75\linewidth]{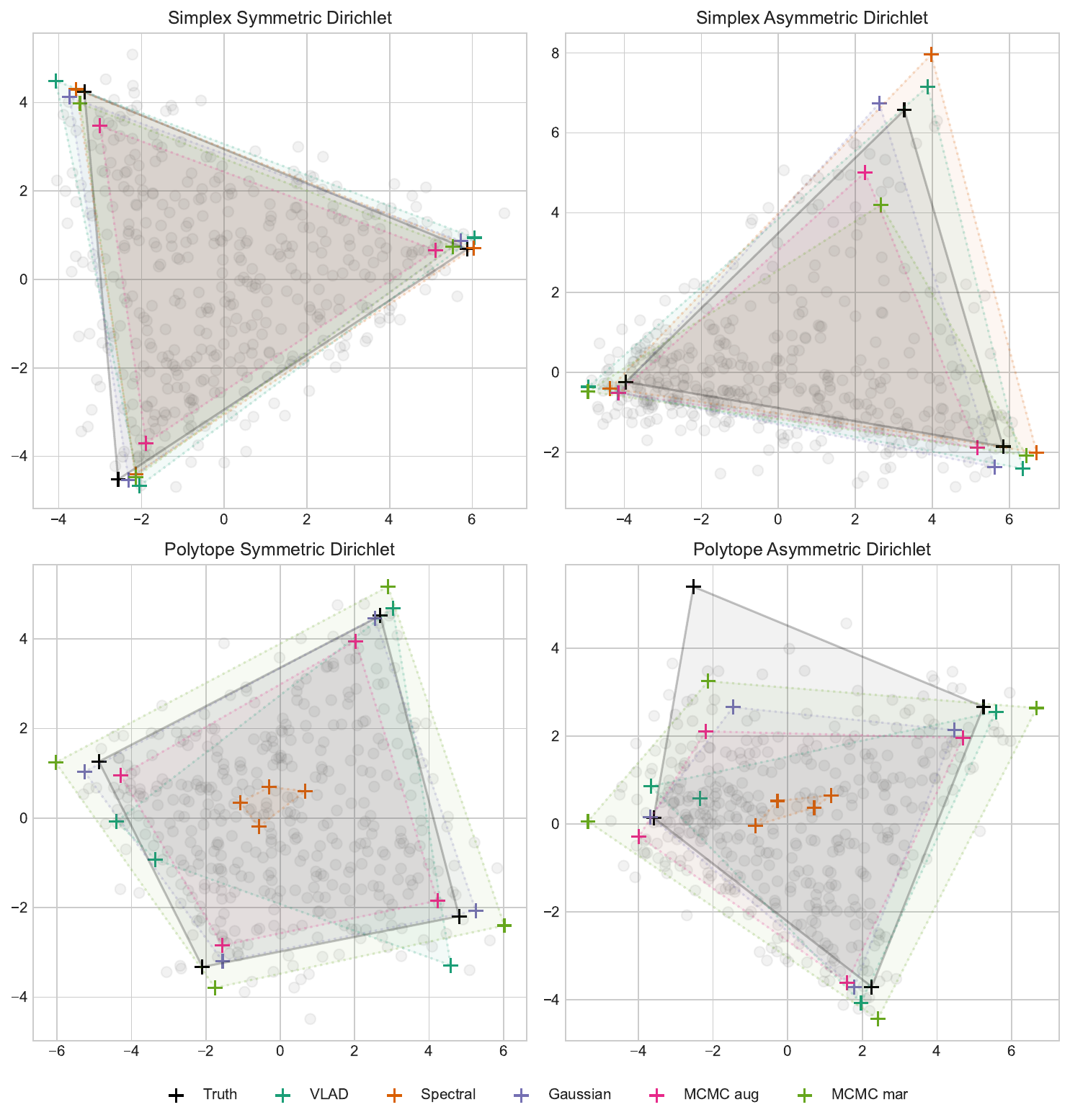}
    \caption{Settings for single component simulations in Section \ref{sec:simulations}}
    \label{fig:single_component_settings}
\end{figure}

\begin{figure}
    \centering
    \includegraphics[width=0.8\linewidth]{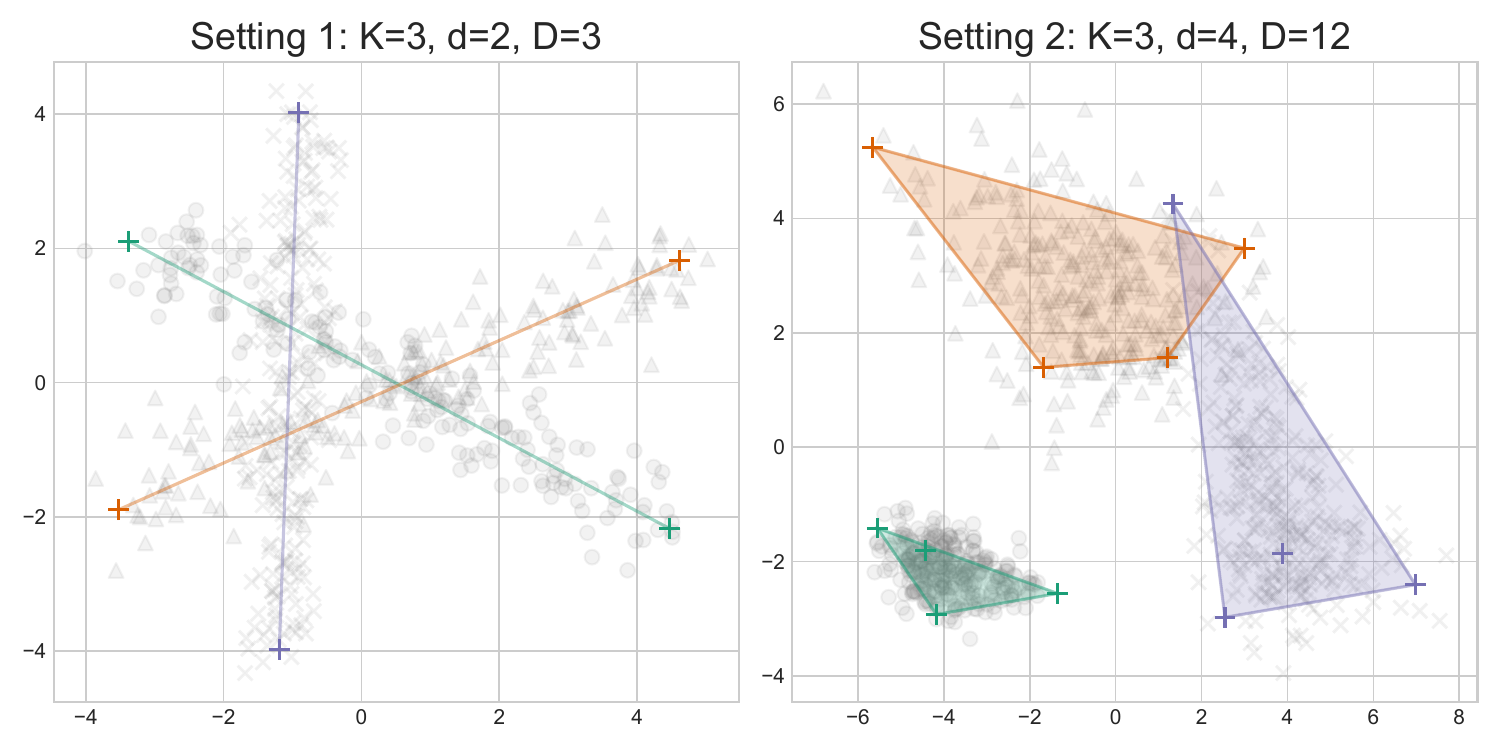}
    \caption{Settings for Simulations in Section \ref{sec:simulations} for multiple components - visualization via PCA using first 2 principle components}
    \label{fig:multiple_components_setting}
\end{figure}

\subsubsection{Geometric Algorithm}\label{subsec:geometric_multiple_components}
This algorithm first fits a Mixture of Probabilistic PCA model on the data to get a clustering and then uses VLAD on each cluster separately to estimate the component polytopes.

\subsection{Simulation Details}
 \begin{table}
    \centering
\begin{tabular}{cccc}
 \hline
 Algorithm & $n=200$  & $n\approx 1000$ & $n=4000$\\
 \hline
 Gaussian & 0.81 & 1.47 & 1.64 \\
 MCMC & 7.60 & 7.83 & 9.46 \\
 EM(50) & 0.55 & 0.62 & 0.69 \\
 EM(100) & 0.56 & 0.65 & 0.66 \\
 EM (400) & 0.67 & 0.69 & 0.76 \\
 \hline
\end{tabular}
\caption{Average run time (in seconds) for the algorithms in Setting 2 (multiple components)}
\label{table:algorithms_time}
\end{table}

\begin{figure}
    \centering
    \includegraphics[width=0.38\linewidth]{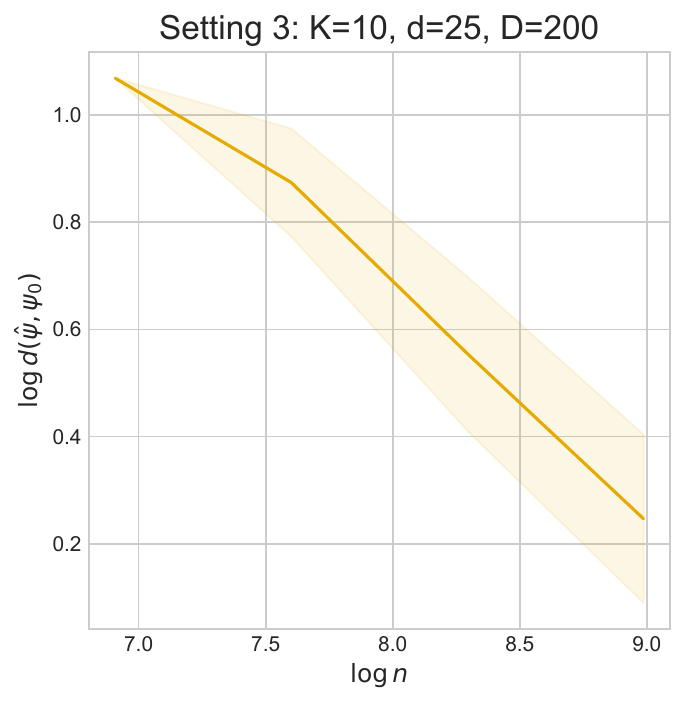}
    \includegraphics[width=0.4\linewidth]{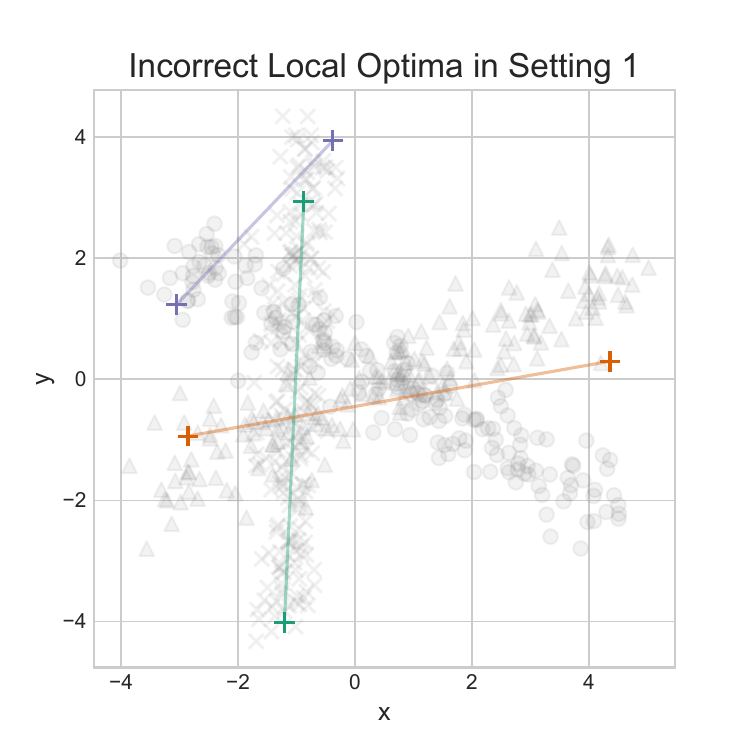}
    \caption{(Left) Simulation results for Setting 3, (Right) Illustration of a Local Model in Setting 1}
    \label{fig:setting3_local_mode}
\end{figure} 
The settings for the simulation studies for single component and multiple components in Section \ref{sec:numerical study} are shown in Figures \ref{fig:single_component_settings} and \ref{fig:multiple_components_setting} respectively. Results for setting 3 in the multiple component case (only using the Gaussian approximation algorithm) are shown in Figure \ref{fig:setting3_local_mode} (left). For the MCMC algorithms, the tuning parameters include the prior hyperparameters and the proposal hyperparameters. The algorithm is relatively stable with respect to the former, but the acceptance rate for the Metropolis step depends crucially on the proposal hyperparameters. Particularly in high dimensions, it is not straight-forward to tune. This issue can be mitigated with more sophisticated proposals (rather than simple random walks), i.e., using Langevin or Hamiltonian Monte Carlos. In our case, we ran the MCMC for 20000 iterations for each experiment and used the last 5000 (with a thinning of 100) to get 50 samples, based on which the metrics were computed.

We also provide a time comparison for the algorithms used for the multiple component simulations (for Setting 2). The results are shown in Table \ref{table:algorithms_time}, where all algorithms were implemented using Jax in Python and GPU was used for training. We do not include the Geometric Algorithm, for which we used an existing implementation (not in Jax). The EM algorithm is the fastest -- however, we note that for very large dataset and large $M$, our implementation is not efficient in terms of memory consumption and the current version runs into OOM-issues, which can be handled by a more careful and efficient implementation.

\subsubsection{Mis-specified Latent Distribution}
\begin{figure}
    \centering
    \includegraphics[width=0.9\linewidth]{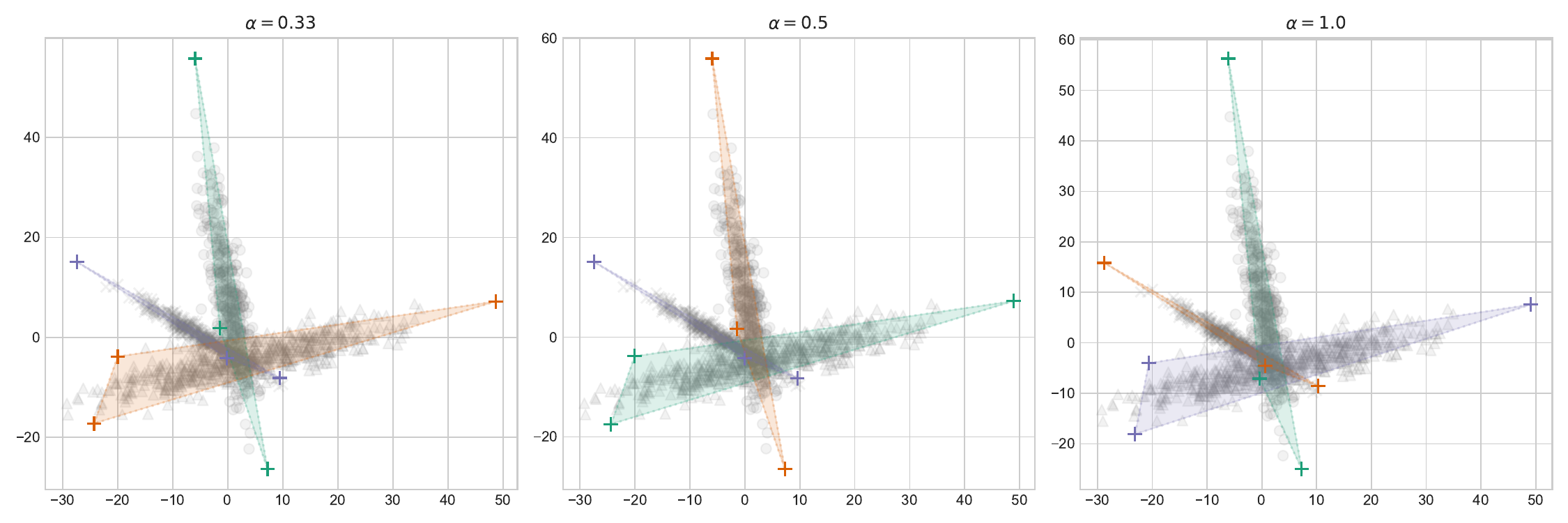}
    \caption{Result of Using Approximate EM Algorithm (assuming Dirichlet latent mixing) when the true latent mixing distribution is Gaussian (corresponding to MPPCA) - each column corresponds to different choice of $\alpha$}
    \label{fig:misspecified1}
\end{figure}

\begin{figure}
    \centering
    \includegraphics[width=0.9\linewidth]{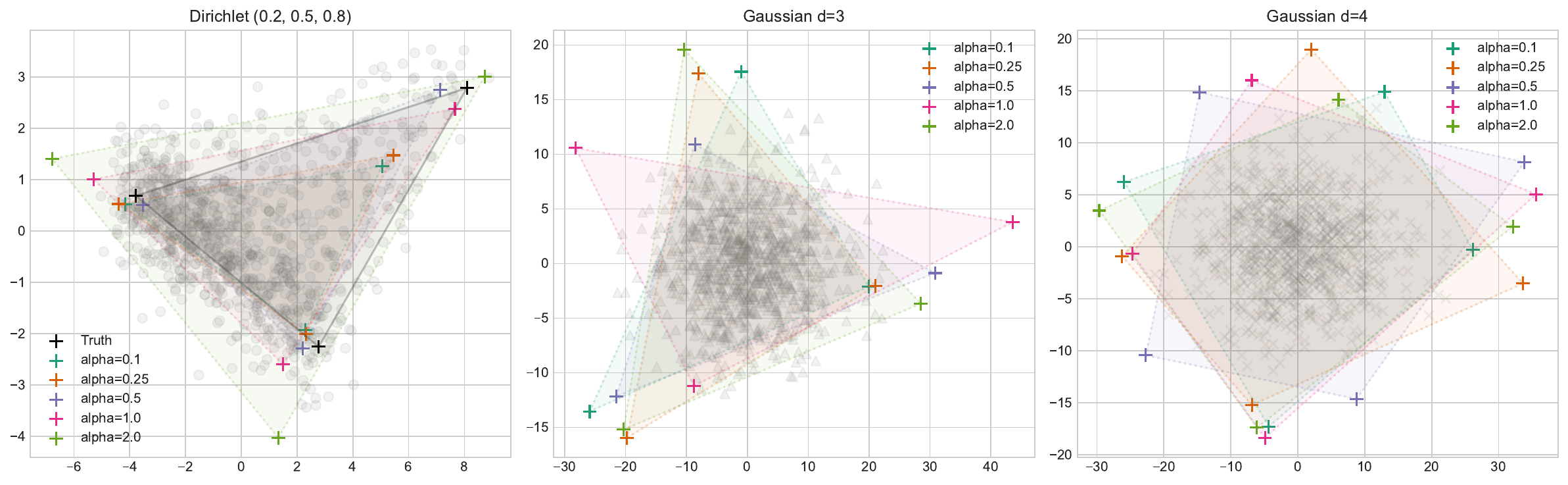}
    \caption{Results of using Approximate EM algorithm when the latent measure is mis-specified}
    \label{fig:misspecified2}
\end{figure}

The previous simulations all used the same $p_{\beta}$ as used in the data generating mechanism. We also perform simple simulations to see the effect of a mis-specified $p_{\beta}$. Figure \ref{fig:misspecified1} shows the result of using the EM algorithm (assuming Dirichlet $p_{\beta}$), when the ground-truth $p_{\beta}$ is Gaussian (i.e., the true data generating model is a Mixture of Probabilistic PCA). We see that the method correctly captures the three clusters and is quite robust to the choice of $\alpha$ (the Dirichlet parameter). Not only does it capture the component affine spaces, it also provides an interpretable clustering of the data.

Figure \ref{fig:misspecified2} takes a closer look at a single component, where the title of the figure shows the true $p_{\beta}$ distribution and the legends depict the choice of $\alpha$ chosen for the EM algorithm. The left-most plot consider a mis-specified Dirichlet distribution (truth is asymmetric) and the right two plots consider the case when $p_{\beta}$ is Gaussian. The rank of the ground-truth covariance is 2 and it is fit with $d=3$ (middle) and $d=4$ (right) vertices respectively. The results are mostly robust to the choice of $\alpha$ and we notice that this choice acts as a regularizer, with lower $\alpha$ giving a more tighter bounding convex polytope for the data. For the Gaussian case, it is worth mentioning that due to the rotational invariance of the low-rank covariance, there is no notion of ground-truth vertices -- however, in each case, the affine hull of the estimated vertices provides a good estimate for the underlying column space of the covariance.

\end{appendix}
\end{document}